\documentclass[a4paper,twoside]{article}
\usepackage{a4}
\usepackage{amssymb}
\usepackage{amsmath}
\usepackage{upref}
\usepackage{url}
\usepackage{bbm}
\usepackage[active]{srcltx}
\usepackage[dvipsnames]{xcolor} 
\usepackage[pagebackref,colorlinks,citecolor=blue,linkcolor=blue,urlcolor=blue]{hyperref}
\newcount\minutes \newcount\hours
\hours=\time
\divide\hours 60
\minutes=\hours
\multiply\minutes -60
\advance\minutes \time
\newcommand{\klockan}{\the\hours:{\ifnum\minutes<10 0\fi}\the\minutes}
\newcommand{\tid}{\today\ \klockan}
\newcommand{\prtid}{\smash{\raise 10mm \hbox{\LaTeX ed \tid}}}
\makeatletter
\def\sectionmark#1{} 
\def\subsectionmark#1{}
\newcommand{\sectnr}{\ifnum \c@secnumdepth >\z@
                 \thesection.\hskip 1em\relax \fi}
\def\@evenhead{\footnotesize\rm\thepage\hfil\leftmark\hfil\llap{\prtid}}
\def\@oddhead{\footnotesize\rm\rlap{\prtid}\hfil\rightmark\hfil\thepage}
\def\tableofcontents{\section*{Contents} 
 \@starttoc{toc}}
\makeatother
\makeatletter
\def\@biblabel#1{#1.}
\makeatother
\makeatletter
\let\Thebibliography=\thebibliography
\renewcommand{\thebibliography}[1]{\def\@mkboth##1##2{}\Thebibliography{#1}
\addcontentsline{toc}{section}{References}
\frenchspacing 
\setlength{\@topsep}{0pt}
\setlength{\itemsep}{0pt}%
\setlength{\parskip}{0pt plus 2pt}%
}
\makeatother
\makeatletter
\def\mdots@{\mathinner.\nonscript\!.%
 \ifx\next,.\else\ifx\next;.\else\ifx\next..\else
 \nonscript\!\mathinner.\fi\fi\fi}
\let\ldots\mdots@
\let\cdots\mdots@
\let\dotso\mdots@
\let\dotsb\mdots@
\let\dotsm\mdots@
\let\dotsc\mdots@
\def\vdots{\vbox{\baselineskip2.8\p@ \lineskiplimit\z@
    \kern6\p@\hbox{.}\hbox{.}\hbox{.}\kern3\p@}}
\def\ddots{\mathinner{\mkern1mu\raise8.6\p@\vbox{\kern7\p@\hbox{.}}%
    \raise5.8\p@\hbox{.}\raise3\p@\hbox{.}\mkern1mu}}
\makeatother
\makeatletter
\let\Enumerate=\enumerate
\renewcommand{\enumerate}{\Enumerate%
\setlength{\@topsep}{0pt}
\setlength{\itemsep}{0pt}%
\setlength{\parskip}{0pt plus 1pt}%
\renewcommand{\theenumi}{\textup{(\alph{enumi})}}%
\renewcommand{\labelenumi}{\theenumi}%
}
\let\endEnumerate=\endenumerate
\renewcommand{\endenumerate}{\endEnumerate\unskip}
\makeatother
\makeatletter
\def\@seccntformat#1{\csname the#1\endcsname.\quad}
\makeatother
\newcommand{\authortitle}[2]{\author{#1}\title{#2}\markboth{#1}{#2}}
\newcommand{\auth}[2]{{#1, #2.}}
\newcommand{\idxauth}[2]{{#1, #2.}}
\newcommand{\art}[6]{{\sc #1, \rm #2, \it #3 \bf #4 \rm (#5), \mbox{#6}.}}
\newcommand{\artin}[3]{{\sc #1, \rm #2,  in #3.}}
\newcommand{\artprep}[3]{{\sc #1, \rm #2, #3.}}

\newcommand{\book}[3]{{\sc #1, \it #2, \rm #3.}}
\newcommand{\AND}{{\rm and }}

\RequirePackage{amsthm}
\newtheoremstyle{descriptive}%
  {\topsep}   
  {\topsep}   
  {\rmfamily} 
  {}          
  {\bfseries} 
  {.}         
  { }         
  {}          
\newtheoremstyle{propositional}%
  {\topsep}   
  {\topsep}   
  {\itshape}  
  {}          
  {\bfseries} 
  {.}         
  { }         
  {}          
\newtheoremstyle{remarkstyle}%
  {\topsep}   
  {\topsep}   
  {\rmfamily}  
  {}          
  {\itshape} 
  {.}         
  { }         
  {}          
\theoremstyle{propositional}
\newtheorem{thm}{Theorem}[section]
\newtheorem{prop}[thm]{Proposition}
\newtheorem{lem}[thm]{Lemma}
\newtheorem{cor}[thm]{Corollary}
\theoremstyle{descriptive}
\newtheorem{deff}[thm]{Definition}
\newtheorem{example}[thm]{Example}
\newtheorem{remark}[thm]{Remark}
\makeatletter
\renewenvironment{proof}[1][\proofname]{\par
  \pushQED{\qed}%
  \normalfont
  \trivlist
  \item[\hskip\labelsep
        \itshape
    #1\@addpunct{.}]\ignorespaces
}{%
  \popQED\endtrivlist\@endpefalse
}
\makeatother
\newcommand{\setm}{\setminus}
\renewcommand{\emptyset}{\varnothing}
\def\vint{\mathop{\mathchoice%
          {\setbox0\hbox{$\displaystyle\intop$}\kern 0.22\wd0%
           \vcenter{\hrule width 0.6\wd0}\kern -0.82\wd0}%
          {\setbox0\hbox{$\textstyle\intop$}\kern 0.2\wd0%
           \vcenter{\hrule width 0.6\wd0}\kern -0.8\wd0}%
          {\setbox0\hbox{$\scriptstyle\intop$}\kern 0.2\wd0%
           \vcenter{\hrule width 0.6\wd0}\kern -0.8\wd0}%
          {\setbox0\hbox{$\scriptscriptstyle\intop$}\kern 0.2\wd0%
           \vcenter{\hrule width 0.6\wd0}\kern -0.8\wd0}}%
          \mathopen{}\int}
{\catcode`p =12 \catcode`t =12 \gdef\eeaa#1pt{#1}}      
\def\accentadjtext#1{\setbox0\hbox{$#1$}\kern   
                \expandafter\eeaa\the\fontdimen1\textfont1 \ht0 }
\def\accentadjscript#1{\setbox0\hbox{$#1$}\kern 
                \expandafter\eeaa\the\fontdimen1\scriptfont1 \ht0 }
\def\accentadjscriptscript#1{\setbox0\hbox{$#1$}\kern   
                \expandafter\eeaa\the\fontdimen1\scriptscriptfont1 \ht0 }
\def\accentadjtextback#1{\setbox0\hbox{$#1$}\kern       
                -\expandafter\eeaa\the\fontdimen1\textfont1 \ht0 }
\def\accentadjscriptback#1{\setbox0\hbox{$#1$}\kern     
                -\expandafter\eeaa\the\fontdimen1\scriptfont1 \ht0 }
\def\accentadjscriptscriptback#1{\setbox0\hbox{$#1$}\kern 
                -\expandafter\eeaa\the\fontdimen1\scriptscriptfont1 \ht0 }
\def\itoverline#1{{\mathsurround0pt\mathchoice
        {\rlap{$\accentadjtext{\displaystyle #1}
                \accentadjtext{\vrule height1.593pt}
                \overline{\phantom{\displaystyle #1}
                \accentadjtextback{\displaystyle #1}}$}{#1}}
        {\rlap{$\accentadjtext{\textstyle #1}
                \accentadjtext{\vrule height1.593pt}
                \overline{\phantom{\textstyle #1}
                \accentadjtextback{\textstyle #1}}$}{#1}}
        {\rlap{$\accentadjscript{\scriptstyle #1}
                \accentadjscript{\vrule height1.593pt}
                \overline{\phantom{\scriptstyle #1}
                \accentadjscriptback{\scriptstyle #1}}$}{#1}}
        {\rlap{$\accentadjscriptscript{\scriptscriptstyle #1}
                \accentadjscriptscript{\vrule height1.593pt}
                \overline{\phantom{\scriptscriptstyle #1}
                \accentadjscriptscriptback{\scriptscriptstyle #1}}$}{#1}}}}
\def\itunderline#1{{\mathsurround0pt\mathchoice
        {\rlap{$\underline{\phantom{\displaystyle #1}
                \accentadjtextback{\displaystyle #1}}$}{#1}}
        {\rlap{$\underline{\phantom{\textstyle #1}
                \accentadjtextback{\textstyle #1}}$}{#1}}
        {\rlap{$\underline{\phantom{\scriptstyle #1}
                \accentadjscriptback{\scriptstyle #1}}$}{#1}}
        {\rlap{$\underline{\phantom{\scriptscriptstyle #1}
                \accentadjscriptscriptback{\scriptscriptstyle #1}}$}{#1}}}}
\newcommand{\Cp}{{C_p}}
\DeclareMathOperator{\diam}{diam}
\DeclareMathOperator{\supp}{supp}
\DeclareMathOperator{\dist}{dist}
\DeclareMathOperator{\Div}{div}
\DeclareMathOperator{\capp}{cap}
\DeclareMathOperator{\Capp}{Cap}
\newcommand{\cp}{\capp_p}
\newcommand{\cpt}{\widetilde{\capp}_p}
\newcommand{\cpbar}{\overline{\capp}_p}
\newcommand{\cpLip}{\capp_p^{\Lip}}
\newcommand{\cpmu}{\capp_{p,\mu}}
\newcommand{\grad}{\nabla}
\DeclareMathOperator{\Lip}{Lip}
\newcommand{\Lipc}{{\Lip_c}}
\DeclareMathOperator*{\essliminf}{ess\,lim\,inf}
\DeclareMathOperator*{\essinf}{ess\,inf}
\newcommand{\bdry}{\partial}
\newcommand{\bdy}{\bdry}
\newcommand{\loc}{_{\rm loc}}
\DeclareMathOperator{\Mod}{Mod}
\newcommand{\Modp}{\Mod_p}
\newcommand{\alp}{\alpha}
\newcommand{\be}{\beta}
\newcommand{\ga}{\gamma}
\newcommand{\de}{\delta}
\newcommand{\eps}{\varepsilon}
\newcommand{\la}{\lambda}
\newcommand{\Om}{\Omega}
\newcommand{\K}{{\cal K}}
\newcommand{\uhat}{\hat{u}}
\newcommand{\etah}{\hat{\eta}}
\newcommand{\Ghat}{\widehat{G}}
\renewcommand{\phi}{\varphi}
\newcommand{\p}{{$p\mspace{1mu}$}}
\newcommand{\R}{\mathbf{R}}
\newcommand{\eR}{{\overline{\R}}}
\def\cprime{{\mathsurround0pt$'$}}
\newcommand{\limplus}{{\mathchoice{\vcenter{\hbox{$\scriptstyle +$}}}
  {\vcenter{\hbox{$\scriptstyle +$}}}
  {\vcenter{\hbox{$\scriptscriptstyle +$}}}
  {\vcenter{\hbox{$\scriptscriptstyle +$}}}
}}
\newcommand{\limminus}{{\mathchoice{\vcenter{\hbox{$\scriptstyle -$}}}
  {\vcenter{\hbox{$\scriptstyle -$}}}
  {\vcenter{\hbox{$\scriptscriptstyle -$}}}
  {\vcenter{\hbox{$\scriptscriptstyle -$}}}
}}
\newcommand{\Np}{N^{1,p}}
\newcommand{\Dp}{D^{p}}
\newcommand{\Nploc}{N^{1,p}\loc}
\newcommand{\ut}{\tilde{u}}
\newcommand{\al}{\alpha}
\newcommand{\Ga}{\Gamma}
\newcommand{\Lploc}{L^p\loc}
\newcommand{\clB}{\itoverline{B}}
\newcommand{\LL}{\mathcal{L}}
\newcommand{\UU}{\mathcal{U}}
\newcommand{\uP}{\itoverline{P}}     
\newcommand{\lP}{\itunderline{P}} 
\newcommand{\cpDp}{\capp_{\Dp}}
\newcommand{\bdystar}{\partial^*}
\newcommand{\Xstar}{X^*}
\newcommand{\vt}{\tilde{v}}
\newcommand{\pot}[2]{Q^{#1}_{#2}}
\newcommand{\PsiQ}[2]{\Psi^{#1}_{#2}}
\newcommand{\Qhat}{\widehat{Q}}
\newcommand{\chione}{\mathbbm{1}}
\newcommand{\clE}{\itoverline{E}}

\newcommand{\clU}{\overline{U}}
\newcommand{\clOm}{{\overline{\Om}}}
\numberwithin{equation}{section}
\newcommand{\eqv}{\ensuremath{
\mathchoice{\quad \Longleftrightarrow \quad}{\Leftrightarrow}
                {\Leftrightarrow}{\Leftrightarrow}}}
\newcommand{\imp}{\ensuremath{\Rightarrow}}
\newenvironment{ack}{\medskip{\it Acknowledgement.}}{}

\newcounter{saveenumi}

\begin{document}

\authortitle{Anders Bj\"orn and Jana Bj\"orn}
{Condenser capacities,  capacitary potentials 
and global \p-harmonic Green functions}
\title{Condenser capacities and capacitary potentials \\
for unbounded sets, and global \p-harmonic \\ Green functions on metric spaces}
\author{
Anders Bj\"orn and Jana Bj\"orn}

\author{
Anders Bj\"orn \\
\it\small Department of Mathematics, Link\"oping University, SE-581 83 Link\"oping, Sweden\\
\it \small anders.bjorn@liu.se, ORCID\/\textup{:} 0000-0002-9677-8321
\\
\\
Jana Bj\"orn \\
\it\small Department of Mathematics, Link\"oping University, SE-581 83 Link\"oping, Sweden\\
\it \small jana.bjorn@liu.se, ORCID\/\textup{:} 0000-0002-1238-6751
}

\date{}

\maketitle

\noindent{\small
{\bf Abstract}. 
We study the condenser capacity $\cp(E,\Om)$ on \emph{unbounded} open sets $\Om$  in 
a proper connected metric space $X$
equipped with a locally doubling measure 
supporting a local \p-Poincar\'e inequality, where $1<p<\infty$.
Using a new definition of capacitary potentials, 
we show that $\cp$ is countably subadditive and that it is a
Choquet capacity.
We next obtain formulas  for the capacity
of superlevel sets for the capacitary potential.
These are then used to show that
\p-harmonic Green functions exist in an unbounded domain $\Om$ if and only if
either $X$ is \p-hyperbolic or the Sobolev capacity $\Cp(X\setm \Om)>0$.
As an application, we  deduce new results for Perron solutions and boundary regularity
for the Dirichlet boundary value problem for \p-harmonic functions
in unbounded open sets.
}

\medskip

\noindent {\small \emph{Key words and phrases}:
boundary regularity,
condenser  capacity,
capacitary potential,
Dirichlet problem,
locally doubling measure,
global \p-harmonic Green function,
metric space,
Perron solution,
\p-harmonic function,
\p-hyperbolic space,
local Poincar\'e inequality,
singular function.
}

\medskip


\section{Introduction}

It is well known that every domain  $\Om\subset\R^n$, $n\ge3$, has
a \emph{Green function},
i.e\  a solution of 
the Poisson equation $-\Delta u= \delta_{x_0}$ in $\Om$
with zero boundary data on $\bdy\Om$ and
at $\infty$ when $\Om$ is unbounded. 
Here $\de_{x_0}$ is the Dirac distribution at $x_0\in\Om$.
For $n=2$, Green functions exist only in domains $\Om$   
that are \emph{Greenian}, 
a property equivalent to the complement $\R^2\setm\Om$ being a nonpolar set,
see e.g.\ Armitage--Gardiner~\cite[Theorem~5.3.8]{AG} or
Doob~\cite[Theorems~1.V.6 and~1.VII.7]{Doob}.
In particular, the plane itself is not Greenian.

This distinction between the plane and higher-dimensional spaces is 
reflected in many other phenomena, such as the (non)existence of nonnegative
nonconstant superharmonic functions or
the recurrent/transient behaviour of the Brownian motion, and can be 
captured by the notions of \emph{parabolic} and \emph{hyperbolic} spaces.
An analogue in the setting of nonlinear PDEs, such as for \p-harmonic functions,
is that for all $n > p>1$, the \p-Green function 
\[
u(x)=C_{n,p} |x-x_0|^{(p-n)/(p-1)},
\]
with a suitable constant $C_{n,p}>0$, satisfies 
\[
\Delta_p u := \Div( |\grad u|^{p-2} \grad u) = -\delta_{x_0}  \quad \text{in }  \R^n
\qquad \text{and} \qquad \lim_{x\to\infty} u(x)=0, 
\]
while there is no such \p-Green function for $n\le p$.
Green functions are useful tools for  PDEs 
and together with potentials play an important role 
in classical potential theory.

In this paper, we study the above notions for unbounded sets
in rather general metric measure spaces.
More precisely, let for the rest of the introduction (unless
said otherwise)
$X$ be 
a proper connected metric space
equipped with a locally doubling measure $\mu$ 
supporting a local \p-Poincar\'e inequality, where $1<p<\infty$.
$\Om$ will always denote a nonempty open set.
Weighted $\R^n$, 
equipped with a \p-admissible weight
as in Hei\-no\-nen--Kil\-pe\-l\"ai\-nen--Martio~\cite{HeKiMa},
and (sub)Riemannian manifolds, as in 
Coulhon--Holo\-pai\-nen--Saloff-Coste~\cite{CoHoSC} 
and Holo\-pai\-nen~\cite{Ho},~\cite{HoDuke}, are included as special cases.
So are 
Heisenberg and Carnot groups, 
vector fields satisfying the H\"ormander condition,
and many other situations,
see Haj\l asz--Koskela~\cite[Sections~10--13]{HaKo}. 

In the generality of metric spaces, \emph{\p-harmonic} functions are 
defined as continuous minimizers of the \p-energy integral
\begin{equation}  \label{eq-energy-int-intro}
\int g_u^p\,d\mu,
\end{equation}
where $g_u$ is the minimal \p-weak upper gradient of $u$.
The associated superminimizers then give rise to \emph{\p-superharmonic} functions, 
see Section~\ref{sect-pharm} for the precise definitions.
Our results and proofs 
also apply if the \p-harmonicity is replaced by
Cheeger \p-harmonicity, cf.\ 
Cheeger~\cite[Remark~7.19 and pp.~434, 462]{Cheeg},
Bj\"orn--Bj\"orn~\cite[Appendix~B]{BBbook} and
Bj\"orn--Bj\"orn--Lehrb\"ack~\cite[Section~13]{BBLehGreen}.
In that case, the \p-weak upper gradient $g_u$ is replaced
by the norm $\|Du\|$ of the Cheeger gradient $Du$ 
provided by~\cite[Theorem~4.38 and Definition~4.42]{Cheeg}
(or \cite[Theorem~13.4.4 and (13.4.6)]{HKSTbook}).
By the discussion in~\cite[p.\ 460]{Cheeg}, $\|\cdot\|$ can be chosen to be 
an inner product norm, 
in which case the minimizers of the \p-energy integral
\eqref{eq-energy-int-intro} actually satisfy a \p-Laplace type equation in a weak sense.
Our approach does not require Cheeger's differentiable structure and is based solely on 
minimizing the \p-energy integral \eqref{eq-energy-int-intro}.

To obtain our results on \p-singular and \p-Green functions, we first
study the condenser capacity $\cp(E,\Om)$ and its properties 
for unbounded open sets $\Om \subset X$.
This we do
in Sections~\ref{sect-cap} and \ref{sect-potentials}--\ref{sect-level}.
In particular, we obtain capacity formulas for superlevel sets 
of capacitary potentials, 
which later lead to the normalization  in~\eqref{eq-normalized-Green-intro}.
The condenser capacity $\cp(E,\Om)$
 has been studied and used in the metric setting for more than 20 years
and is well understood for bounded $\Om$, see e.g.\  \cite[Section~6.3]{BBbook}.

For unbounded~$\Om$, we propose a modification 
(through the limit \eqref{eq-cp-lim-deff})
of the earlier definition~\eqref{eq-deff-cp} of the condenser capacity
in metric spaces from Holo\-pai\-nen--Shan\-mu\-ga\-lin\-gam~\cite{HoSh},
as otherwise it would not be countably subadditive in \p-parabolic
spaces with infinite measure, see Proposition~\ref{prop-cpt}.
Our proof of countable subadditivity 
(Theorem~\ref{thm-cp}\ref{cp-subadd})
is therefore considerably more
involved for unbounded $\Om$ than for bounded $\Om$,
and also requires stronger assumptions on the measure $\mu$.
In particular, an
important ingredient is Lemma~\ref{lem-reflex-conv}, which ultimately
relies on the reflexivity of Newtonian spaces,
which is a highly nontrivial property proved by 
Cheeger~\cite[Theorem~4.48]{Cheeg}
and Ambrosio--Colombo--Di Marino~\cite[Corollary~41]{AmbCD}.
Also the theory of \p-harmonic functions is used extensively.

The results in weighted $\R^n$ in
  Hei\-no\-nen--Kil\-pe\-l\"ai\-nen--Martio~\cite[Theorem~2.2]{HeKiMa}
  do cover unbounded $\Om$, but as their condenser capacity 
is defined differently 
their proofs do not easily
carry over to metric spaces.
The condenser capacity $\cp(E,\Om)$ 
has been designed to reflect the normalization~\eqref{eq-normalized-Green-intro}
of Green functions in unbounded domains.
We point out that it 
differs for unbounded $\Om$ from other condenser capacities used 
e.g.\ in Hei\-no\-nen--Koskela~\cite{HeKo98}, 
Kallunki--Shan\-mu\-ga\-lin\-gam~\cite{KaSh},
Bj\"orn--Bj\"orn--Shan\-mu\-ga\-lin\-gam~\cite{BBShypend}
and Eriksson-Bique--Poggi-Corradini~\cite{EB-PC},
see Section~\ref{sect-cap} and Example~\ref{ex-cap-not-same}.
Further discussion of the condenser capacity 
$\cp(E,\Om)$
is at the end of the introduction.

We now present some 
of our main results for \p-Green and \p-singular functions
associated with \p-harmonic functions in metric spaces.
The following theorem is proved at the end of Section~\ref{sect-Green}.
It does not seem to appear in the literature even for unweighted $\R^n$,
even though the
equivalence \ref{m-superh}\eqv\ref{m-hyp-cp}
essentially follows from Theorems~9.22 and~9.23 in 
Hei\-no\-nen--Kil\-pe\-l\"ai\-nen--Martio~\cite{HeKiMa}
when $X$ is $\R^n$ equipped with a \p-admissible weight as in~\cite{HeKiMa}.

\begin{thm} \label{thm-main}
Assume that $1<p<\infty$ and that  $X$ is a proper connected metric space
equipped with a locally doubling measure $\mu$ 
supporting a local \p-Poincar\'e inequality.
Let $\Om \subset X$ be a domain, i.e.\ a connected
open set, and let $x_0 \in\Om$.
Then the following are equivalent\/\textup{:}
\begin{enumerate}
\item \label{m-Green}
There is a \p-Green function in $\Om$ with singularity at $x_0$.
\item \label{m-sing}
There is a \p-singular function in $\Om$ with singularity at $x_0$.
\item \label{m-superh}
There is a nonconstant bounded \p-superharmonic function in $\Om$.
\item \label{m-hyp-cp}
$X$ is \p-hyperbolic or $\Cp(X \setm \Om)>0$, where $\Cp(\cdot)$ is
the Sobolev capacity defined in~\eqref{eq-Cp-deff}.
\end{enumerate}
\end{thm}

We use the following notions of \p-Green and \p-singular functions,
see also 
Definition~\ref{deff-sing} and Proposition~\ref{prop-sing-eqv}.
When $\Om$ is bounded, they are equivalent to the ones in
Bj\"orn--Bj\"orn--Lehrb\"ack~\cite[Definition~1.1]{BBLehGreen},
see Proposition~\ref{prop-sing-lika}.

A  nonconstant function $u:\Om \to (0,\infty]$, extended by zero outside $\Om$,
is a \emph{\p-singular function} in a domain $\Om\subset X$
with singularity at $x_0 \in \Om$ if it is \p-superharmonic
in $\Om$
and for every ball $B = B(x_0,r) \Subset \Om$ satisfies 
\begin{equation}  \label{eq-u=P-intro}
 u=P_{\Om \setm \clB} u \quad \text{in }  \Om \setm \clB.
\end{equation}
Here $P_{\Om \setm \clB} u$ is
the \p-harmonic \emph{Perron solution} in $\Om \setm \clB$ with boundary data $u$
on $\bdy B$ and 0 on the extended boundary
\[
\bdystar \Om:= \begin{cases}   
     \bdy\Om \cup \{\infty\},   &  \text{if $\Om$ is unbounded,}  \\ 
     \bdy \Om,  & \text{otherwise,}  
\end{cases}
\]
where $\infty$ is the point added in the \emph{one-point compactification} of $X$
when $X$ is unbounded.

A \emph{\p-Green function} 
is a \p-singular function which in addition satisfies
\begin{equation} \label{eq-normalized-Green-intro}
\cp(\{x\in \Om:u(x)\ge b\},\Om) = b^{1-p},
\quad \text{when }
   0  <b < u(x_0),
\end{equation}
where the condenser capacity~$\cp$ is as in Definition~\ref{deff-cp}.

Roughly speaking, condition~\eqref{eq-u=P-intro} means that
$u$ is \p-harmonic in $\Om\setm\{x_0\}$ and  has zero boundary values on $\bdystar \Om$
in a weak sense, while \eqref{eq-normalized-Green-intro} properly
normalizes \p-singular functions.
This close connection between \p-singular and \p-Green functions is captured 
in the following theorem, which is proved in Section~\ref{sect-Green}.
From now on we drop the prefix $p$ in our terminology, 
i.e.\ we write ``Green'' rather than ``\p-Green'' and so on,
but (unless said otherwise) it is implicitly assumed  that $1<p<\infty$ is fixed.

\begin{thm} \label{thm-Green-intro}
Assume that  $1<p<\infty$ and that $X$ is a proper connected metric space
equipped with a locally doubling measure $\mu$ 
supporting a local \p-Poincar\'e inequality.

Let $v$ be a singular function in a domain $\Om\subset X$ with singularity at $x_0$.
Then there is a unique $\alp>0$ such that $u=\alp v$
is a Green function.
Moreover,
for $0 \le a <b \le  u(x_0)$,
\[
    \cp(\Om^b,\Om_a)=(b-a)^{1-p}
\quad \text{and} \quad
      \int_{\Om_a \setm \Om^b} g_u^p\,d\mu = b-a,
\]
where $\Om_a=\{x\in \Om:u(x)>a\}$,  $\Om^b=\{x\in \Om:u(x)\ge b\}$
and we interpret $\infty^{1-p}$ as~$0$.
\end{thm}

Green functions associated with quasilinear elliptic equations 
of the \p-Laplacian type on Riemannian manifolds were 
introduced in Holo\-pai\-nen~\cite[Section~3]{Ho},
and subsequently studied by e.g.\ 
Coulhon--Holo\-pai\-nen--Saloff-Coste~\cite{CoHoSC},
Holo\-pai\-nen~\cite{HoDuke}
and Kura~\cite{Kura}.

For bounded domains $\Om$ in metric spaces, the existence of singular and Green functions 
(with our definition) was proved in
Bj\"orn--Bj\"orn--Lehrb\"ack~\cite[Theorem~1.3]{BBLehGreen}.
Under different
assumptions and with a somewhat different definition, 
singular functions on metric spaces
were introduced already in
Holo\-pai\-nen--Shan\-mu\-ga\-lin\-gam~\cite[Theorems~3.12 and 3.14]{HoSh}.
They considered singular functions both on bounded domains and on the entire space.
However, in the latter case the capacity estimates in 
\cite[pp.~325--326]{HoSh} do not seem to be fully justified 
and thus it is not evident
that the obtained limit function is always nonconstant,
see  Remark~\ref{rmk-HoSh} for more details.

A different type of global Green functions 
based on the Cheeger differentiable structure~\cite{Cheeg}
on metric spaces
was studied by Bonk--Capogna--Zhou~\cite{BCZ}.
Those Green functions are defined in $Q$-parabolic Ahlfors $Q$-regular
spaces and are
unbounded both from above and from below.
The uniqueness Theorem~1.2 in~\cite{BCZ} therefore does not apply to the 
global Green functions considered  in this paper, 
which only exist in \p-hyperbolic spaces.
When $X$ is Ahlfors $Q$-regular and $p=Q$,
they also obtained uniqueness for 
$Q$-Green functions and Cheeger $Q$-Green functions of 
our type
in bounded regular domains, see \cite[Theorem~1.1 and Remark~1.1]{BCZ}.
Local estimates near the singularity for Green and singular functions 
in metric spaces were obtained in
Bj\"orn--Bj\"orn--Lehrb\"ack~\cite{BBLehIntGreen} and
Danielli--Garofalo--Marola~\cite{DaGaMa}.

Theorem~\ref{thm-main} with $\Om=X$ provides a characterization of \p-hyperbolic spaces.
Additional characterizations in terms of capacity
are given in Proposition~\ref{prop-hyp-char-1}.
Moreover, under global assumptions on the measure $\mu$, we have the following
result, which is proved at the end of Section~\ref{sect-reg}.
All these characterizations extend and complement 
Theorem~9.22 in~\cite{HeKiMa}, 
and contain new statements even for weighted~$\R^n$.
Example~\ref{ex-R-half-weighted} shows that 
\ref{i-hyp}$\not\imp$\ref{i-reg}
if the
global assumptions are replaced by local ones.
On the other hand, the proof shows
that \ref{i-reg}\imp\ref{i-reg-x0}\imp\ref{i-hyp}\eqv\ref{i-cap(X,X)} 
hold in general.
See Definition~\ref{deff-cp} for the definition of the condenser capacity~$\cp$.

\begin{thm} \label{thm-char-hyp-intro}
Assume that $1<p<\infty$ and that  $X$ is an unbounded complete metric space
equipped with a globally doubling measure $\mu$ 
supporting a global \p-Poincar\'e inequality.
Let $x_0 \in X$. 
Then the following are equivalent\/\textup{:}
\begin{enumerate}
\item \label{i-hyp}
$X$ is \p-hyperbolic, i.e.\ there is a compact set $K$ with $\cp(K,X)>0$.
\item \label{i-cap(X,X)}
$\cp(X,X)=\infty$.
\item \label{i-reg}
$\infty$ is regular for every unbounded open set.
\item \label{i-reg-x0}
$\infty$ is regular for $X\setm\{x_0\}$.
\item \label{i-integral}
\begin{equation*} 
\int_1^\infty \biggl( \frac{\rho}{\mu(B(x_0,\rho))} \biggr)^{1/(p-1)} \,d\rho<\infty.
\end{equation*}
\end{enumerate}
\end{thm}

In the literature on  Riemannian manifolds, there are various definitions 
and characterizations of \p-hyperbolicity,
which are equivalent under suitable assumptions, see e.g.\ 
Kesel\cprime man--Zorich~\cite{KeZor} ($p=n$), 
Holo\-pai\-nen~\cite[Proposition~1.7 and Corollary~4.12]{HoDuke}
and 
Coulhon--Holo\-pai\-nen--Saloff-Coste~\cite[Corollary~3.2 and Proposition~3.4]{CoHoSC}.
Similar notions in metric spaces
were considered in Holo\-pai\-nen--Koskela~\cite[Theorem~1.7]{HoKo},
Holo\-pai\-nen--Shan\-mu\-ga\-lin\-gam~\cite{HoSh}
and Bj\"orn--Bj\"orn--Lehrb\"ack~\cite{BBLehIntGreen}.

Another consequence of the existence of singular functions is the
following resolutivity result.
Here $f:\bdystar \Om \to \eR$ is \emph{resolutive}
if the upper and the lower Perron solutions 
of the Dirichlet boundary value problem coincide for the boundary
data~$f$.
As in Theorem~\ref{thm-main},  $\Cp(\cdot)$ is
the Sobolev capacity defined in~\eqref{eq-Cp-deff}.

\begin{thm} \label{thm-main-resolutive}
Assume that $1<p<\infty$ and that  $X$ is a complete metric space
equipped with a globally doubling measure $\mu$ 
supporting a global \p-Poincar\'e inequality.

Let $\Om\subset X$ be open.
If $X$ is \p-hyperbolic or $\Cp(X \setm \Om)>0$,
then every $f \in C(\bdystar \Om)$ is resolutive.
\end{thm}

Under global assumptions on $\mu$,
Theorem~\ref{thm-main-resolutive}
together with Example~\ref{ex-bdystar-one-pt} and 
Propositions~\ref{prop-Cp=0-para} and~\ref{prop-resolutive-cp=0}
gives a complete characterization of when all continuous functions
are resolutive.
In the setting of weighted $\R^n$, this was already shown by
Hei\-no\-nen--Kil\-pe\-l\"ai\-nen--Martio~\cite[p.~183]{HeKiMa}.
Proposition~\ref{prop-Cp=0-para}
shows that the condition $\Cp(X \setm \Om)>0$ cannot be dropped
from Theorem~\ref{thm-main-resolutive}.
In Proposition~\ref{prop-perturbation} we
obtain a perturbation result for Perron solutions, which is new
even in unweighted $\R^n$ for 
$2 \ne p <n$.

In addition to Perron solutions there are at least two
other approaches  that have been used to solve
the Dirichlet problem for \p-harmonic
functions in unbounded open sets $\Om$ in metric spaces
(with local or global assumptions as in 
Theorem~\ref{thm-main} or~\ref{thm-main-resolutive}).
Hansevi~\cite{Hansevi1} showed that when 
$f\in \Dp(\Om)$ (i.e.\  $g_f \in L^p(\Om)$),
there is a  unique solution $Hf$ with
boundary values~$f$ that 
minimizes the \p-energy~\eqref{eq-energy-int-intro} among all functions
with the same boundary values (in the Sobolev sense).
In the subsequent paper \cite{Hansevi2} 
he showed
that $Hf = Pf$ if $X$ is \p-parabolic (or more generally if
$\Om$ is a \p-parabolic set, see Definition~\ref{def-p-par-set}).
However, this equality can fail in the \p-hyperbolic case, e.g.\ for  
\[
\Om=\R^n\setm B(0,1),  \quad 1<p<n
\quad \text{and} \quad f=1 \text{ on } \bdy\Om, 
\]
where $Hf\equiv1$
and 
\begin{equation}   \label{eq-Pf-with-c}
Pf(x)=c+ (1-c)|x|^{(p-n)/(p-1)}, \quad \text{when $f(\infty)=c$},
\end{equation}
cf.\ also Example~\ref{ex-cap-not-same}.

Another approach is to use a quasiconformal mapping from $X$ to a bounded space
with one point corresponding to $\infty$, where the quasiconformal mapping
is tailored to preserve \p-harmonic functions. 
This also requires a suitable choice of measure on the target side.
With such an approach, the unbounded open set $\Om$ becomes 
bounded
and the point $\infty$ becomes a normal finite point, which can either be seen
as a boundary point of the transformed open set or (in some cases) as an inner point.
One can therefore apply the more developed theory for the Dirichlet
problem and boundary regularity on bounded open sets.
Such an approach (with $\infty$ as a boundary point)
was taken by Bj\"orn--Bj\"orn--Li~\cite{BBLi} 
for certain $p$ in Ahlfors $Q$-regular spaces under a suitable 
global $q$-Poincar\'e inequality, with $q$ depending on $Q$ and $p$. 

This type of approach, with a different quasiconformal mapping,
was used by
Gibara--Korte--Shan\-mu\-ga\-lin\-gam~\cite{GibaraKS24}
for
unbounded \emph{uniform} domains $\Om$ with bounded boundary $\bdy \Om$
in $X=\clOm$ and Besov  boundary values in a trace sense at $\bdy\Om$.
Such boundary values and solutions correspond in principle to Sobolev 
boundary values on the transformed bounded open set.
The requirements on $X$ are weaker in~\cite{GibaraKS24} than  in~\cite{BBLi}
(see \cite{GibaraKS24} for the precise assumptions),
but the uniformity assumption on $\Om$ also adds extra implicit assumptions on $X$,
as well as of course on~$\Om$.
For instance, $X$ can only have one end towards $\infty$.
Perron solutions $Pf$ are not discussed in~\cite{GibaraKS24},
but the uniqueness in \p-parabolic sets \cite[Theorem~9.1]{GibaraKS24} can be seen as 
an analogue of the equality $Hf = Pf$ in Hansevi~\cite{Hansevi1},
while the nonuniqueness in \p-hyperbolic sets reflects the boundary values at $\infty$
for the Perron solution $Pf$ as in \eqref{eq-Pf-with-c}.

As already mentioned, 
to obtain our results on singular and Green functions, we first
study the condenser capacity $\cp(E,\Om)$ and its properties.
It is defined for bounded $E$ in the usual way by minimizing the \p-energy integral 
in \eqref{eq-energy-int-intro} among functions satisfying $u=1$ on $E$ and $u=0$ outside $\Om$,
while  for unbounded $E$,  
\begin{equation*} 
\cp(E,\Om) :=\lim_{j \to \infty} \cp(E \cap B(x_0,j),\Om),
\end{equation*}
where $x_0 \in X$ is an arbitrary point,
see Definition~\ref{deff-cp}.

Capacitary potentials for $\cp(E,\Om)$
are a useful tool in our investigations.
When $\Om$ is bounded, they are defined as (regularized) 
admissible functions minimizing the \p-energy 
\eqref{eq-deff-cp} in Definition~\ref{deff-cp} of $\cp(E,\Om)$.
For unbounded $\Om$, such minimizing admissible functions need not exist, see 
Proposition~\ref{prop-not-admissible}.
We therefore give a new definition~\eqref{def-pot-unbdd} of capacitary 
potentials $\pot{\Om}{E}$ using superharmonic functions.
It coincides with the older definition
for bounded $\Om$  but exists also for unbounded~$E$ and~$\Om$.
A different construction for compact $E$ in Riemannian manifolds
using exhaustions by bounded sets
appears in Holo\-pai\-nen~\cite[p.~10]{Ho}.
We include noncompact $E$ in our investigations.

Even though the capacitary potential $\pot{\Om}{E}$ need not be an
admissible function for the capacity,
its energy is always equal to $\cp(E,\Om)$ 
provided that
$\cp(E,\Om)< \infty$,
see Propositions~\ref{prop-ex-pot-unbdd} and~\ref{prop-ex-pot-bdd}.
However, it can happen that the energy of $\pot{\Om}{E}$ is 
finite even when the capacity is infinite,
see Examples~\ref{ex-6.7} and~\ref{ex-R-half-weighted-cap}.
In fact, $\pot{X}{X} \equiv 1$ regardless of whether  $\cp(X,X)=0$ or $\cp(X,X)=\infty$
(these are the only possibilities by Proposition~\ref{prop-hyp-char-1}).

On (sub)Riemannian manifolds, condenser capacities were 
used in connection with Green functions in
Holo\-pai\-nen~\cite{Ho},~\cite{HoDuke} and
Coulhon--Holo\-pai\-nen--Saloff-Coste~\cite{CoHoSC}.
However, they did not study such properties as 
countable subadditivity.

The outline of the paper is as follows:
In Section~\ref{sect-ug} we introduce the necessary background  from
first-order analysis on metric spaces, while
in Section~\ref{sect-cap} we introduce and discuss basic properties for the condenser capacity
$\cp$  in rather general metric spaces.
The following is the main result in Section~\ref{sect-cap}.
Note that $p=1$ is included here.

\begin{thm} \label{thm-cp-loccomp}
Assume 
that $1 \le p<\infty$ and that
$X$ is a locally complete metric space equipped
with a positive complete  Borel  measure $\mu$
such that $0<\mu(B)<\infty$ for all 
balls $B \subset X$.
Let $\Om\subset X$ be open.
Then the following properties hold\/\textup{:}
\begin{enumerate}
\item \label{cp-outercap}
\textup{(}Outer capacity\/\textup)
If $E \subset \Om$ then
\begin{equation}   \label{eq-outer-cap}
  \cp(E,\Om) 
     = \inf_{\substack{G\text{ open} \\  E \subset G \subset \Om }} \cp(G,\Om).
\end{equation}
\item \label{cp-Choq-K}
If $K_1 \supset K_2 \supset \cdots $ are compact subsets 
of\/ $\Om$, then
\[
      \cp\biggl(\bigcap_{i=1}^\infty K_i,\Om\biggr) 
          = \lim_{i \to \infty} \cp(K_i,\Om).
\]
\end{enumerate}
\end{thm}
\medskip

In Sections~\ref{sect-pharm} and~\ref{sect-bdyreg} we recall the definitions of
\p-harmonic functions, Perron
solutions and regular boundary points, and discuss results that will be 
needed in the sequel.

In Section~\ref{sect-potentials}, we turn to capacitary potentials.
Our definition is given in \eqref{def-pot-unbdd}, and the key properties are deduced.
This is a bit technical and relies quite heavily on the nonlinear potential theory associated
with \p-harmonic and superharmonic functions.
The capacitary potentials are then used in Section~\ref{sect-prop-cap}
to show  the following result.

\begin{thm} \label{thm-cp}
Assume that  $1<p<\infty$ and that $X$ is a proper connected metric space
equipped with a locally doubling measure $\mu$ 
supporting a local \p-Poincar\'e inequality.
Then the following properties hold also for unbounded $\Om$\/\textup{:}
\begin{enumerate}
\item \label{cp-Choq-E}
If   $E_1 \subset E_2 \subset \cdots \subset \Om$ 
 then
\[
      \cp\biggl(\bigcup_{i=1}^\infty E_i,\Om\biggr) 
          = \lim_{i \to \infty} \cp(E_i,\Om).
\]
\item \label{cp-subadd}
\textup(Countable subadditivity\/\textup)
If $E_1, E_2,\dots \subset \Om$ then
\[
      \cp\biggl(\bigcup_{i=1}^\infty E_i,\Om\biggr) 
          \le \sum_{i=1}^\infty \cp(E_i,\Om).
\]
\end{enumerate}
\medskip
\end{thm}

In particular,  $\cp$ is a \emph{Choquet capacity},
i.e.\ it satisfies the properties in
Theorems~\ref{thm-cp-loccomp}\ref{cp-Choq-K} and~\ref{thm-cp}\ref{cp-Choq-E}
and is monotone in the first argument.
In Proposition~\ref{prop-cp-inc-gen}  we show that  
\begin{equation} \label{eq-Omj}
\cp(E,\Om_j) \to \cp(E,\Om), 
\quad \text{as } j \to \infty,
\end{equation}
when $E \subset \Om_1\subset \Om_2 \subset \dots \subset \Om = \bigcup_{j=1}^\infty \Om_j$
and $\cp(E,\Om_1)<\infty$.
At the same time, Example~\ref{ex-warning-ring} shows that in general
(even if $\Om$ is bounded), 
\[
\cp(E_j,\Om_j) \not\to \cp(E,\Om), 
\quad \text{as } j \to \infty,
\]
when $E_j\subset \Om_j$ are such that
$E_1\subset E_2 \subset \dots \subset E = \bigcup_{j=1}^\infty E_j$
and $\Om_1\subset \Om_2 \subset \dots \subset \Om = \bigcup_{j=1}^\infty \Om_j$,
even though the corresponding capacitary potentials 
always increase to the one for $\cp(E,\Om)$, by Lemma~\ref{lem-pot-conv}.
Proposition~\ref{prop-cp=HeKiMa} shows  that 
our condenser capacity $\cp$
coincides with the one in 
Hei\-no\-nen--Kil\-pe\-l\"ai\-nen--Martio~\cite{HeKiMa},
when $X$ is weighted~$\R^n$ as in~\cite{HeKiMa}.

Superlevel set identities are part of the normalization \eqref{eq-normalized-Green-intro}
of Green functions.
In Section~\ref{sect-level}, we study such identities for 
capacitary potentials.
These results are used in Section~\ref{sect-Green} to obtain
similar formulas  for singular functions leading to 
Theorems~\ref{thm-main} and~\ref{thm-Green-intro}.
A perhaps surprising fact is that even though we are always assuming that $\Om$ 
is open, we need to use a pasting lemma for fine superminimizers
on finely open sets to obtain our results in Section~\ref{sect-level},
see Remark~\ref{rmk-fine}.

In Sections~\ref{sect-singular} and~\ref{sect-Green} we use many of our results
about the condenser capacity $\cp$ 
to study singular and Green functions.
Finally, in Section~\ref{sect-reg} 
 we apply our results on Green functions  to 
obtain new results about  
boundary regularity and Perron solutions for the Dirichlet boundary value problem
in unbounded domains.
Theorems~\ref{thm-char-hyp-intro} and~\ref{thm-main-resolutive}
are proved therein.

Our primary focus in this paper is on unbounded $\Om$.
However, there are some new results also for bounded $\Om$ and even bounded $X$.
In particular, 
the convergence \eqref{eq-Omj} seems to be new
even when $\Om$ is bounded
and $X$ is unweighted $\R^n$.

\begin{ack}
A.~B. resp.\ J.~B. were supported by the Swedish Research Council,
  grants 2020-04011 resp.\ 2018-04106 and 2022-04048.
  
We would like to thank the two anonymous referees 
whose insightful comments helped us
improve the exposition and the historical discussion.
\end{ack}

\section{Upper gradients and Newtonian spaces}
\label{sect-ug}

\emph{We assume throughout the paper
that $1 \le p<\infty$ and that
$X=(X,d,\mu)$ is a metric space equipped
with a metric $d$ and a positive complete  Borel  measure $\mu$
such that $0<\mu(B)<\infty$ for all 
balls $B \subset X$.
We also fix a point $x_0 \in X$, write $B_r=B(x_0,r)$,
and assume that $\Om \subset X$ is a nonempty open set.
In this and the next section we do not impose any other assumptions
on the space or the measure.
Additional standing
assumptions are added at  the beginning of Section~\ref{sect-pharm}.}

\medskip

In this section, we introduce
the necessary metric space concepts used in this paper.
Proofs of the results in 
this section can be found in the monographs
Bj\"orn--Bj\"orn~\cite{BBbook} and
Hei\-no\-nen--Koskela--Shan\-mu\-ga\-lin\-gam--Tyson~\cite{HKSTbook}.
Note that it follows from the 
assumption $0<\mu(B)<\infty$ for all balls $B$
that $X$ is separable.

A \emph{curve} is a continuous mapping from an interval,
and a \emph{rectifiable} curve is a curve with finite length.
We will only consider curves which are nonconstant, compact and rectifiable.
A curve can thus be parameterized by its arc length $ds$.
A property holds for \emph{\p-almost every curve}
if the curve family $\Ga$ for which it fails has zero \p-modulus,
i.e.\ there is $\rho\in L^p(X)$ such that
$\int_\ga \rho\,ds=\infty$ for every $\ga\in\Ga$.
Following Hei\-no\-nen--Koskela~\cite{HeKo98} and Koskela--MacManus~\cite{KoMc}
we define  \p-weak upper gradients.

\begin{deff} \label{deff-ug}
A measurable function $g:X \to [0,\infty]$ is a \emph{\p-weak upper gradient}
of $f:X \to \eR:=[-\infty,\infty]$
if for \p-almost all curves
$\gamma: [0,l_{\gamma}] \to X$,
\begin{equation*} 
        |f(\gamma(0)) - f(\gamma(l_{\gamma}))| \le \int_{\gamma} g\,ds,
\end{equation*}
where the left-hand side is $\infty$
whenever at least one of the
terms therein is infinite.
\end{deff}

If $f$ has a \p-weak upper gradient in $\Lploc(X)$, then
it has a \emph{minimal \p-weak upper gradient}
$g_f \in \Lploc(X)$
in the sense that
$g_f \le g$ a.e.\
for every \p-weak upper gradient $g \in \Lploc(X)$ of $f$.
The Newtonian Sobolev space on a metric space $X$ was
introduced as follows by 
Shan\-mu\-ga\-lin\-gam~\cite{Sh-rev}.

\begin{deff} \label{deff-Np}
For measurable $f$, let
\[
        \|f\|_{\Np(X)} = \biggl( \int_X |f|^p \, d\mu
                + \inf_g  \int_X g^p \, d\mu \biggr)^{1/p},
\]
where the infimum is taken over all \p-weak upper gradients of $f$.
The \emph{Newtonian space} on $X$ is
\[
        \Np (X) = \{f: \|f\|_{\Np(X)} <\infty \}.
\]
\end{deff}
\medskip

The space $\Np(X)/{\sim}$, where  $f \sim h$ if and only if $\|f-h\|_{\Np(X)}=0$,
is a Banach space and a lattice.
The \emph{Dirichlet space} $\Dp(X)$ is the collection of all
measurable functions on $X$ 
that have a \p-weak upper gradient in $L^p(X)$.
In this paper, it is convenient to
assume that functions in $\Np$ and $\Dp$
 are defined everywhere (with values in $\eR$),
not just up to an equivalence class in the corresponding function space.

We say  that $f \in \Nploc(X)$ if
for every $x \in X$ there exists a ball $B_x\ni x$ such that
$f \in \Np(B_x)$. 
If $f,h \in \Nploc(X)$,
then $g_f=g_h$ a.e.\ in $\{x \in X : f(x)=h(x)\}$.
In particular $g_{\min\{f,c\}}=g_f \chione_{\{f < c\}}$ a.e.\ in $X$
for $c \in \R$, where $\chione$ denotes the characteristic function.
For a measurable set $E\subset X$, the 
spaces $\Np(E)$ and $\Nploc(E)$ are defined by
considering $(E,d|_E,\mu|_E)$ as a metric space in its own right.

The  \emph{Sobolev capacity} of an arbitrary set $E\subset X$ is
\begin{equation} \label{eq-Cp-deff}
\Cp(E) =\inf_{f}\|f\|_{\Np(X)}^p,
\end{equation}
where the infimum is taken over all $f \in \Np(X)$ such that
$f\geq 1$ on $E$.
It is easy to see that the Sobolev capacity is countably subadditive.
For further properties, see 
e.g.\ \cite{BBbook}, Eriksson-Bique--Poggi-Corradini~\cite{EB-PC}
and Hei\-no\-nen--Koskela--Shan\-mu\-ga\-lin\-gam--Tyson~\cite{HKSTbook}.

A property holds \emph{quasieverywhere} (q.e.)\
if the set of points  for which it fails
has Sobolev capacity zero.
The Sobolev capacity is the correct gauge
for distinguishing between two Newtonian functions:
$h \sim f$ if and only if $h=f$ q.e.
Moreover, if $f,h \in \Np(X)$ and $f= h$ a.e., then $f=h$ q.e.

For a ball $B=B(x,r)$ with centre $x$ and radius $r$, we let
$\lambda B = B(x, \lambda r)$. In metric spaces
it can happen that balls with different centres or
radii denote the same set.
We will, however,
make the convention that a ball $B$ comes with a predetermined
centre and radius.
All balls are assumed to be open in this paper.
By $\clB$ we mean the closure of the open ball $B=B(x,r)$,
not the (possibly larger) set $\{y:d(x,y)\le r\}$.

\begin{deff} 
The measure $\mu$ is \emph{doubling within  $\Om$}
if there is $C>0$  
such that $\mu(2B)\le C \mu(B)$ for all balls $B \subset \Om$.

Similarly, the
\emph{\p-Poincar\'e inequality holds within $\Om$}
if there are constants $C>0$ and $\lambda \ge 1$
such that for all balls $B\subset \Om$,
all integrable functions $u$ on $\la B$, and all 
\p-weak upper gradients $g$ of $u$ in $\la B$,

\begin{equation*}  
        \vint_{B} |u-u_B| \,d\mu
        \le C r_B \biggl( \vint_{\lambda B} g^{p} \,d\mu \biggr)^{1/p},
\end{equation*}
where $u_B:=\vint_B u \,d\mu := \int_B u\, d\mu/\mu(B)$.

Each of these properties is called \emph{local} if 
for every $x \in X$ there is $r>0$
such that the property holds within $B(x,r)$.
If a property holds within $\Om=X$, then it is called \emph{global}.
\end{deff}

In $\R^n$ equipped with a globally doubling measure $d\mu=w\,dx$,
the global \p-Poincar\'e inequality  with $1<p<\infty$
is equivalent to the \emph{\p-admissibility} of the weight~$w$ in the
sense of Hei\-no\-nen--Kil\-pe\-l\"ai\-nen--Martio~\cite{HeKiMa}, see
Corollary~20.9 in~\cite{HeKiMa}
and Proposition~A.17 in~\cite{BBbook}.
Moreover, in this case $g_u=|\nabla u|$ a.e.\ if $u \in \Np(\R^n)$,
where $\nabla u$ is the gradient from~\cite{HeKiMa}.

The space $X$ is \emph{proper} if  
all closed bounded subsets are compact.
A \emph{domain} is a nonempty connected open set.
As usual, we write $f_\limplus= \max\{f,0\}$.
In this  paper, a continuous function is always assumed
to be real-valued (as opposed to $\eR$-valued). 
As usual, by $E \Subset \Om$ we mean that $\itoverline{E}$
is a compact subset of $\Om$.

We will need the following $\Dp$-version of
Corollary~6.3 in~\cite{BBbook} several times.
See also Lemma~\ref{lem-reflex-conv} for a stronger statement
under additional assumptions.
Example~6.5 in~\cite{BBbook} shows that Lemma~\ref{lem-Dp-6.3}
fails for $p=1$.

\begin{lem}   \label{lem-Dp-6.3}
Assume that $1<p<\infty$.
  Let $u_j \in \Dp(X)$ be a 
sequence 
such that $u_j \to u$ q.e.\ in $X$.
Assume that  there is a constant $M$ such that 
\[
|u_j|\le M \text{ q.e.\ in } X \quad \text{and} \quad
\int_X g_{u_j}^p\,d\mu \le M
\qquad \text{for } j=1,2,\dots. 
\]
Then $u \in \Dp(X)$ and
\[
  \int_X g_{u}^p\,d\mu
  \le \liminf_{j \to \infty} \int_X g_{u_j}^p\,d\mu.
\]
\end{lem}

\begin{proof}
Recall that $x_0 \in X$ and $B_k=B(x_0,k)$.
It follows from Corollary~6.3 in~\cite{BBbook} that $u \in \Np(B_k)$ for
$k=1,2,\dots$   and 
\[
\int_X g^p_{u} \, d\mu
= \lim_{k \to \infty} \int_{B_k} g^p_{u} \, d\mu
\le \lim_{k\to \infty} \liminf_{j \to \infty} \int_{B_k} g^p_{u_j} \, d\mu 
\le \liminf_{j \to \infty} \int_{X} g^p_{u_j} \, d\mu.
\]
In particular, $u \in \Dp(X)$.
\end{proof}

\section{The condenser capacity}
\label{sect-cap}

Recall the standing assumptions from the beginning of Section~\ref{sect-ug}.
We define the condenser capacity in the following way.
It is also called ``variational'' (e.g.\ in \cite{BBbook})
and  ``relative'' in various parts of the literature.

\begin{deff} \label{deff-cp}
The \emph{condenser capacity} of a bounded set
$E \subset \Om$ with respect to (a possibly unbounded) $\Om$ is
\begin{equation}  \label{eq-deff-cp}
\cp(E,\Om) = \inf_u\int_{\Om} g_u^p\, d\mu,
\end{equation}
where the infimum is taken over all $u \in \Np(X)$ 
such that $\chione_E  \le u \le \chione_\Om$.
We call such $u$ \emph{admissible} for 
$\cp(E,\Om)$.
If no such function $u$ exists then $\cp(E,\Om)=\infty$.
We also let $\cp(\emptyset,\emptyset)=0$.
If $E \subset \Om$ is unbounded, we define
\begin{equation} \label{eq-cp-lim-deff}
\cp(E,\Om)=\lim_{j \to \infty} \cp(E \cap B_j,\Om).
\end{equation}
\end{deff}

It is easy to see that the limit in \eqref{eq-cp-lim-deff}
does not depend on the choice of $x_0$ in $B_j=B(x_0,j)$.

There are  other condenser capacities appearing in the 
nonlinear metric space literature,
such as $\cpDp$, which is defined for disjoint sets $E$ and $F$ by
\[
\cpDp(E,F)=\inf_u \int_X g_u^p\, d\mu,
\] 
with the infimum taken over all functions $u\in \Dp(X)$ with $u=1$ on $E$
and $u=0$ on~$F$.
Without the measurability requirement for $u\in \Dp(X)$, 
the corresponding capacity $\cpbar(E,F)$
was defined in 
Hei\-no\-nen--Koskela~\cite[(2.12)]{HeKo98}.
They showed that it
can equivalently be defined using the \p-modulus as
\begin{equation*} 
\cpbar(E,F) = \Modp(E,F) := \inf_\rho \int_X\rho^p\, d\mu,
\end{equation*}
where the infimum is taken over all 
Borel functions $\rho \ge 0$  such that 
$\int_\ga\rho\, ds\ge 1$ for every 
curve $\ga$ in $X$ with one end point in $E$ and the other in $F$,
see \cite[Proposition~2.17]{HeKo98}.
(Their proof of this identity holds for arbitrary $E$ and $F$.)
These condenser capacities have also been studied
by e.g.\ Kallunki--Shan\-mu\-ga\-lin\-gam~\cite{KaSh},
Bj\"orn--Bj\"orn--Shan\-mu\-ga\-lin\-gam~\cite{BBShypend}
and Eriksson-Bique--Poggi-Corradini~\cite{EB-PC}.

For closed $E$ and $F$ in a locally complete space $X$, 
it was shown in~\cite[Theorem~1.1]{EB-PC} that
\begin{equation*} 
  \cpbar(E,F)=\cpDp(E,F) = \cpLip(E,F),
\end{equation*}
where $\cpLip(E,F)$ is defined as $\cpDp(E,F)$ but 
using only locally Lipschitz functions $u$.
(The separability assumption in~\cite{EB-PC}
is redundant since they assume (as here) that balls have positive and
finite measure.)

The papers \cite{BBShypend}, \cite{EB-PC}, \cite{HeKo98}, \cite{KaSh}
only consider condenser capacities for closed sets, and do not discuss
such properties as those in Theorems~\ref{thm-cp-loccomp} and~\ref{thm-cp}
for condenser capacities.

When $\Om$ is bounded, we have $\cp(E,\Om)=\cpDp(E,X \setm \Om)$
for arbitrary $E\subset\Om$,
but this is not at all true for unbounded $\Om$, see Example~\ref{ex-cap-not-same} below.
In particular, if $\Om=X$ then 
$\cpDp(E,X \setm \Om)=\cpDp(E,\emptyset)=0$ for every $E$,  
so $\cpDp$ is not well adapted
for studying global Green functions, 
which are normalized by \eqref{eq-normalized-Green-intro}.
In the rest of the paper (except for Example~\ref{ex-cap-not-same}), 
we therefore do not consider such capacities.

The basic properties  of $\cp$ are quite well understood for bounded $\Om$,
see e.g.\ 
Bj\"orn--Bj\"orn~\cite[Section~6.3]{BBbook} and \cite{BBvarcap}, as well as
Costea~\cite{Costea09}, Lehrb\"ack~\cite{Lehr12} and
Martio~\cite{Martio21}
with slightly different definitions.
But there seem to be few such studies 
for unbounded $\Om$ in metric spaces.
In weighted $\R^n$ (with a somewhat different definition of $\cp$),
this was done in
Hei\-no\-nen--Kil\-pe\-l\"ai\-nen--Martio~\cite[Chapter~2]{HeKiMa}.

That the Sobolev capacity~\eqref{eq-Cp-deff} is 
a Choquet (and in particular outer) capacity 
has been shown under various assumptions, most recently by
Eriksson-Bique--Poggi-Corradini~\cite[Theorem~1.5 and Corollary~1.6]{EB-PC}
only assuming that $X$ is locally complete 
(in addition to the assumptions from the beginning of  Section~\ref{sect-ug}).

In general metric spaces, the zero sets for the 
Sobolev capacity~\eqref{eq-Cp-deff}
determine the equivalence classes in
the Newtonian space $\Np(X)$, see Section~\ref{sect-ug}.
Similarly, the condenser capacity $\cp(E,\Om)$ is defined so as to have
essentially the same zero sets as the Sobolev capacity,
see Proposition~\ref{prop-cp}\ref{cp-Cp} below.
In particular, it takes advantage of the fact that functions in $\Np(X)$ are pointwise
defined, and so the requirement $u=1$ makes sense on arbitrary sets $E$.
This is a reason why these definitions differ somewhat from some
classical definitions.

With less precisely defined functions, capacities are often defined by requiring 
$u=1$ in an
open neighbourhood or by a three step procedure (first for compact sets using sufficiently
smooth functions, then for open and finally for arbitrary sets $E$).
This 
directly makes them into outer capacities.
However, the other property necessary for Choquet capacities, 
namely \ref{cp-Choq-E}
in Theorem~\ref{thm-cp} (and consequently the countable
subadditivity), is more involved.
The elegant argument for this from Kinnunen--Martio~\cite{KiMaNov}
uses a uniform bound on the $L^p$ norms of the functions used for defining the 
capacities,
which for our condenser capacity with $\Np$ functions
can fail in unbounded $\Om$.

As far as we know, the condenser capacity with $\Np$ functions for unbounded $\Om$ in 
metric spaces has earlier only been defined by \eqref{eq-deff-cp}, without
the second step \eqref{eq-cp-lim-deff}.
For bounded $E$ this does not play 
any role, but for unbounded $E$, \eqref{eq-cp-lim-deff} is essential,
as we shall see.
For compact $E$ and unbounded $\Om$ in metric spaces,
$\cp$  seems 
to appear for the first time in Holo\-pai\-nen--Koskela~\cite{HoKo}.
In Holo\-pai\-nen--Shan\-mu\-ga\-lin\-gam~\cite[Definition~2.3]{HoSh},
it was defined by \eqref{eq-deff-cp} and used
in connection with singular functions, see Remark~\ref{rmk-HoSh}.
However, Proposition~\ref{prop-cpt} below shows that
if $X$ is a \p-parabolic space with $\mu(X)=\infty$ then the 
capacity defined only by \eqref{eq-deff-cp} is neither
countably subadditive nor a Choquet capacity.

The proofs of countable subadditivity and Choquet capacity 
for bounded $\Om$  in~\cite{BBbook}
(based  on Kinnunen--Martio~\cite{KiMaNov})
do not seem to generalize to unbounded $\Om$
because of the lack of a uniform $L^p$ bound for the 
admissible functions.
The results in weighted $\R^n$ in
Hei\-no\-nen--Kil\-pe\-l\"ai\-nen--Martio~\cite{HeKiMa}
  do cover unbounded $\Om$, but the condenser capacity 
therein is defined differently by the three step procedure.
However, Proposition~\ref{prop-cp=HeKiMa} shows  that 
our condenser capacity $\cp$ coincides with the one in 
Hei\-no\-nen--Kil\-pe\-l\"ai\-nen--Martio~\cite{HeKiMa},
when $X$ is weighted~$\R^n$ as in~\cite{HeKiMa}.
This in particular gives an alternative proof of
Theorem~2.2\,(v) in~\cite{HeKiMa}.

The following simple properties can be deduced
without any further assumptions on $X$.
(Note that $p=1$ is included in this section.)

\begin{prop} \label{prop-cp}
The following properties hold\/\textup{:}
\begin{enumerate}
\renewcommand{\theenumi}{\textup{(\roman{enumi})}}%
\item \label{cp-Cp}
If $E \subset \Om$
and $\Cp(E)=0$, then $\cp(E,\Om)=0$.
\item \label{cp-subset}
\textup{(}Monotonicity\/\textup{)}
If $E_1 \subset E_2 \subset \Om \subset \Om'$ and $\Om'$ is open,
 then 
\[ 
\cp(E_1,\Om') \le \cp(E_2,\Om).
\]
\item \label{cp-strong-subadd}
\textup{(}Strong subadditivity\/\textup{)}
If  $E_1,  E_2 \subset \Om$ 
then
\[ 
   \cp(E_1 \cup E_2,\Om) + \cp(E_1 \cap E_2,\Om) \le \cp(E_1,\Om)+\cp(E_2,\Om).
\]
\item \label{cp-F-bdyF}
If $F$ is a bounded closed set and $F \subset \Om$, 
then\/ $\cp(F,\Om)=\cp(\bdy F,\Om)$.
\end{enumerate}
\end{prop}

To prove \ref{cp-F-bdyF} we will use
the following simple lemma,
which will also be needed later on.
See Proposition~\ref{prop-cp-inc-gen} for a stronger result, under stronger
assumptions.

\begin{lem} \label{lem-cp-inc-1}
Assume that $E$ is bounded.
Then
\[
  \cp(E,\Om)=\lim_{j \to \infty} \cp(E,\Om \cap B_j).
\]
\end{lem}

\begin{proof}
Assume first that $\cp(E,\Om)< \infty$.
Let $u$ be admissible for $\cp(E,\Om)$ and
$k$ be so large that $E \subset B_k$.
For $j > k$, let $\eta_j(x)=(1-\dist(x,B_{j-1}))_\limplus$.
Then $u\eta_j$ is admissible for $\cp(E,\Om \cap B_{j})$.
Moreover,  $u\eta_j\to u$ in $\Np(X)$, as $j \to \infty$, and thus
\[
  \int_X g_{u}^p\,d\mu =  \lim_{j \to \infty}   \int_X g_{u\eta_j}^p\,d\mu
  \ge \lim_{j \to \infty} \cp(E,\Om \cap B_{j}).
\]
Taking the infimum over all $u$ admissible for $\cp(E,\Om)$ gives
\[
\cp(E,\Om) \ge \lim_{j \to \infty} \cp(E,\Om \cap B_j),
\]
which is trivial if $\cp(E,\Om)=\infty$.
The converse inequality follows directly from monotonicity.
\end{proof}

\begin{proof}[Proof of Proposition~\ref{prop-cp}]
  \ref{cp-Cp}
The characteristic function $\chione_E$ is admissible for $\cp(E,\Om)$
  and hence $\cp(E,\Om)=0$.

  \ref{cp-subset}
This follows directly from the definition since there are 
fewer admissible functions for the latter capacity.

\ref{cp-strong-subadd}
We may assume that the right-hand side is finite (otherwise
there is nothing to prove).
Fix $j \ge 1$ and $\eps >0$.
Then there are $u_k$  admissible for
$\cp(E_k \cap B_j,\Om)$
such that 
\[
\int_{\Om} g_{u_k}^p \,d\mu < \cp(E_k \cap B_j,\Om) + \eps, 
\quad k=1,2.
\]
It follows that $u:=\max\{u_1,u_2\}$ 
is  admissible for
$\cp((E_1 \cup E_2)\cap B_j,\Om)$ and
$v:=\min\{u_1,u_2\}$ for
$\cp((E_1 \cap E_2)\cap B_j,\Om)$.
Thus, by  Corollary~2.20 in~\cite{BBbook}, 
\begin{align*}
   \cp((E_1 \cup E_2) \cap B_j, \Om) & + \cp((E_1 \cap E_2)  \cap B_j, \Om) \\
     &  \le  \int_X ( g_u^p + g_v^p) \, d\mu \\
      & = \int_X ( g_{u_1}^p + g_{u_2}^p) \, d\mu\\
   &\le \cp(E_1 \cap B_j, \Om) + \cp(E_2  \cap B_j, \Om)
   + 2\eps \\
   &\le \cp(E_1, \Om) + \cp(E_2, \Om)
   + 2\eps.
\end{align*}
Letting $\eps \to 0$ and then $j \to \infty$ completes the proof
of~\ref{cp-strong-subadd}.

\ref{cp-F-bdyF}  
Let $j$ be so large that $F \subset B_j$.
Then $\cp(F,\Om \cap B_j)=\cp(\bdy F,\Om \cap B_j)$, by
Theorem~6.17 in~\cite{BBbook}.
The identity therefore follows from Lemma~\ref{lem-cp-inc-1} after
letting $j \to \infty$.
\end{proof}

\begin{proof}[Proof of Theorem~\ref{thm-cp-loccomp}]
\ref{cp-outercap}
For bounded $\Om$, 
this follows 
from the quasicontinuity of Newtonian functions (Theorem~1.10 in
Eriksson-Bique--Poggi-Corradini~\cite{EB-PC})
in the same way as in the proof of 
Theorem~6.19\,(vii) in~\cite{BBbook}.
(Observe that $X$ is separable, see Section~\ref{sect-ug}.)
So assume that $\Om$ is unbounded.

We may assume that $\cp(E,\Om)< \infty$.
Let $\eps >0$ and $E_j=E \cap B_j$, $j=1,2,\dots$\,.
Using Lemma~\ref{lem-cp-inc-1}, define 
$k_j > k_{j-1} \ge j$ recursively so that
\begin{equation*}  
\cp(E_j, \Om \cap B_{k_j}) < \cp(E_j,\Om) + 2^{-j} \eps \le \cp(E,\Om) + 2^{-j} \eps,
\quad j=1,2,\dots.
\end{equation*}
As $\Om_j:=\Om \cap B_{k_j}$ is bounded, there is an open set $G_j$ such that 
$E_j\subset G_j \subset \Om_j$ and
\begin{equation*}  
\cp(G_j, \Om) \le \cp(G_j, \Om_j) 
< \cp(E_j, \Om) + 2^{-j} \eps.
\end{equation*}
Let $\Ghat_j = \bigcup_{i=1}^j G_i$. 
We then have by induction 
\begin{equation}  \label{eq-cap-Ghat-j}
\cp(\Ghat_j, \Om) \le \cp(E_j, \Om) + \sum_{i=1}^{j} 2^{-i}\eps
< \cp(E, \Om) + \eps,
\quad j=1,2,\dots.
\end{equation}
Indeed, the strong subadditivity (Proposition~\ref{prop-cp}\ref{cp-strong-subadd})
gives
\begin{align*} 
\cp(\Ghat_j, \Om) 
&\le \cp(\Ghat_{j-1}, \Om) + \cp(G_j, \Om) - \cp(\Ghat_{j-1}\cap G_j, \Om) \\
&\le \cp(E_{j-1}, \Om) + \sum_{i=1}^{j-1} 2^{-i}\eps
             + \cp(E_j, \Om) + 2^{-j} \eps - \cp(E_{j-1}, \Om) \\
&= \cp(E_j, \Om) + \sum_{i=1}^{j} 2^{-i}\eps.
\end{align*}
Next let
\[
G  = \bigcup_{j=2}^\infty (\Ghat_j \setm \clB_{j-2}),
\quad \text{where }  B_0:=\emptyset.
\]
Then $G$ is open and $E \subset G$.
If $x\in G \cap B_j$, then clearly $x\notin \Ghat_i \setm \clB_{i-2}$
for $i\ge j+2$.
Since $\Ghat_{j} \subset \Ghat_{j+1}$, $j=1,2,\dots$\,,
this implies that
\(
  G \cap B_j \subset \Ghat_{j+1}
\)
and thus \eqref{eq-cap-Ghat-j} shows that
\[ 
  \cp(G \cap B_j,\Om)
   \le   \cp(\Ghat_{j+1},\Om) 
  \le \cp(E, \Om) + \eps.
\] 
Letting $j \to \infty$ shows that  
$  \cp(G,\Om) \le \cp(E, \Om) +\eps$.
Since $\eps$ was arbitrary, this gives the $\ge$-inequality in
\eqref{eq-outer-cap},
while the reverse inequality is trivial.

\ref{cp-Choq-K}
This  follows from \ref{cp-outercap} as in
the proof of Theorem~6.7\,(viii) in~\cite{BBbook}.
\end{proof}

We conclude this section by showing that for unbounded $\Om$, 
the condenser capacities $\cp(E,\Om)$ and $\cpDp(E,X\setm\Om)$
need not coincide (even when $E$ is compact).

\begin{example}    \label{ex-cap-not-same}
Let $X=\R^n$ (unweighted) and $1<p<n$.
For some fixed $z\in\R^n$ with $|z|>1$, let $\Om=\R^n\setm \itoverline{B(z,1)}$
and $E=\itoverline{B(-z,1)}$.
Then by Proposition~3.6 in Bj\"orn--Bj\"orn--Shan\-mu\-ga\-lin\-gam~\cite{BBShypend},
\[
\cpDp(E,X\setm\Om) = \int_{\Om\setm E} |\grad u|^p\,dx,
\]
where $u\in\Dp(\Om\setm E)$ is the unique solution of the Dirichlet
problem in $\Om\setm E$ with boundary data $u=1$ on $E$ and $u=0$ on $\bdy\Om$,
as in Hansevi~\cite[Definition~4.6 and Theorem~4.1]{Hansevi1}.
By symmetry and the uniqueness of the solution, we see
that $u(x)=1-u(-x)$ and thus
\[
\lim_{x\to\infty} u(x)= \tfrac12,
\]
since the limit exists by Lemma~6.15 in 
Hei\-no\-nen--Kil\-pe\-l\"ai\-nen--Martio~\cite{HeKiMa}.
On the other hand, Proposition~\ref{prop-ex-pot-unbdd}
below shows that 
\[
\cp(E,\Om) = \int_{\Om\setm E} |\grad v|^p\,dx,
\]
where $v= \pot{\Om}{E}$ is the capacitary potential for $\cp(E,\Om)$
as defined by \eqref{def-pot-unbdd}.
Moreover, $v= P_{\Om \setm E} \chione_{E}$ is the Perron  solution
with boundary data $v=1$ on $E$ and $v=0$ on $\bdystar\Om$
as in Definition~\ref{def:Perron} below.
As $\R^n$ is \p-hyperbolic (since $p<n$), Theorem~\ref{thm-char-hyp-intro}
implies that $\infty$ is a regular boundary point as in 
Definition~\ref{deff-reg-pt} below,
which means that
\[
\lim_{x\to\infty} v(x)= 0.
\]
In particular, $v\ne u$. 
Since the minimizer in the Dirichlet problem for $\cpDp$ is unique, 
it follows that
\[
\cp(E,\Om) = \int_{\Om\setm E} |\grad v|^p\,dx 
> \int_{\Om\setm E} |\grad u|^p\,dx = \cpDp(E,X\setm\Om).
\]
\end{example}

\section{\texorpdfstring{\p}{p}-harmonic functions and Perron solutions}
\label{sect-pharm} 

\emph{In addition to the assumptions from the beginning of Section~\ref{sect-ug},
  we assume from now on that
  $X$ is a proper connected metric space
equipped with a locally doubling measure $\mu$ that
supports a local \p-Poincar\'e inequality, where $1<p<\infty$.
}

\medskip

By Proposition~1.2 and Theorem~1.3 in
Bj\"orn--Bj\"orn~\cite{BBsemilocal},
it follows from our assumptions
that $\mu$ is doubling and supports a \p-Poincar\'e
inequality within every ball.
(These properties are called semilocal in~\cite{BBsemilocal}.)

Even though we do not assume global doubling and a global \p-Poincar\'e inequality,
all the results
in Chapters~6--14 in~\cite{BBbook} (except for the Liouville theorem),
as well as the results in Hansevi~\cite{Hansevi2} 
hold under our assumptions, see the discussions in
Bj\"orn--Bj\"orn~\cite[Section~10]{BBsemilocal}
and Bj\"orn--Bj\"orn--Shan\-mu\-ga\-lin\-gam~\cite[Remark~3.7]{BBShypend}.
Also the results in Bj\"orn--Hansevi~\cite{BH1} 
and Bj\"orn--Bj\"orn--M\"ak\"a\-l\"ainen--Par\-vi\-ai\-nen~\cite{BBMP}
hold under these assumptions.

A function $u \in \Nploc(\Om)$ is a
\emph{\textup{(}super\/\textup{)}minimizer} in $\Om$
if 
\[ 
      \int_{\phi \ne 0} g^p_u \, d\mu
           \le \int_{\phi \ne 0} g_{u+\phi}^p \, d\mu
           \quad \text{for all (nonnegative) } \phi \in \Np_0(\Om),
\] 
where 
\[
\Np_0(\Om):=
  \{f|_{\Om} : f \in \Np(X) \text{ and }
        f=0 \text{ on } X \setm \Om\}.
\]
A continuous minimizer is called a \emph{\p-harmonic function}.
(It follows from Kinnunen--Shan\-mu\-ga\-lin\-gam~\cite{KiSh01} that
every minimizer has a continuous representative.)

The \emph{lsc-regularization} $u^*$ of a function $u:\Om \to \eR$ is defined by
\begin{equation}   \label{eq-def-lsc-reg}
 u^*(x):=\essliminf_{y\to x} u(y):= \lim_{r \to 0} \essinf_{B(x,r)} u,
\quad x\in\Om,
\end{equation}
and $u$ is \emph{lsc-regularized} if $u=u^*$.
By Theorem~5.1 in Kinnunen--Martio~\cite{KiMa02}
(or \cite[Theorem~8.22]{BBbook}), 
every superminimizer has 
an lsc-regularized representative.

Minimizers, \p-harmonic functions, superminimizers and superharmonic functions
were introduced in metric spaces
by 
Shan\-mu\-ga\-lin\-gam~\cite{Sh-harm},
Kinnunen--Shan\-mu\-ga\-lin\-gam~\cite{KiSh01}
and Kinnunen--Martio~\cite{KiMa02}.

A function $u : \Om \to (-\infty,\infty]$ is \emph{superharmonic}
  in $\Om$ if $\min\{u,k\}$ is an lsc-regularized superminimizer for every $k \in \R$
and $u$ is not identically $\infty$ in any component of $\Om$.
In particular, $u$ itself is lsc-regularized.
 A function $u: \Om \to [-\infty,\infty)$ is \emph{subharmonic}
   if $-u$ is superharmonic.

By Theorem~6.1 in Bj\"orn~\cite{ABsuper}
(or \cite[Theorems~9.24 and 14.10]{BBbook}),
this definition of superharmonicity is equivalent to  the ones used
both in the Euclidean and metric space literature, e.g.\
in Hei\-no\-nen--Kil\-pe\-l\"ai\-nen--Martio~\cite{HeKiMa},
Kinnunen--Martio~\cite{KiMa02}
and Bj\"orn--Bj\"orn~\cite{BBbook}. 
It is not difficult to see that a
function is \p-harmonic if and only if it is both
sub- and superharmonic.

We will need the following pasting lemma from
Bj\"orn--Bj\"orn--M\"ak\"a\-l\"ainen--Par\-vi\-ai\-nen~\cite[Lemma~3.13]{BBMP}
(or \cite[Lemma~10.27]{BBbook})
several times.

\begin{lem} \label{lem-pasting}
\textup{(Pasting lemma)}
Assume that $\Om_1 \subset \Om_2$ 
are arbitrary nonempty open sets and that
$u_1$  and $u_2$ are superharmonic in $\Om_1$ and $\Om_2$\textup{,} respectively.
Let
\[
     u=\begin{cases}
        u_2 & \text{in } \Om_2 \setm \Om_1, \\ 
        \min\{u_1,u_2\} & \text{in } \Om_1.
	\end{cases}
\]
If $u$ is lower semicontinuous\textup{,} then it is superharmonic in $\Om_2$.
\end{lem}

We are now ready to introduce the Perron solutions
for the Dirichlet problem.
Perron solutions will always be considered with respect to the
extended boundary 
\[
\bdystar \Om:= \begin{cases}   
     \bdy\Om \cup \{\infty\},   &  \text{if $\Om$ is unbounded,}  \\ 
     \bdy \Om,  & \text{otherwise,}  
\end{cases}
\]
where $\infty$ is the point added in the \emph{one-point compactification} 
$\Xstar:=X \cup \{\infty\}$
of $X$
when $X$ is unbounded.
If $X$ is bounded we simply let $\Xstar=X$.

In particular, $\bdystar \Om$ and $\Xstar$ are always compact.
Note that $\bdystar \Om=\emptyset$ if and only if $X$ is bounded and $\Om=X$.
In that
case, the Dirichlet problem for \p-harmonic functions
does not make sense.

\begin{deff}\label{def:Perron}
Assume that $\bdystar \Om \ne \emptyset$.
  Given $f:\bdystar\Omega\to\eR$, 
let $\UU_f(\Omega)$ be the collection of all superharmonic functions 
$u$ in $\Omega$ that are bounded from below and such that 
\[
	\liminf_{\Omega\ni y\to x} u(y) \geq f(x)
	\quad\textup{for all }x\in\bdystar\Omega.
\]
The \emph{upper Perron solution} of $f$ is defined by 
\[
	\uP_\Omega f(x)
	= \inf_{u\in\UU_f(\Omega)} u(x),
	\quad x\in\Omega.
\]
Let $\LL_f(\Omega)$ be the collection of all subharmonic functions 
$v$ in $\Omega$ that are bounded from above and such that 
\[
	\limsup_{\Omega\ni y\to x} v(y) \leq f(x)
	\quad\textup{for all }x\in\bdystar\Omega.
\]
The \emph{lower Perron solution} of $f$ is defined by 
\[
	\lP_\Omega f(x)
	= \sup_{v\in\LL_f(\Omega)} v(x),
	\quad x\in\Omega.
\]
If $\uP_\Omega f=\lP_\Omega f$, then we 
denote the common value by $P_\Omega f$.
If in addition $P_\Omega f$ is real-valued, then $f$ is said to be 
\emph{resolutive} (for $\Omega$). 
For simplicity,  we
often write $P f$ instead of $P_\Omega f$,
and similarly $\uP f$ and $\lP f$.
\end{deff}

In each component of $\Omega$, $\uP f$ is either \p-harmonic or 
identically $\pm\infty$, 
by Theorem~4.1 in Bj\"orn--Bj\"orn--Shan\-mu\-ga\-lin\-gam~\cite{BBS2}
(or \cite[Theorem~10.10]{BBbook}) whose
proof applies also here.

It follows immediately from the comparison principle
between sub- and superharmonic functions in 
Hansevi~\cite[Theorem~6.2]{Hansevi2} that
\begin{equation} \label{eq-lP-uP}
  \lP f \le \uP f
  \quad \text{for all } f: \bdystar \Om \to \eR.
\end{equation}  
(In \cite{Hansevi2} it is assumed that $\Cp(X \setm\Om)>0$,
but this is not needed in the proof, provided that
$\bdystar \Om \ne \emptyset$.)

We will need the following comparison principles and restriction results.
First, it follows directly from the definition that
\begin{equation*}  
\lP f_1\le \lP f_2 \quad \text{and} \quad \uP f_1\le \uP f_2
\qquad \text{when } f_1\le f_2. 
\end{equation*}

\begin{lem} \label{lem-simple}
Let $G_j$ be 
open sets such that  
$\bdystar G_j \ne \emptyset$, $j=1,2$.
 Let $f_j:\bdystar G_j \to[0,\infty]$
 be such that $f_1\le f_2$ on $\bdystar G_1 \cap \bdystar G_2$,
 \[
f_1=0 \ \text{on }  \bdy G_1 \cap G_2
\quad \text{and}  \quad 
f_2\ge \lP_{G_1} f_1 \ \text{on } \bdy G_2 \cap G_1.
 \]
Then $\lP_{G_1} f_1 \le \lP_{G_2}f_2$ in $G_1\cap G_2$.
\end{lem}

\begin{proof}
Let $v \in \LL_{f_1}(G_1)$ and 
\[
    \vt=\begin{cases}
        \max\{v,0\} & \text{in } G_1\cap G_2,  \\
        0 & \text{in } G_2 \setm G_1.
      \end{cases}
\]
Since  
\[
\limsup_{G_1 \ni y\to x} v(y) \le f_1(x) = 0 = \vt(x) \quad
\text{for }  x\in \bdy G_1\cap G_2, 
\]
we see that $\vt$ is upper semicontinuous in $G_2$.
Hence, by the pasting Lemma~\ref{lem-pasting},
$\vt$ is subharmonic in $G_2$. 
As also
\[
\limsup_{y\to x} \vt(y) \le 
\lP_{G_1}f_1(x) 
\le f_2(x) \quad
\text{for }  x\in \bdy G_2\cap G_1
\]
and $f_1\le f_2$ on $\bdystar G_1 \cap \bdystar G_2$,
it follows that $\vt \in \LL_{f_2}(G_2)$. 
Hence $\lP_{G_2} f_2 \ge \vt \ge v$ in $G_1\cap G_2$
and taking the supremum over all $v \in \LL_{f_1}(G_1)$ 
concludes the proof.
\end{proof}

\begin{lem} \label{lem-subset-Perron}
Assume that $\bdystar \Om \ne \emptyset$  and
$f:\bdystar\Om\to\eR$.
Let
\[
     h=\begin{cases}
        f & \text{on } \bdystar \Om, \\
        \uP f & \text{in }  \Om.
        \end{cases}
\]
Let $G \subset \Om$ be nonempty and open.
Then 
\begin{equation*} 
       \uP_{G} h \le \uP f
       \quad \text{in } G.
\end{equation*}
In particular, if $f$ is resolutive  for $\Om$ then $h$
is resolutive for $G$ and $P_{G} h = P f$ in~$G$.
\end{lem}

\begin{proof}
Since it is easily verified that every $v \in \UU_{f}(\Om)$ belongs
to $\UU_{h}(G)$, the first inequality is immediate.
The second part follows by applying the first part to both
$f$ and $-f$, together with \eqref{eq-lP-uP} and the resolutivity of $f$.
\end{proof}

The following 
simple observation will play a role later in this paper.

\begin{example} \label{ex-bdystar-one-pt}
Assume that $\bdystar \Om=\{a\}$ consists of one point (with $a=\infty$ 
if $\Om=X$ is unbounded).
Then every function $f:\bdystar \Om \to \eR$ is constant, and 
it is easy to see that 
$P f \equiv f(a)$.
Hence every $f:\bdystar \Om \to \R$ 
is resolutive and $a$ is 
trivially a regular boundary point (see Definition~\ref{deff-reg-pt} below).
\end{example}

The following convergence result will be used several times.

\begin{lem}   \label{lem-reflex-conv}
Let $u_j \in \Dp(X)$ be a  sequence 
such that $u_j \to u$ q.e.\ in $X$.
Assume that  there is a constant $M$ such that 
\[
|u_j|\le 1
\text{ q.e.\ in } X \quad \text{and} \quad
\int_X g_{u_j}^p\,d\mu \le M
\qquad \text{for } j=1,2,\dots. 
\]
Then $u \in \Dp(X)$, $|u| \le 1$ q.e.\ in $X$,
and for each $j=1,2,\dots$\,, there is a 
convex combination $\uhat_j$ of the sequence
$\{u_i\}_{i=j}^\infty$ such that $\uhat_j \to u$ q.e.\ in $X$
and
\[
\|g_{\uhat_j}-g_u\|_{L^p(X)}\to 0, \quad \text{as } j\to\infty.  
\]
\end{lem}

Recall that a \emph{convex combination} of $\{u_i\}_{i=j}^\infty$
  is a finite linear combination
  $\sum_{i=j}^N \alp_i u_i$
with nonnegative coefficients $\alp_i \ge 0$ such that 
$\sum_{i=j}^N \alp_i=1$.

\begin{proof}
Clearly, $\uhat_j \to u$ q.e.\ in $X$ and $|u| \le 1$ q.e.\ in $X$.
It follows from Lemma~\ref{lem-Dp-6.3} that $u \in \Dp(X)$, and thus
$u \in \Np(B)$ for every ball $B$.
The last conclusion can now be shown verbatim as in the proof
of Lemma~3.4 in
Bj\"orn--Bj\"orn--Shan\-mu\-ga\-lin\-gam~\cite{BBShypend}, 
ignoring the second paragraph of that proof which  is not relevant here.
\end{proof}

The classification of unbounded spaces as hyperbolic or parabolic 
plays an important role in this paper.

\begin{deff} \label{def-p-par}
Assume that $X$ is unbounded.
The space $X$ is \emph{\p-parabolic}
if $\cp(K,X)=0$ for all compact sets $K\subset X$.
Otherwise, $X$ is \emph{\p-hyperbolic}. 
\end{deff}

Note that unweighted $\R^n$ is \p-parabolic if and only if $p \ge n$
(and $p>1$).
Also any unbounded space with finite measure is \p-parabolic
since the constant function $1$ is admissible for $\cp(K,X)$.
We are now able to explain why 
it is essential to have the second step \eqref{eq-cp-lim-deff}
in Definition~\ref{deff-cp}.

\begin{prop} \label{prop-cpt}
For arbitrary\/ \textup{(}even unbounded\/\textup{)} $E \subset \Om$, let
\begin{equation*}  
\cpt(E,\Om) = \inf_u\int_{\Om} g_u^p\, d\mu,
\end{equation*}
where the infimum is taken over all $u \in \Np(X)$
such that $\chione_E  \le u \le \chione_\Om$.
Assume that $X$ is \p-parabolic and that $\mu(X)=\infty$.
Then the following hold\textup{:}
\begin{enumerate}
\item \label{k-b}
$\cpt(E,X)=0$ for every bounded  set $E\subset X$.
\item \label{k-c}
  $\cpt(X,X)=\infty$.
\item \label{k-d}
  $\cpt$ is not countably subadditive.
\item \label{k-e}
  $\cpt$ is not a Choquet capacity. 
\end{enumerate}
\end{prop}  

Note that for bounded $E$,
\[
\cpt(E,\Om) = \cp(E,\Om).
\]

\begin{proof}
\ref{k-b} 
This follows directly from the definitions, since closed bounded sets
are compact.

\ref{k-c}
As $\mu(X)=\infty$ there is no $u \in \Np(X)$ such that $u = 1 $ on $X$.
Hence $\cpt(X,X)=\infty$.

\ref{k-d} and \ref{k-e}
By \ref{k-b} and~\ref{k-c},
\[ 
      \lim_{r \to \infty} \cpt(\clB_r,X) = 0 < \infty = \cpt(X,X),
\] 
showing that $\cpt$ is not countably subadditive.
This also shows that $\cpt$ fails
the property in Theorem~\ref{thm-cp}\ref{cp-Choq-E},
and is thus not a Choquet capacity.
\end{proof}

\section{Boundary regularity}
\label{sect-bdyreg}

In this generality it can happen that
  there are nonresolutive $f\in C(\bdystar \Om)$,
see Proposition~\ref{prop-Cp=0-para} below.
We therefore make the following definition of regular boundary
points.

\begin{deff} \label{deff-reg-pt}
A boundary point $x \in \bdystar \Om$ is 
\emph{regular} (for $\Omega$) if
\[
\lim_{\Om \ni y \to x} \uP f(y) = f(x)
\quad \text{for all } f\in C(\bdystar \Om).
\]
Otherwise, $x$  is \emph{irregular}.
\end{deff}

We will rely on several results about boundary regularity in 
unbounded sets from Bj\"orn--Hansevi~\cite{BH1}.
As mentioned in Section~\ref{sect-pharm}, the
results in~\cite{BH1} hold under local assumptions as here.
Since \cite{BH1}  assumes
that $\Cp(X \setm \Om)>0$, some
care is needed to see that the results we need also hold when $\Cp(X \setm \Om)=0$.
The following statement clearly does as $I_\Om \subset X^* \setm\Om$.

\begin{thm} \label{thm-Kellogg}
\textup{(The Kellogg property \cite[Theorem~7.1]{BH1})}
If\/ $I_\Om \subset \bdystar \Om$
is the set of  irregular boundary points
for $\Om$, then $\Cp(I_\Om \cap X)=0$.
\end{thm}

The proof of the following result in~\cite{BH1}
did not use the assumption $\Cp(X \setm \Om)>0$.

\begin{prop}\label{prop-reg5.1}
\textup{(\cite[Theorem~5.1]{BH1})}
Let $x\in\bdystar\Omega$. 
Define $d_{x}: X^*\to[0,1]$  by 
\begin{equation*}
	d_{x}(y) 
	= \begin{cases}
		\min\{d(y,x),1\},
			& \text{if }y\neq\infty, \\
		1, 
			& \text{if }y=\infty,
	\end{cases}
\quad \text{when } x \ne \infty,
\end{equation*}
and 
\[
	d_\infty(y) 
	= \begin{cases}
         e^{-d(y,x_0)},	& \text{if }y\neq\infty, \\
		0, 	& \text{if }y=\infty.
	\end{cases}
\]

Then the following are equivalent\/\textup{:} 
\begin{enumerate}
\item\label{reg-reg} 
The point $x$ is regular. 
\item\label{reg-dx0}
It is true that 
\[
	\lim_{\Omega\ni y\to x}\uP d_{x}(y)
	= 0.
\]
\item\label{reg-semicont-x0}
It is true that 
\[
	\limsup_{\Omega\ni y\to x}\uP f(y)
	\leq f(x)
\]
for all $f:\bdystar\Omega\to[-\infty,\infty)$ 
that are bounded from above on 
$\bdystar\Omega$ and upper semicontinuous at $x$. 
\item\label{reg-cont-x0}
It is true that 
\[
	\lim_{\Omega\ni y\to x}\uP f(y)
	= f(x)
\]
for all $f:\bdystar\Omega\to\R$ that are 
bounded on $\bdystar\Omega$ and continuous at $x$. 
\item\label{reg-cont}
It is true that 
\[
	\limsup_{\Omega\ni y\to x}\uP f(y)
	\leq f(x)
\]
for all $f\in C(\bdystar\Omega)$. 
\end{enumerate}
\end{prop}

We will also need barriers.

\begin{deff}\label{def:barrier}
A function $u$ is a \emph{barrier} 
(for $\Omega$)
at $x\in\bdystar\Omega$ if
\begin{enumerate}
\renewcommand{\theenumi}{\textup{(\roman{enumi})}}%
\item\label{barrier-i}
$u$ is superharmonic in $\Omega$,
\item\label{barrier-ii}
$\lim_{\Omega\ni y\to x}u(y)=0$,
\item\label{barrier-iii}
$\liminf_{\Omega\ni y\to z}u(y)>0$ for every $z\in\bdystar\Omega\setm\{x\}$.
\end{enumerate}
\end{deff}

\begin{prop} \label{prop-barrier=>}
\textup{(\cite[Theorem~6.2]{BH1})}
If there is a barrier at $x \in \bdystar \Om$,
then $x$ is regular.
\end{prop}

This follows from the proof of (b)$\imp$(a) of Theorem~6.2 in~\cite{BH1}.
Note that the case $x=\infty$ is included in the proof of that implication
in \cite[p.~187]{BH1} 
and that the proof also holds when $\Cp(X \setm \Om)=0$
(since  Proposition~\ref{prop-reg5.1} above does.)

\begin{lem} \label{lem-reg-B=>Om}
Let $x \in \bdy \Om$ and $0 < r < \tfrac12 \diam X$.
If $x$ is regular for $\Om \cap B(x,r)$,
then it is regular for $\Om$.

If moreover $\Cp(X \setm \Om)>0$, then 
$x$ is regular for $\Om \cap B(x,r)$
if and only if it is regular for $\Om$.
\end{lem}

The condition $\Cp(X \setm \Om)>0$ in the second part
is essential, see Propositions~\ref{prop-Cp=0-para} and~\ref{prop-resolutive-cp=0} 
below.

\begin{proof}
For the implication in the first part,
note that $\Cp(X \setm (\Om \cap B(x,r))) >0$.
Applying  Theorem~6.2\,(a)\imp(f) in~\cite{BH1}
to $\Om \cap B(x,r)$ then
shows that regularity for $\Om \cap B(x,r)$ is equivalent to the
existence of a positive continuous barrier $u$ 
for $\Om \cap B(x,r)$
such that $u(x)\ge d(x,x_0)$ for all $x\in\Om \cap B(x,r)$. 
The pasting Lemma~\ref{lem-pasting}
shows that the
function $\min\{u,r\}$, extended by $r$ in $\Om\setm B(x,r)$, is 
a barrier in $\Om$. 
Proposition~\ref{prop-barrier=>} then implies regularity with respect to
$\Om$.

The equivalence in the second part 
is contained in Theorem~6.2 in~\cite{BH1}.
\end{proof}

The following result partly extends 
Proposition~7.3
in Bj\"orn--Bj\"orn--Shan\-mu\-ga\-lin\-gam~\cite{BBS2}
(and \cite[Proposition~10.32]{BBbook})
to unbounded sets.

\begin{lem} \label{lem-Perron-chi}
Assume that $\bdystar \Om \ne \emptyset$.
Let $E \subset \bdy \Om$ be  relatively open in $\bdy \Om$,
and let $I \subset \bdystar \Om$ be the set of 
irregular boundary points for $\Om$.
Then $\chione_{E \setm I}$ is resolutive and
\begin{equation*} 
      P\chione_{E \setm I} = \lP \chione_E.
\end{equation*}
\end{lem}

\begin{proof}
As $\chione_E$ is lower semicontinuous, it follows from
the regularity of $x \in E \setm I$ and
Proposition~\ref{prop-reg5.1}\ref{reg-reg}\imp\ref{reg-semicont-x0} that
\[
\liminf_{\Om \ni y \to x} \lP \chione_E(y) \ge 1
\quad \text{for every }
x \in E \setm I.
\]
Hence $\lP \chione_E \in \UU_{\chione_{E \setm I}}(\Om)$ and thus
\begin{equation} \label{eq-chi}
  \lP \chione_E \ge \uP \chione_{E \setm I}.
\end{equation}  

Let $v \in \LL_{\chione_E}(\Om)$ and $\eps >0$.
Then $V:=\{x \in \Om: v(x)>\eps\}$ is bounded 
since 
\[
\limsup_{x\to\infty} v(x) \le \chione_E(\infty)=0 
\quad \text{if $\Om$ is unbounded}.
\]
Let $R>1/\eps$ be such that $B_R\Supset V$.
Set $G=\Om \cap B_R$ ($G=\Om$ if $\Om$ is bounded)
and $E'=E\cap B_R$.
Since  $v$ is subharmonic and upper semicontinuous, we have
$v-\eps \in \LL_{\chione_{E'}}(G)$ and hence $v-\eps \le \lP_{G} \chione_{E'}$.

Let $I_G$ be the set of irregular boundary points for $G$.
It follows from Lemma~\ref{lem-reg-B=>Om} that 
$E'\setm I_G \subset E'\setm I$, 
and from the Kellogg property (Theorem~\ref{thm-Kellogg}) that
$\Cp(I_G)=0$.
Note that $\chione_{E'}$ is lower semicontinuous on $\bdy G$.
Proposition~7.3
in Bj\"orn--Bj\"orn--Shan\-mu\-ga\-lin\-gam~\cite{BBS2}
(or \cite[Proposition~10.32]{BBbook}),
applied on the bounded set $G$ to the boundary data
$f=-\chione_{E'}$ and $f+h=-\chione_{E'\setm I_G}$, 
together with Lemma~\ref{lem-simple}, then implies
\[
v-\eps \le \lP_{G} \chione_{E'} =
P_{G} \chione_{E' \setm I_G}
\le \lP \chione_{E \setm I}
\quad \text{in } G.
\]
Letting $\eps \to 0$  (and thus $R\to\infty$) gives 
\[
v\le \lP \chione_{E \setm I} \quad \text{in } \Om.
\]
Taking the supremum over all $v \in \LL_{\chione_E}(\Om)$
shows that $\lP \chione_E \le  \lP \chione_{E \setm I}$,
which together with \eqref{eq-lP-uP} and~\eqref{eq-chi}
concludes the proof.
\end{proof}

\section{Capacitary potentials}
\label{sect-potentials}

If $\Om$ is bounded 
and $\cp(E,\Om)<\infty$, then there exists $u\in\Np_0(\Om)$
such that $\chione_E  \le u \le \chione_\Om$ q.e.\ in $\Om$ and 
\begin{equation} \label{eq-11.19}
\cp(E,\Om)=\int_{\Om} g_u^p \, d\mu.
\end{equation}
If in addition $\Cp(X \setm \Om)>0$, then $u$
 can be obtained as the unique \emph{lsc-regularized solution}
of the $\K_{\chione_E,0}(\Om)$-obstacle problem, i.e.\ it satisfies
$u=u^*$ in the sense of \eqref{eq-def-lsc-reg} and
\[
\int_\Om g^p_{u} \, d\mu \le \int_\Om g^p_{v} \, d\mu
       \quad \text{for all $v \in \Np_0(\Om)$ with $v\ge\chione_E$ q.e.},
\]
see \cite[Definition~11.15 and Lemma~11.19]{BBbook}.
When extended by $0$ outside $\Om$, this function
is called the \emph{capacitary potential} for $\cp(E,\Om)$.
It is superharmonic and a superminimizer in $\Om$.

After a modification on the set $\{x\in E : u(x)<1\}$, the 
solution $u$ of the $\K_{\chione_E,0}(\Om)$-obstacle problem
is also admissible for $\cp(E,\Om)$.
In fact, minimizing admissible functions, i.e.\ 
those fulfilling \eqref{eq-11.19},
exist under much weaker assumptions than here, provided that $\Om$ is bounded,
see Bj\"orn--Bj\"orn~\cite[Theorem~4.2 and Remark~5.6]{BBnonopen}.
(The same proof applies also when $\Om$ is unbounded and $\mu(\Om)<\infty$.)

In contrast, the following result shows that when $\Om$ is unbounded there need
not exist any minimizing admissible function for $\cp(E,\Om)$.

\begin{prop} \label{prop-not-admissible}
Assume that $X$ is \p-parabolic and that $\mu(X)=\infty$.
Then 
\begin{equation} \label{eq-not-adm}
\cp(E,X)=0 \quad \text{for every set $E\subset X$.}
\end{equation}
Moreover, if $\Cp(E)>0$ then there is
no admissible function $u$ for $\cp(E,X)$ with
$\int_{X} g_u^p\, d\mu=0$, i.e.\ 
the infimum in~\eqref{eq-deff-cp}
is not attained for $\cp(E,X)$.
\end{prop}  

\begin{proof}
That \eqref{eq-not-adm} holds follows directly from Definitions~\ref{deff-cp}
and~\ref{def-p-par}.

If there was
an admissible function for $\cp(E,X)$ with 
$\int_{X} g_u^p\, d\mu=0$, then by the local \p-Poincar\'e inequality
and the connectedness of $X$, it would be constant a.e.\ 
(and thus q.e.\ by Proposition~1.59 in~\cite{BBbook}).
Since $\Cp(E)>0$, it would follow
that $u=1$ a.e.\ in $X$, which is impossible
as $\mu(X)=\infty$ and $u \in \Np(X)$.
\end{proof}

Despite the fact that there may not be any minimizing admissible function,
our
next aim is to construct suitable
capacitary potentials also for unbounded $\Om$.
(As a by-product we obtain them also when $X$ is bounded and
  $\Cp(X \setm \Om)=0$.)
For any open $\Om$, we define the \emph{capacitary potential} as
\begin{align} 
\pot{\Om}{E}(x) = & \begin{cases}
  \inf \{\psi(x):\psi \in \PsiQ{\Om}{E}\}, & \text{if } x \in \Om, \\
  0, & \text{otherwise},
  \end{cases} \label{def-pot-unbdd} \\
 & \quad \text{where }
 \PsiQ{\Om}{E}= \{\psi:\psi
    \text{ is superharmonic in $\Om$ and $\psi \ge \chione_E$ q.e.}\}.
\nonumber
\end{align}
This is a special case of the $Q$-balayage studied in
Bj\"orn--Bj\"orn--M\"ak\"a\-l\"ainen--Par\-vi\-ai\-nen~\cite{BBMP},
where the function in~\eqref{def-pot-unbdd}
is denoted
$\pot{1}{E}= \pot{\chione_E}{} = \Qhat^1_E=\Qhat^{\chione_E}$
with $\Om$ implicit.
Because of the q.e.-inequality in \eqref{def-pot-unbdd}, this definition
differs from the usual balayage and potential definition in e.g.\ 
Hei\-no\-nen--Kil\-pe\-l\"ai\-nen--Martio~\cite[Chapter~8]{HeKiMa}.
In particular, here we do not need to lsc-regularize the infimum in \eqref{def-pot-unbdd}.
By Theorem~4.4 in~\cite{BBMP}, $\pot{\Om}{E}\in \PsiQ{\Om}{E}$,
and it is therefore the minimal superharmonic function in $\Om$
which is  $\ge\chione_E$ q.e.\ in $\Om$.
It is convenient to have $\pot{\Om}{E}\equiv 0$ on $X \setm \Om$.

\begin{lem} \label{lem-pot-conv}
  Let   $\Om_1 \subset \Om_2 \subset \dots \subset \Om=\bigcup_{j=1}^\infty \Om_j$
  be open sets.

  If $E_1 \subset E_2 \subset \dots \subset E=\bigcup_{j=1}^\infty E_j$
  and $E_j \subset \Om_j$, $j=1,2,\dots$\,,
  then
  \[
  \pot{\Om_j}{E_j} \nearrow \pot{\Om}{E}
  \quad \text{everywhere in } X,  \text{ as } j \to \infty.
  \]
In particular, if $E_1 \subset E_2$, $\Om_1 \subset \Om_2$
  and $E_j \subset \Om_j$, $j=1,2$,
    then $\pot{\Om_1}{E_1} \le \pot{\Om_2}{E_2}$.
\end{lem}  

\begin{proof}
The restriction of any function in $\PsiQ{\Om_{j+1}}{E_{j+1}}$
belongs to $\PsiQ{\Om_{j}}{E_{j}}$ and thus
$\pot{\Om_j}{E_j} \le \pot{\Om_{j+1}}{E_{j+1}}$.
Similarly, $\pot{\Om_j}{E_j} \le \pot{\Om}{E}$.
So the limit $u:=\lim_{j \to \infty} \pot{\Om_j}{E_j}$ exists and 
$u \le \pot{\Om}{E}$.

For the reverse inequality, note first that
$u$ is a bounded increasing limit of superharmonic functions 
and thus itself superharmonic in $\Om$,
by Lemma~7.1 in Kinnunen--Martio~\cite{KiMa02}
(or \cite[Theorem~9.27]{BBbook}) and \cite[Proposition~9.21]{BBbook}.
Moreover,
$u  \ge \chione_{E_j}$ q.e.\ for each $j=1,2,\dots$\,, and
thus $u \ge \chione_E$ q.e.
Hence $u \in \PsiQ{\Om}{E}$
and $u=0$ outside $\Om$, and so $u \ge \pot{\Om}{E}$.
\end{proof}  

Next we 
show that our definition
of capacitary potentials agrees with the one in 
\eqref{eq-11.19}
when that one is defined.
For $E \Subset \Om$ this follows from Theorem~5.3 in
Bj\"orn--Bj\"orn--M\"ak\"a\-l\"ainen--Par\-vi\-ai\-nen~\cite{BBMP}.

\begin{prop} \label{prop-pot-Om-bdd}
If   $\Om$ is bounded, $\Cp(X \setm\Om)>0$ and $\cp(E,\Om)< \infty$, then
$\pot{\Om}{E}$ is the unique
lsc-regularized solution of the $\K_{\chione_E,0}(\Om)$-obstacle problem 
as in~\eqref{eq-11.19}.
\end{prop}

\begin{proof}
Let $u_j$ be the unique lsc-regularized solution 
of the $\K_{\chione_{E_j},0}(\Om)$-obstacle problem, where
\[
E_j = \{x \in E : \dist(x,X \setm \Om)\ge 1/j\} \Subset \Om,
\quad j=1,2,\dots.
\]
By Lemma~\ref{lem-pot-conv} and~\cite[Theorem~5.3]{BBMP},
\[
\pot{\Om}{E}=\lim_{j \to \infty} \pot{\Om}{E_j} = \lim_{j \to \infty} u_j.
\]
Lemma~\ref{lem-Dp-6.3} and~\eqref{eq-11.19} then imply that
\[
\int_X g^p_{\pot{\Om}{E}} \, d\mu
\le \liminf_{j \to \infty}   \int_X g^p_{u_j} \, d\mu
=  \liminf_{j \to \infty} \cp(E_j,\Om)
 \le \cp(E,\Om) <\infty.
 \]
Hence $\pot{\Om}{E} \in \Np_0(\Om)$ and so
$\pot{\Om}{E} \in \K_{\chione_E,0}(\Om)$.
Thus, in view of~\eqref{eq-11.19},  $\pot{\Om}{E}$ is an
lsc-regularized solution of the $\K_{\chione_E,0}(\Om)$-obstacle problem,
which is unique by Theorem~8.27 in~\cite{BBbook}.
\end{proof}

\begin{prop} \label{prop-ex-pot-unbdd} 
If $\cp(E,\Om)< \infty$, then  $\pot{\Om}{E} \in \Dp(X)$,
\begin{equation*} 
     \cp(E,\Om)=\int_{X} g_{\pot{\Om}{E}}^p \, d\mu
\end{equation*}
and $\pot{\Om}{E}$ is \p-harmonic in $\Om\setm\clE$.
If moreover $E\ne\Om$ is relatively closed in $\Om$,
  and $\bdystar \Om \ne \emptyset$ or $\Cp(E)>0$, then 
\begin{equation} \label{eq-cp-E-unbdd}
\pot{\Om}{E}=\lP_{\Om \setm E} \chione_E=P_{\Om \setm E} \chione_{E \setm I}
\quad \text{in } \Om \setm E,
\end{equation}
where $I \subset \bdystar (\Om \setm E)$
is the set of irregular boundary points for $\Om \setm E$.
\end{prop}

\begin{example} 
If $\bdystar \Om= \emptyset$ then $X$ must be bounded and $\Om=X$.
If in addition $E$ is closed and $\Cp(E)=0$,
 then $\pot{\Om}{E} \equiv 0$ but \eqref{eq-cp-E-unbdd} never holds: 
If $E=\{x\}$, then $I=\emptyset$ and
\[
P_{\Om \setm E} \chione_{E \setm I}  = P_{\Om \setm E} \chione_E\equiv 1,
\]
see Example~\ref{ex-bdystar-one-pt}.
If $E$  contains at least two points,
  then it follows from Proposition~\ref{prop-Cp=0-para} below that 
 $I=E$, 
\[
P_{\Om \setm E} \chione_{E \setm I}\equiv 0 \quad \text{and}
\quad  P_{\Om \setm E} \chione_E \equiv 1.
\]
Finally, if $E= \emptyset$, then the Perron solutions in \eqref{eq-cp-E-unbdd}
  are not even defined.
\end{example}

\begin{proof}[Proof of Proposition~\ref{prop-ex-pot-unbdd}]
  Let $E_j= E \cap \clB_j$, $j=1,2,\dots$\,.
By Lemma~\ref{lem-cp-inc-1} there is  $R_j > j$ such that
\[
\cp(E_j,\Om_j) <  \cp(E_j,\Om) + 1/j,
\quad \text{where } \Om_j = \Om \cap B_{R_j}.
\]
We may assume that $R_{j+1} \ge R_j$ for $j=1,2,\dots$\,.
By definition~\eqref{eq-cp-lim-deff} and monotonicity,
\begin{align*}
\cp(E,\Om)
& =\lim_{j \to \infty} \cp(E_j,\Om)
\le \liminf_{j \to \infty} \cp(E_j,\Om_j) \\
& \le \limsup_{j \to \infty} \cp(E_j,\Om_j)
\le \lim_{j \to \infty} (\cp(E_j,\Om) +1/j)
= \cp(E,\Om).
\end{align*}
Hence, by Proposition~\ref{prop-pot-Om-bdd} and~\eqref{eq-11.19},
\[
\cp(E,\Om)=\lim_{j \to \infty} \cp(E_j,\Om_j)
= \lim_{j \to \infty} \int_X g^p_{v_j}  
\, d\mu,
\quad \text{where } v_j:=\pot{\Om_j}{E_j}.
\]
Moreover, by Lemma~\ref{lem-pot-conv},
$v_j\nearrow \pot{\Om}{E}$.
Thus by Lemma~\ref{lem-Dp-6.3},
\[ 
\int_X g^p_{\pot{\Om}{E}} \, d\mu
\le \liminf_{j \to \infty} \int_{X} g^p_{v_j} \, d\mu
= \cp(E,\Om).
\] 

To prove the opposite inequality, Lemma~\ref{lem-reflex-conv} 
provides us with convex combinations
\[
u_j=\sum_{k=j}^{N_j}\la_{j,k}v_k,
\quad \text{where $0\le \la_{j,k}\le 1$ and
$\sum_{k=j}^{N_j}\la_{j,k}=1$,}
\]
such that $\|g_{u_j} - g_{\pot{\Om}{E}}\|_{L^p(X)} \to 0$,
as $j \to \infty$.
It follows from Proposition~\ref{prop-pot-Om-bdd} that
$u_j$ is admissible for $\cp(E_j,\Om_{N_j})$,
from which it follows that
\begin{align*} 
\int_{X} g_{\pot{\Om}{E}}^p \,d\mu
& =  \lim_{j \to \infty} \int_X g_{u_j}^p \,d\mu
\ge   \liminf_{j \to \infty} \cp(E_j,\Om_{N_j})  \\
& \ge   \lim_{j \to \infty} \cp(E_j,\Om)
=   \cp(E,\Om). 
\end{align*}

The \p-harmonicity of $\pot{\Om}{E}$
in $\Om\setm\clE$ follows  directly from Theorem~6.1 in 
Bj\"orn--Bj\"orn--M\"ak\"a\-l\"ainen--Par\-vi\-ai\-nen~\cite{BBMP}.

Next, assume that $E$ is relatively closed in $\Om$,
with $\bdystar \Om \ne \emptyset$ or $\Cp(E)>0$.
Consider first the case when $\Cp(E)=0$.
Then $\pot{\Om_j}{E_j} \equiv 0$
for all $j$ and hence also $\pot{\Om}{E} \equiv 0$.
Since $\Cp(E)=0$, every  $v \in \LL_{\chione_E}(\Om \setm E)$
has a subharmonic extension $\vt$ to~$\Om$, by
Theorem~6.3 in Bj\"orn~\cite{ABremove} (or \cite[Theorem~12.3]{BBbook}).
Using that $\bdystar \Om \ne \emptyset$ and
\[
\limsup_{\Om\setm E\ni y\to x} v(y)\le 0 \quad \text{for all } 
x\in\bdystar\Om \subset \bdystar (\Om\setm E),
\]
together with the usc-regularity
of $\vt$, the empty interior of $E$
and the maximum principle for subharmonic functions, 
we conclude that $\vt \le 0$, and so
$\lP_{\Om \setm E} \chione_E \equiv 0 $ in  $\Om \setm E$.
The last identity in \eqref{eq-cp-E-unbdd}
follows from Lemma~\ref{lem-Perron-chi}.

Finally, consider the case when $\Cp(E)>0$.  
Since $\pot{\Om_j}{E_j}\in \Np(X)$ is \p-harmonic in 
the bounded open set $\Om_j \setm E_j$,
it is the unique solution of the Dirichlet problem with itself as boundary data.
More precisely, 
Corollary~5.7 and Theorem~5.1 
in Bj\"orn--Bj\"orn--Shan\-mu\-ga\-lin\-gam~\cite{BBS2}
(or \cite[Corollary~10.16 and Theorem~10.15]{BBbook})
imply that 
\begin{equation}  \label{eq-1012}
\pot{\Om_j}{E_j}=P_{\Om_j \setm E_j} \pot{\Om_j}{E_j}=P_{\Om_j \setm E_j} \chione_{E_j}
\quad \text{in } \Om_j \setm E_j, \quad j=1,2,\dots,
\end{equation}
where in the second equality we have also used that $\pot{\Om_j}{E_j}=\chione_{E_j}$
q.e.\ on $\bdy (\Om_j \setm E_j)$.

Let $I$ be the set of irregular boundary points for $\Om \setm E$.
Since \[
\Cp(X\setm(\Om\setm E))\ge \Cp(E)>0, 
\]
Lemma~\ref{lem-reg-B=>Om} implies that
every $x \in \bdy(\Om \setm E) \cap(E\setm I)$ is also
regular for $\Om_k \setm E_k$ when $k>d(x,x_0)$.
For such $x$ and $k$, we have by
Lemma~\ref{lem-pot-conv}, \eqref{eq-1012},
Proposition~\ref{prop-reg5.1}\ref{reg-reg}$\imp$\ref{reg-cont-x0}
and the continuity of $\chione_{E_k}|_{\bdy(\Om_k\setm E_k)}$ at $x$
that
\[
\liminf_{\Om \setm E \ni y \to x} \pot{\Om}{E}(y)
   \ge \liminf_{\Om_k \setm E_k \ni y \to x} \pot{\Om_k}{E_k}(y) 
   =  \liminf_{\Om_k \setm E_k \ni y \to x} P_{\Om_k \setm E_k} \chione_{E_k}(y)
  = 1.
\]
Since also $\pot{\Om}{E} \ge 0$,
we get that $\pot{\Om}{E} \in \UU_{\chione_{E\setm I}}(\Om \setm E)$.
Noting also that $E\cap \bdy(\Om \setm E)$ is relatively open in $\bdy(\Om \setm E)$,
Lemma~\ref{lem-Perron-chi} implies that
\begin{equation}   \label{eq-pot-ge-lP}
\pot{\Om}{E} \ge P_{\Om \setm E} \chione_{E\setm I} 
= \lP_{\Om \setm E} \chione_{E}.
\end{equation}

Next, for $j=1,2,\dots$\,, Lemma~\ref{lem-simple} with $G_1=\Om_j\setm E_j$ 
and $G_2= \Om\setm E$ implies that
\[
\lP_{\Om \setm E} \chione_{E} \ge \lP_{\Om_j \setm E_j} \chione_{E_j}
\quad \text{in }\Om_j \setm E.
\]
Hence, using also \eqref{eq-1012} and \eqref{eq-pot-ge-lP},
\[
\pot{\Om}{E} \ge P_{\Om \setm E} \chione_{E\setm I} = \lP_{\Om \setm E} \chione_{E}
\ge 
\lP_{\Om_j \setm E_j} \chione_{E_j}
= \pot{\Om_j}{E_j}
\quad \text{in } \Om_j \setm E.
\]
Letting $j\to\infty$ and using Lemma~\ref{lem-pot-conv} shows
\eqref{eq-cp-E-unbdd}.
\end{proof}

When $E$ is bounded or $X$ is \p-parabolic,
the following equivalence holds.

\begin{prop} \label{prop-ex-pot-bdd} 
Let $E \subset \Om$ and assume that $E$ is bounded
or that $X$ is \p-parabolic. 
Then 
\begin{equation} \label{eq-cp-lim-Omj-bdd}
     \cp(E,\Om)=\int_{X} g_{\pot{\Om}{E}}^p \, d\mu.
\end{equation}
In particular, $\cp(E,\Om)< \infty$ if and only if $\pot{\Om}{E} \in \Dp(X)$.
\end{prop}  

\begin{proof}
If $\cp(E,\Om)< \infty$, then the conclusion follows from
Proposition~\ref{prop-ex-pot-unbdd}.
Assume conversely that  $\pot{\Om}{E} \in \Dp(X)$.

If $E$ is bounded, let $k$ be so large that $E \subset B_{k}$ 
and
let $\eta(x)=(1-\dist(x,B_k))_\limplus$ and $v=\eta \pot{\Om}{E}$.
Then $g_v \le g_\eta + g_{\pot{\Om}{E}}$ (by \cite[Theorem~2.15]{BBbook})
and thus $v \in \Dp(X)$.
As $v$ has bounded support, it follows that $v \in\Np(X)$
and hence
\[
   \cp(E,\Om) \le \int_X 
g_v^p\, d\mu < \infty.
\]

If $X$ is \p-parabolic, then for each $k=1,2,\dots$\,, find
$\eta_k\in\Np(X)$ such that 
\[
\chione_{B_k}\le\eta_k\le1 
\quad \text{and} \quad
\|g_{\eta_k}\|_{L^p(X)} \le1/k.
\]
Then $v_k:=\eta_k \pot{\Om}{E}$ is admissible for $\cp(E\cap B_k,\Om)$
and hence
\begin{align*}
\cp(E\cap B_k,\Om)^{1/p} 
 &\le \|g_{v_k}\|_{L^p(X)} 
\le \|g_{\eta_k}\|_{L^p(X)} + \|g_{\pot{\Om}{E}} \|_{L^p(X)}  \\
 &\le \|g_{\pot{\Om}{E}} \|_{L^p(X)} + 1/k. 
\end{align*}
Letting $k\to\infty$ shows that $\cp(E,\Om) < \infty$ also in this case.
\end{proof}

\begin{example}  \label{ex-6.7}
When $E$ is unbounded, the converse in Proposition~\ref{prop-ex-pot-bdd}
need not hold.
To see this, let
$E=\Om=X$ in a \p-hyperbolic space $X$.
In this case, clearly $\pot{X}{X} \equiv 1 \in \Dp(X)$.
However, as $X$ is \p-hyperbolic,
$\cp(X,X) \ge \cp(\clB_1,X)>0$ and \eqref{eq-cp-lim-Omj-bdd} does not
hold.
By Proposition~\ref{prop-ex-pot-unbdd}
or~\ref{prop-hyp-char-1},
we must have $\cp(X,X)=\infty$.
\end{example}

\begin{example} \label{ex-R-half-weighted-cap}
For a perhaps more illuminating example let
$X=\R$ and 
$d\mu=w\,dx$, where
$w>0$ is a continuous weight such that
$w(x) \equiv 1$ if $x \le 0$ and 
\begin{equation*} 
    \al_r:=  \int_r^{+\infty} w^{1/(1-p)}(t) \,dt \le \al_0 < \infty
\quad \text{for all $r\ge 0$.}
\end{equation*}
Since the weight $w$ is locally bounded from above and below from $0$,
$\mu$ is locally doubling and supports a local $1$-Poincar\'e inequality.
Moreover, all points have positive capacity and thus every
superharmonic and Newtonian function is continuous.

Let $E \subset [0,+\infty)$ be unbounded and $\Om=(a,+\infty)$,
with $a<0$.
We shall show that 
\[
    \int_{\R} g_{\pot{\Om}{E}}^p \, d\mu < \infty = \cp(E,\Om).
\]
Let  $r \ge 0$.
Lemma~6.2 in Bj\"orn--Bj\"orn--Shan\-mu\-ga\-lin\-gam~\cite{BBSliouville}
(or a direct calculation)
implies that  the function
\begin{equation}    \label{eq-def-ur}
u_r(x) := \frac{1}{\al_r} \int_x^{+\infty} w^{1/(1-p)}(t) \,dt
\end{equation}
is \p-harmonic in $(r,+\infty)$ with 
\begin{equation*}  
\lim_{x \to r\limplus} u_r(x)=1
\quad \text{and}  \quad \lim_{x \to +\infty} u_r(x)=0.
\end{equation*}
Clearly, $\cp(\{r\},\Om)< \infty$ and thus it follows from 
\eqref{eq-cp-E-unbdd} in Proposition~\ref{prop-ex-pot-unbdd}
that $\pot{\Om}{\{r\}}=u_r$ in $(r,+\infty)$ and 
\begin{equation} \label{eq-ur-alp_r}
     \cp(\{r\},\Om) 
= \int_{\R} g_{\pot{\Om}{\{r\}}}^p \, d\mu
\ge  \int_r^{+\infty} g_{u_r}^p \, d\mu
= \alp_r^{1-p}.
\end{equation}
Hence, as $E$ is unbounded,
\[
   \cp(E,\Om) 
\ge \sup_{r \in E}     \cp(\{r\},\Om) 
\ge \sup_{r \in E} \alp_r^{1-p} 
= \infty.
\]

On the other hand, 
it follows from the strong minimum principle 
(\cite[Theorem~9.13]{BBbook}) 
and another use of \eqref{eq-cp-E-unbdd} that
\[
    \pot{\Om}{E}(x)=\begin{cases}
      1, & \text{if } x \ge m, \\
      0, & \text{if } x \le a, \\
      P_{(a,m)} \chione_{\{m\}}(x), & \text{if }  a < x < m,
      \end{cases}
\qquad  
\text{where } m=\inf E.
\] 
Thus,
\[
    \int_{\R} g_{\pot{\Om}{E}}^p \, d\mu 
=\int_a^m g_{\pot{\Om}{E}}^p \, d\mu
\le \cp(\{m\},\Om) < \infty.
\]
Moreover, by Lemma~\ref{lem-cp-inc-1} and \eqref{eq-ur-alp_r},
\[
    \cp(\{0\},\R) 
    = \lim_{k \to \infty} \cp(\{0\},(-k,k)) 
    \ge  \lim_{k \to \infty} \cp(\{0\},(-k,+\infty)) 
    \ge \alp_0^{1-p}>0,
\]
and thus $X$ is \p-hyperbolic, by definition. 
\end{example}

\begin{lem} \label{lem-cp=0}
Let $E \subset \Om$.
If $\Cp(E)=0$, then $ \cp(E,\Om)=0$ and $\pot{\Om}{E} \equiv 0$.

If $\cp(E,\Om)=0 <\Cp(E)$, then $\pot{\Om}{E} \equiv \chione_\Om$
and $\Cp(X \setm \Om)=0$.
\end{lem}

\begin{proof}
If $\Cp(E)=0$, then $0 \in \Psi^\Om_E$, and so $\pot{\Om}{E} \equiv 0$.
Moreover, $\cp(E,\Om)=0$ by Proposition~\ref{prop-cp}.
  
If $\cp(E,\Om)=0 <\Cp(E)$, then 
by  Proposition~\ref{prop-ex-pot-unbdd},
\begin{equation*} 
     \int_{X} g_{\pot{\Om}{E}}^p \, d\mu =0.
\end{equation*}
Thus the \p-Poincar\'e inequality on larger and larger balls shows
that there is $c$ such that $\pot{\Om}{E} = c$ q.e.\ in $X$.
Since $\pot{\Om}{E}=1$ q.e.\ in $E$ and $\Cp(E)>0$, we see that $c=1$.
As $\pot{\Om}{E}$ is superharmonic and thus lsc-regularized  
in $\Om$ we get that $\pot{\Om}{E}\equiv 1$ in $\Om$.
On the other hand, $\pot{\Om}{E}\equiv 0$ in $X \setm \Om$, by definition,
which in particular implies that 
$\Cp(X \setm \Om)=0$.
\end{proof}

We can now characterize \p-hyperbolicity in the following ways.

\begin{prop} \label{prop-hyp-char-1}
  The following are equivalent\/\textup{:}
\begin{enumerate}
\renewcommand{\theenumi}{\textup{(\roman{enumi})}}%
\item \label{b-hyp}
$X$ is \p-hyperbolic, i.e.\ there is a compact set $K$ such
that $\cp(K,X)>0$,
\item \label{b-cp-2}
$\cp(E,X)>0$ for some  set $E$ with $\Cp(E)>0$,
\item \label{b-cp}
$\cp(E,X)>0$ for every set $E$ with $\Cp(E)>0$,
\item \label{b-cp-X}
$\cp(X,X)>0$,
\item \label{b-X}
$\cp(X,X)=\infty$.
\end{enumerate}
\end{prop}

\begin{proof}
Observe first that none of the conditions is true when $X$ is bounded.

  \ref{b-hyp}\imp\ref{b-cp-2}
This is trivial.

$\neg$\ref{b-cp}\imp$\neg$\ref{b-cp-2}
By assumption there is $E'$ with $\cp(E',X)=0<\Cp(E')$.
We need to show that  $\cp(E,X)=0$ for every set $E$  with $\Cp(E)>0$.
By monotonicity, we can assume that $E\supset E'$.
Let $k$ be so large that $\Cp(E' \cap B_k)>0$.
As $\cp(E' \cap B_k,X)=0$, Lemmas~\ref{lem-pot-conv} and~\ref{lem-cp=0} show that
\[
    1 \ge \pot{X}{E \cap B_k} \ge\pot{X}{E' \cap B_k} \equiv 1.
\]
Thus by Proposition~\ref{prop-ex-pot-bdd},
$\cp(E \cap B_k,X)=0$.
As  this holds for all large $k$, we see that $\cp(E,X)=0$, by~\eqref{eq-cp-lim-deff}.

\ref{b-cp}\imp\ref{b-cp-X}
This is trivial.

$\neg$\ref{b-X}\imp$\neg$\ref{b-cp-X}
By definition and lsc-regularity, $\pot{X}{X} \equiv 1$.
Since $\cp(X,X)<\infty$, Proposition~\ref{prop-ex-pot-unbdd} 
therefore implies that $\cp(X,X)=0$.

\ref{b-X}\imp\ref{b-hyp}
By definition, there is $k$ such that $\cp(\clB_k,X)\ge\cp(B_k,X)>0$.
\end{proof}

\begin{cor}  \label{cor-cpE-Om>0}
Let $E \subset \Om$ be  such that $\Cp(E)>0$.
Then $\cp(E,\Om)>0$
if and only if $X$ is \p-hyperbolic or 
$\Cp(X \setm \Om)>0$.
\end{cor}

\begin{proof}
If $X$ is \p-hyperbolic, then $\cp(E,\Om) \ge \cp(E,X)>0$,
by Proposition~\ref{prop-hyp-char-1}.
On the other hand, if $\cp(E,\Om)=0$, 
then $\Cp(X \setm \Om)=0$, by Lemma~\ref{lem-cp=0}.

Conversely, assume that $\Cp(X \setm \Om)=0$
and in addition that $X$ is \p-parabolic or bounded.
Let $j>0$.
If $u$ is admissible for $\cp(E \cap B_j,X)$, then $u\chione_\Om$
is admissible for $\cp(E \cap B_j,\Om)$.
Thus $\cp(E \cap B_j,\Om)=\cp(E \cap B_j,X)$.
Letting $j \to \infty$ shows that $\cp(E,\Om)=\cp(E,X)=0$,
by Proposition~\ref{prop-hyp-char-1}.
\end{proof}

\section{Further properties of the condenser capacity}
\label{sect-prop-cap}

In addition to the properties of $\cp$ deduced in 
Section~\ref{sect-cap},
we can now derive 
Theorem~\ref{thm-cp}.
It would be interesting to know if 
it holds in more general situations,
such as under the assumptions from the beginning
of Section~\ref{sect-ug}.
Note that \ref{cp-Choq-E} in Theorem~\ref{thm-cp} does not always
hold when $p=1$ even when $\Om$ is bounded, see
Example~6.18 in~\cite{BBbook}.
For bounded $\Om$,  
Theorem~\ref{thm-cp}
was
obtained in Theorem~6.17 in~\cite{BBbook}.

\begin{proof}[Proof of Theorem~\ref{thm-cp}]
\ref{cp-Choq-E}
The $\ge$-inequality is trivial, by monotonicity, and we may therefore
assume that all the capacities on the right-hand side are finite.
Let $j \ge 1$.
By Lemma~\ref{lem-pot-conv},
$\pot{\Om}{E_k \cap B_j} \nearrow \pot{\Om}{E \cap B_j}$, as $k \to \infty$.
As $B_j$ is bounded, it thus follows from
Lemma~\ref{lem-Dp-6.3} and 
Proposition~\ref{prop-ex-pot-bdd} that
\begin{align*}
\cp(E \cap B_j, \Om) & = \int_X g^p_{\pot{\Om}{E \cap B_j}} \, d\mu
\le \liminf_{k \to \infty} \int_{X} g^p_{\pot{\Om}{E_k \cap B_j}} \, d\mu \\
& = \lim_{k \to \infty} \cp(E_k \cap B_j,\Om)
\le \lim_{k \to \infty} \cp(E_k,\Om).
\end{align*}
Letting $j \to \infty$ shows the $\le$-inequality, and thus concludes
the proof of this part.

\ref{cp-subadd}
By Proposition~\ref{prop-cp}\ref{cp-strong-subadd} and induction,
for every $k=1,2,\dots$\,,
\[
        \cp\biggl(\bigcup_{i=1}^k E_i,\Om\biggr) 
          \le \sum_{i=1}^k \cp(E_i,\Om)
          \le \sum_{i=1}^\infty \cp(E_i,\Om).
\]
The conclusion thus follows from \ref{cp-Choq-E}.
\end{proof}  

As $\cp$ is monotone in the first argument and
satisfies the properties in
Theorems~\ref{thm-cp-loccomp}\ref{cp-Choq-K} and~\ref{thm-cp}\ref{cp-Choq-E},
it is a  \emph{Choquet capacity}.
An important consequence is the following result.
See e.g.\
Aikawa--Ess\'en~\cite[Part~2, Section~10]{AE} for a proof.
Note that every Borel set is a Suslin set (also called analytic set).

\begin{thm} \label{thm-Choq-cap}
\textup{(Choquet's capacitability theorem)}
If $E \subset \Om$ is a Borel or Suslin set, then 
$E$ is \emph{capacitable}, i.e. 
\[
     \cp(E,\Om)=\sup_{\substack{ K \text{ compact} \\K \subset E  }}
\cp(K,\Om)
         =\inf_{\substack{G\text{ open} \\  E \subset G \subset \Om  }}
\cp(G,\Om).
\]
\end{thm}

The following result extends Theorem~5.1 
from Bj\"orn--Bj\"orn~\cite{BBvarcap} to unbounded~$\Om$.
See Hei\-no\-nen--Kil\-pe\-l\"ai\-nen--Martio~\cite[Section~1.1]{HeKiMa} 
and Section~\ref{sect-ug} for the definition
of \p-admissible weights.

\begin{prop} \label{prop-cp=HeKiMa}
Assume that $X=\R^n$ equipped with a  measure
$d\mu = w\,dx$, where $w$ is a \p-admissible weight.
Let $\cpmu$ be the condenser capacity defined  in 
Hei\-no\-nen--Kil\-pe\-l\"ai\-nen--Martio\/~\textup{\cite[p.~27]{HeKiMa}}.
If  $E \subset \Om$, then
\begin{equation} \label{eq-cp=cpt}
  \cp(E,\Om)=\cpmu(E,\Om).
\end{equation}
\end{prop}

\begin{proof}
Consider first $K \subset \Om$ compact.
It follows directly from the definition in~\cite[p.~27]{HeKiMa}
that
\[
  \cpmu(K,\Om)=\lim_{j \to \infty} \cpmu(K,\Om \cap B_j).
\]
Hence, by \cite[Theorem~5.1]{BBvarcap} and Lemma~\ref{lem-cp-inc-1},
\[
\cpmu(K,\Om)
=\lim_{j \to \infty} \cpmu(K,\Om \cap B_j)
=\lim_{j \to \infty} \cp(K,\Om \cap B_j)
= \cp(K,\Om).
\]
Theorem~\ref{thm-Choq-cap} and the definition
in~\cite[p.~27]{HeKiMa} then imply that \eqref{eq-cp=cpt} holds for open $E$.
Finally, by Theorem~\ref{thm-cp-loccomp}\ref{cp-outercap}
and  the definition in~\cite[p.~27]{HeKiMa},
we conclude
that  \eqref{eq-cp=cpt}  holds for arbitrary
$E \subset \Om$.
\end{proof}

The following continuity result generalizes Lemma~\ref{lem-cp-inc-1}
and will be needed later on.

\begin{prop} \label{prop-cp-inc-gen}
Let $\Om_1 \subset \Om_2 \subset \dots \subset \Om=\bigcup_{j=1}^\infty \Om_j$
be an increasing sequence of open sets, and  $E \subset \Om_1$ be 
such that $\cp(E,\Om_1)< \infty$.
Then 
\begin{equation} \label{eq-cp-lim-Omj}
\cp(E,\Om) = \lim_{j \to \infty} \cp(E,\Om_j).
\end{equation}
\end{prop}  
\medskip

The following example shows that 
the condition $\cp(E,\Om_1) < \infty$ cannot be dropped.

\begin{example} \label{ex-warning-E}
Let $X=\R^2$, $1<p \le 2$, and
\begin{alignat*}{2}
E& =[-1,0) \times [-1,1], \\
\Om& =((-2,2) \times (-2,2)) \setm \{(0,0)\}, \\
\Om_j &= \Om\setm ([0,1/j] \times [-1/j,1/j]), & \quad &j=1,2,\dots. 
\end{alignat*}
Then $\Om=\bigcup_{j=1}^\infty \Om_j$.
In this case,  $\cp(E,\Om_j)=\infty$ for each $j$
since any admissible function $u$ for $\cp(E,\Om_j)$ would
satisfy 
\[
u=1 \text{ in }E \quad \text{and} \quad  u=0 \text{ in }
[0,1/j] \times [-1/j,1/j],
\]
and thus 
violate the \p-Poincar\'e inequality on the ball
$B((0,0),1/j)$.
However, as $\Cp(\{(0,0)\})=0$, we see that
\[
\cp(E,\Om)=\cp(E,(-2,2) \times (-2,2)) < \infty.
\]
Observe that all the involved sets are bounded.
The argument above works also for weighted $\R^2$ provided
that $\Cp(\{(0,0)\})=0$.
\end{example}

\begin{proof}[Proof of Proposition~\ref{prop-cp-inc-gen}]
The limit exists since the capacities therein decrease with~$j$.
The  $\le$-inequality in~\eqref{eq-cp-lim-Omj}
follows directly from monotonicity.
In particular, $\cp(E,\Om)<\infty$.

Conversely,
Lemma~\ref{lem-pot-conv} shows that $v_j:=\pot{\Om_j}{E} \nearrow \pot{\Om}{E}$.
By Proposition~\ref{prop-ex-pot-unbdd} and monotonicity,
\[
\int_X g^p_{v_j}\, d\mu 
= \cp(E,\Om_j) \le \cp(E,\Om_1) < \infty.
\]
Hence, Lemma~\ref{lem-reflex-conv} provides us with convex combinations
\begin{equation}  \label{eq-def-uj}
u_j=\sum_{i=j}^{N_j}\la_{j,i}v_i,
\quad \text{where $0\le \la_{j,i}\le 1$ and
$\sum_{i=j}^{N_j}\la_{j,i}=1$,}
\end{equation}
such that 
\begin{equation} \label{eq-g-conv}
\|g_{u_j} - g_{\pot{\Om}{E}}\|_{L^p(X)} \to 0,
\quad \text{as }j \to \infty.
\end{equation}
By definition, $\chione_E  \le u_j \le \chione_{\Om_{N_j}}$ q.e.\ in $X$.
If we knew that all $u_j \in \Np(X)$, this would conclude
the proof. 
Since this cannot be guaranteed
we proceed as follows.

For each $k=1,2,\dots$\,, there are $R_k\ge R_{k-1}\ge k$ such that the
capacitary potentials
\[
\eta_k:=\pot{\Om\cap B_{R_k}}{E\cap B_k}\in\Np(X)
\] 
satisfy 
\begin{align*}
\int_X g_{\eta_k}^p \,d\mu & = \cp(E\cap B_k,\Om\cap B_{R_k}) \\
&\le \cp(E\cap B_k,\Om) + 1/k
\le \cp(E,\Om) + 1/k.
\end{align*}
Lemma~\ref{lem-pot-conv} shows that $\eta_k \nearrow \pot{\Om}{E}$.
Another use of Lemma~\ref{lem-reflex-conv} provides us with convex combinations
\begin{equation*}  
\etah_k=\sum_{i=k}^{I_k}\al_{k,i}\eta_i,
\quad \text{where $0\le \al_{k,i}\le 1$ and
$\sum_{i=k}^{I_k}\al_{k,i}=1$,}
\end{equation*}
such that 
\begin{equation} \label{eq-g-conv-etah}
\|g_{\etah_k} - g_{\pot{\Om}{E}}\|_{L^p(X)} <1/k, \quad k=1,2,\dots.
\end{equation}
With $u_j$ as in \eqref{eq-def-uj}, the function $u_{j,k}:=\min\{2\etah_k,u_j\}$ 
is admissible 
(after modification on a set of capacity zero)
 for $\cp(E\cap B_k,\Om_{N_j})$  and 
\[
g_{u_{j,k}} = \begin{cases}
             g_{u_j}, & \text{if } u_j\le 2\etah_k, \\
             2g_{\etah_k}, & \text{if } 2\etah_k < u_j, \end{cases}
\]
(by \cite[Corollary~2.20]{BBbook}).
Hence 
\begin{equation}   
\label{eq-cp-with-min}
\cp(E\cap B_k,\Om_{N_j}) \le \int_X g_{u_{j,k}}^p\,d\mu 
\le \int_X g_{u_j}^p \,d\mu 
+ 2^p\int_{\{2\etah_k < u_j\}} g_{\etah_k}^p \,d\mu. 
\end{equation}
Note that $\chione_{\{2\etah_k < u_j\}}(x)\to 0$, as $k\to\infty$,
for all $x\in X$ and each $j=1,2,\dots$\,.
For the last integral we therefore have using \eqref{eq-g-conv-etah},
Proposition~\ref{prop-ex-pot-unbdd}
and dominated convergence that
\[
\biggl(\int_{\{2\etah_k < u_j\}} g_{\etah_k}^p \,d\mu\biggr)^{1/p} 
   \le \biggl(\int_{\{2\etah_k < u_j\}} g_{\pot{\Om}{E}}^p \,d\mu\biggr)^{1/p} + 1/k 
\to 0, \quad \text{as }k \to \infty.
\]
Inserting this into~\eqref{eq-cp-with-min}, letting $k\to\infty$ 
and using \eqref{eq-g-conv} and
Proposition~\ref{prop-ex-pot-unbdd}, then shows that 
\begin{equation}   \label{eq-cap-E-Nj}
\cp(E,\Om_{N_j}) \le \int_X g_{u_j}^p \,d\mu
\to \int_{X} g_{\pot{\Om}{E}}^p \,d\mu =
\cp(E,\Om),
\quad \text{as } j\to\infty,
\end{equation}
and so \eqref{eq-cp-lim-Omj} holds.
\end{proof}  

\begin{remark}
If all $\mu(\Om_j)<\infty$ in Proposition~\ref{prop-cp-inc-gen}, then
the proof can be simplified after \eqref{eq-g-conv} as follows:
Since $\mu(\Om_{N_j})<\infty$,  
we directly have $u_j \in \Np(X)$,
which together with 
Propositions~\ref{prop-cp} and~\ref{prop-ex-pot-unbdd}
immediately gives \eqref{eq-cap-E-Nj}.
\end{remark}

\begin{example}   \label{ex-warning-ring}
In view of Lemma~\ref{lem-pot-conv}, Theorem~\ref{thm-cp} 
and Proposition~\ref{prop-cp-inc-gen},
it is tempting to think that also
\begin{equation} \label{eq-warning-ring}
\cp(E_j,\Om_j)\to \cp(E,\Om), 
\quad \text{as } j\to\infty, 
\end{equation}
when
\begin{equation*}   
\Om_1 \subset \Om_2 \subset \dots \subset \Om=\bigcup_{j=1}^\infty \Om_j
\quad \text{and}  \quad
E_1 \subset E_2 \subset \dots \subset E=\bigcup_{j=1}^\infty E_j
\end{equation*}
are increasing sequences of sets with  $\Om_j$ open
and $E_j \subset \Om_j$, $j=1,2,\dots$\,.
By Example~\ref{ex-warning-E} this is not true in general
 even when $E_j=E$ is fixed. 
However in that example the limit is infinite, while 
Proposition~\ref{prop-cp-inc-gen} holds when the limit is finite.
We shall now see that \eqref{eq-warning-ring} is not true in general even 
when the limit of the left-hand side is finite.
Observe that all the involved sets are bounded.

Consider unweighted $\R^n$ with $1< p\le n$ and let $x_0=0$.
Recall that $B_r=B(0,r)$.
Note first that by Theorem~\ref{thm-cp},
\begin{equation} \label{eq-m1}
\lim_{r\to0} \cp(B_r,B_s)=0 
\quad \text{and} \quad \lim_{r\to s\limminus} \cp(B_r,B_s)=\infty
\quad \text{when $s>0$ is fixed}.
\end{equation}
(Exact formulas for the involved capacities in  unweighted $\R^n$
appear in Hei\-no\-nen--Kil\-pe\-l\"ai\-nen--Martio~\cite[Example~2.12]{HeKiMa}, 
but are not needed here.
The arguments below work also for weighted $\R^n$, provided that $\Cp(\{0\})=0$.)

The uniform Wiener type estimate in Corollary~6.36 in~\cite{HeKiMa} 
shows that for every $\eps>0$ there is $\de>0$ such that for all $s\in (s_0,1)$,
the capacitary potential $u_s$ for $\cp(B_r,B_s)$ satisfies
$u_s<\eps$ in $B_s\setm B_{s-\de}$.
Hence $(u_s-\eps)_\limplus/(1-\eps)$ is admissible for $\cp(B_r,B_{s_0})$
when $s<s_0+\de$.
Letting $\eps\to0$ and an application of Proposition~\ref{prop-cp-inc-gen}
(for $s\to s_0\limminus$) then imply that
\begin{equation} \label{eq-m2}
\lim_{s\to s_0} \cp(B_r,B_s) = \cp(B_r,B_{s_0}) \quad \text{when $0<r<s_0<1$ are  fixed}.
\end{equation}

Fix an arbitrary sequence $\{c_j\}_{j=0}^\infty$
such that  $c_j>c_0>0$, $j=1,2,\dots$\,.
We are now going to construct decreasing sequences $\{r_j\}_{j=1}^\infty$
and $\{s_j\}_{j=1}^\infty$  as follows:
By \eqref{eq-m1} and~\eqref{eq-m2}, 
there are
$0<r_1<s_1<1$ so that 
\[
\cp(B_{s_1},B_1) =c_0 \quad \text{and}  \quad
\cp(B_{r_1},B_{s_1}) = c_1 - c_0.
\]
For $j=2,3,\dots$\,,
choose recursively  $0<r_j<\min\{r_{j-1},2^{-j}\}$
and $s_j$ with $r_j < s_j < s_{j-1}$ so that
\[
\cp(B_{r_j},B_{s_{j-1}}) < c_j-c_0 \quad \text{and} \quad
\cp(B_{r_j},B_{s_{j}}) = c_j-c_0,
\]
which again is possible due to \eqref{eq-m1} and~\eqref{eq-m2}.

Let $E_j=\{x:s_j \le |x| \le s_1\}$ and 
$\Om_j=\{x:r_j<|x|<1\}$, $j=1,2,\dots$\,. 
Then
\[
\cp(E_j,\Om_j)= \cp(B_{r_j},B_{s_j}) +  \cp(B_{s_1},B_1) = c_j, 
\quad j=1,2,\dots.
\]
At the same time,  for $\Om:=\bigcup_{j=1}^\infty\Om_j= B(0,1)\setm\{0\}$
and $E:=\bigcup_{j=1}^\infty E_j$ we have
\[
\cp(E,\Om) =\cp(B_{s_1},B_1) = c_0.
\]
Hence, \eqref{eq-warning-ring} fails in general since $\lim_{j\to\infty}c_j$
can be chosen arbitrarily from $[c_0,\infty]$ 
and need not even exist.
\end{example}

Even though Example~\ref{ex-warning-ring} showed that
\eqref{eq-warning-ring} fails in general, one inequality is still true.

\begin{prop} \label{prop-cp-inc-liminf}
Let 
\begin{equation*}   
\Om_1 \subset \Om_2 \subset \dots \subset \Om=\bigcup_{j=1}^\infty \Om_j
\quad \text{and}  \quad
E_1 \subset E_2 \subset \dots \subset E=\bigcup_{j=1}^\infty E_j
\end{equation*}
be increasing sequences of sets with  $\Om_j$ open
and $E_j \subset \Om_j$, $j=1,2,\dots$\,.
Then 
\begin{equation*} 
\cp(E,\Om) \le \liminf_{j \to \infty} \cp(E_j,\Om_j).
\end{equation*}
\end{prop}

\begin{proof}
By monotonicity, $\cp(E_j,\Om) \le \cp(E_j,\Om_j)$.
Thus, by 
Theorem~\ref{thm-cp}\ref{cp-Choq-E},
\[
    \cp(E,\Om)=\lim_{j \to \infty} \cp(E_j,\Om) \le \liminf_{j \to \infty} \cp(E_j,\Om_j).
  \qedhere
\]
\end{proof}

The following equality
strengthens  
Theorem~6.19\,(x) in~\cite{BBbook}
even when $\Om$ is bounded.

\begin{prop} 
If $K \subset \Om$ is compact, then
\[
    \cp(K,\Om) =
        \inf_{\substack{u \in \Lipc(\Om) \\u \ge 1 \text{ on } K}}
\|g_u\|_{L^p(X)}^p,
\]
where\/ $\Lipc(\Om)=\{f \in \Lip(X) : \supp f \Subset \Om \}$.
\end{prop}

\begin{proof}
Since $K$ is compact, there is an integer $k \ge 1$ such that
$\dist(K,X \setm \Om)>1/k$ and $K \subset B_k$.
Let
\[
\Om_j=\{x \in \Om \cap B_j : \dist(x,X \setm \Om)> 1/j\},
\quad j=k,k+1,\dots,
\]
which are bounded open sets.
Then 
\[
K \Subset \Om_k \Subset \Om_{k+1} \Subset \dots \Subset \Om
= \bigcup_{j=k}^\infty \Om_j
\quad \text{and} \quad \cp(K,\Om_k)< \infty.
\]
Let $\eps >0$.
By Proposition~\ref{prop-cp-inc-gen}, there is $k_0 \ge k$ such that
\[
\cp(K,\Om_{k_0}) < \cp(K,\Om) + \eps.
\]
As $\Om_{k_0}$ is bounded,
it follows from
Theorem~6.19\,(x) in~\cite{BBbook} that 
there is $u \in \Lip(X)$ such that $u \ge 1$ on $K$, 
$u=0$ outside $\Om_{k_0}$ and
$\|g_u\|_{L^p(X)}^p < \cp(K,\Om_{k_0}) + \eps$.
Since  $\supp u \Subset \Om$, letting $\eps \to 0$
shows one inequality. 
The reverse inequality  is trivial.
\end{proof}

\section{Capacity of superlevel sets}
\label{sect-level}

Superlevel set identities are part of the definition of Green functions,
through \eqref{eq-normalized-Green-intro}.
We shall now study superlevel sets for capacitary potentials.
They will be used in Section~\ref{sect-Green} 
in our study of Green functions.  

Recall from~\eqref{def-pot-unbdd} that the 
capacitary potential for $E$ in $\Om$ is 
\begin{align*} 
\pot{\Om}{E}(x) = & \begin{cases}
  \inf \{\psi(x):\psi \in \PsiQ{\Om}{E}\}, & \text{if } x \in \Om, \\
  0, & \text{otherwise},
  \end{cases} \\ 
 & \quad \text{where }
 \PsiQ{\Om}{E}= \{\psi:\psi
    \text{ is superharmonic in $\Om$ and $\psi \ge \chione_E$ q.e.}\},
\nonumber
\end{align*}
and that $\pot{\Om}{E}\in\PsiQ{\Om}{E}$.

In this section we fix $E \subset \Om$, let $0 \le a  \le 1$
and consider the superlevel sets
\begin{equation}  
\label{eq-def-level-sets}
    \Om_a  = \{ x \in \Om : \pot{\Om}{E} > a \}
  \quad \text{and} \quad
   \Om^a  = \{ x \in \Om : \pot{\Om}{E} \ge a \}.
\end{equation} 
The set $\Om_a$ is open since
the potential $\pot{\Om}{E}$ is superharmonic, and thus
lower semicontinuous, in $\Om$.
If the potential $\pot{\Om}{E}$ is continuous in $\Om$, then $\Om^a$ is relatively
closed in $\Om$.
Note that $\Om_0=\Om$ whenever $\Om$ is connected and
 $\Cp(E)>0$, by the minimum principle for
superharmonic functions,
while $\Om_0=\Om_a=\Om^a=\emptyset$ for all $0<a\le1$ if $\Cp(E)=0$.
Clearly, $\Om^0=\Om$ for all $E\subset\Om$.
Also, recall that $\cp(\emptyset,\emptyset)=0$.

\begin{lem} \label{lem-level-both}
\begin{alignat}{3}
 \pot{\Om_a}{E \cap \Om_a} &=  \frac{\pot{\Om}{E}-a}{1-a}
   &\quad& \text{in } \Om_a, &\quad&\text{if\/ } 0 \le a <1, \nonumber\\
 \pot{\Om}{\Om_a} = \pot{\Om}{\Om^a}
   &= \min\biggl\{1,\frac{\pot{\Om}{E}}{a}\biggr\} &\quad& \text{in } \Om, 
  & \quad &\text{if\/ } 0 < a <1.
\label{potOmOma}
\end{alignat}
\end{lem} 

\begin{remark} \label{rmk-fine}
The proof below of \eqref{potOmOma}
uses fine continuity of superharmonic functions,
proved in Bj\"orn~\cite{JBfine} and Korte~\cite{RiikkaFine},
together with a pasting lemma for fine superminimizers
from Bj\"orn--Bj\"orn--Latvala~\cite[Lemma~5.7]{BBLat-DP}.
This is essential also when $E$ is compact.

However, if the potential $\pot{\Om}{E}$ is continuous, as will be
the case in our investigations of singular and Green functions,
these fine potential theoretic tools can be replaced by metric 
continuity and 
the
standard pasting 
Lemma~\ref{lem-pasting}.
This somewhat simplifies the proof.
\end{remark}

\begin{proof}[Proof of Lemma~\ref{lem-level-both}]
Let $\psi \in \PsiQ{\Om}{E}$ and $0 \le a < 1$. 
Then $\psi \ge \pot{\Om}{E} >  a$ in the superlevel set $\Om_a$ and thus 
$(\psi-a)/(1-a) \in \PsiQ{\Om_a}{E \cap \Om_a}$.
Taking the infimum over all such $\psi$ shows that
\[
    \pot{\Om_a}{E \cap \Om_a} \le  \frac{\pot{\Om}{E}-a}{1-a}
   \quad \text{in } \Om_a.
\]

Conversely, let
\[
   v=\begin{cases}
      a+ (1-a)\pot{\Om_a}{E \cap \Om_a} & \text{in } \Om_a, \\
      \pot{\Om}{E} & \text{in } \Om \setm \Om_a. 
   \end{cases}
\]
By the first part of the proof, $v \le \pot{\Om}{E}$.
Since $\pot{\Om}{E} \le a$ in $\Om \setm \Om_a$, we see that $v$ is lower
semicontinuous.
Hence, by the pasting Lemma~\ref{lem-pasting},
$v$ is superharmonic in $\Om$.
By definition, $v \ge 1$ q.e.\ in $E$, 
and thus $v \in \PsiQ{\Om}{E}$, 
i.e.\ $v \ge \pot{\Om}{E}$.
This proves the first identity in the statement.

For the second identity, consider $0<a<1$. First
note that the inequality
\begin{equation} \label{eq-pot-2-2}
\pot{\Om}{\Om^a} \le 
 \min\biggl\{1,\frac{\pot{\Om}{E}}{a}\biggr\} \quad \text{in } \Om
\end{equation}
follows immediately from the definition \eqref{def-pot-unbdd}
of $\pot{\Om}{\Om^a}$ since the right-hand side belongs to  $\PsiQ{\Om}{\Om^a}$.

Conversely, $\pot{\Om}{E}$ is superharmonic
in $\Om$ and thus finely continuous, by Bj\"orn~\cite[Theorem~4.4]{JBfine} 
or Korte~\cite[Theorem~4.3]{RiikkaFine} (or \cite[Theorem~11.38]{BBbook}).
Thus, the set $\Om\setm\Om^a$ is finely open.
Let
\[
   v=\begin{cases}
      \pot{\Om}{E} & \text{in } \Om^a, \\
      a \pot{\Om}{\Om^a} = \min\{ a \pot{\Om}{\Om^a},  \pot{\Om}{E}\}  
                   & \text{in } \Om \setm \Om^a, 
   \end{cases}
\]
and note that $v=a \pot{\Om}{\Om^a} + (\pot{\Om}{E} -a)_\limplus$ q.e.\ in $\Om$.
Hence $v\in \Np\loc(\Om)$.
It then follows from the pasting lemma 
for fine superminimizers in 
Bj\"orn--Bj\"orn--Latvala~\cite[Lemma~5.7]{BBLat-DP} that $v$
is a fine superminimizer in~$\Om$.
Since $\Om$ is open, $v$ is a standard superminimizer in $\Om$ 
by Corollary~5.6 in~\cite{BBLat-DP}.
Its lsc-regularization $v^*$, defined by~\eqref{eq-def-lsc-reg} 
and q.e.\ equal to $v$, is then superharmonic in~$\Om$
and admissible for the definition of $\pot{\Om}{E}$.
In particular, this gives
\[
a \pot{\Om}{\Om^a} = v^* \ge \pot{\Om}{E} \quad \text{q.e.\ in } \Om \setm \Om^a.
\]
Hence
\[
\pot{\Om}{\Om^a} \ge \min\biggl\{1,\frac{\pot{\Om}{E}}{a}\biggr\} 
\quad \text{q.e.\ in } \Om,
\]
and since both functions are superharmonic in $\Om$ (and thus
lsc-regularized), it follows that the inequality holds everywhere.
Together with~\eqref{eq-pot-2-2} this shows that
\begin{equation*} 
\pot{\Om}{\Om^a} = 
 \min\biggl\{1,\frac{\pot{\Om}{E}}{a}\biggr\} \quad \text{in } \Om.
\end{equation*}

Finally, we see that
\[
\pot{\Om}{\Om^a} \ge \pot{\Om}{\Om_a}
    \ge   \lim_{\eps \to0\limplus}   \pot{\Om}{\Om^{a+\eps}}
   =  \lim_{\eps \to0\limplus}  \min\biggl\{1,\frac{\pot{\Om}{E}}{a+\eps}\biggr\}
  =  \pot{\Om}{\Om^a}. 
\qedhere
\]
\end{proof}

\begin{thm} \label{thm-cp-level-pot}
The following hold for the superlevel sets 
in~\eqref{eq-def-level-sets}\/\textup:
\begin{enumerate}
\item \label{a1}
$\cp(E \cap \Om_a,\Om_a) = (1-a)^{1-p} \cp(E,\Om)$ if $0 \le a <1$.
\item \label{a2a}
$\cp(\Om_a,\Om)  = a^{1-p} \cp(E,\Om)$ if $0<a<1$.
\item \label{a2b}
$\cp(\Om^a,\Om) = a^{1-p} \cp(E,\Om)$ if $0<a \le 1$.
\end{enumerate}
\end{thm}

For bounded $\Om$, the capacitary identities in Theorem~\ref{thm-cp-level-pot} 
and Corollary~\ref{cor-cp-level-pot-ab} were obtained in
Aikawa--Bj\"orn--Bj\"orn--Shan\-mu\-ga\-lin\-gam~\cite[Proposition~4.1]{ABBSdich}
and Bj\"orn--Bj\"orn--Lehrb\"ack~\cite[Theorem~3.3 and Lemma~3.4]{BBLehGreen}.
See Holo\-pai\-nen~\cite[Lemma~3.8]{Ho} and 
Hei\-no\-nen--Kil\-pe\-l\"ai\-nen--Martio~\cite[p.~118]{HeKiMa} for earlier results
for bounded subsets of manifolds and weighted $\R^n$.
In~\cite{BBLehGreen},
these identities were subsequently used to study Green functions on bounded domains
in metric spaces.

\begin{proof}
First, note that \ref{a2b} follows from \ref{a2a}: Indeed, \ref{a2a} and
\begin{equation*} 
\cp(\Om^a,\Om) 
   \le \lim_{\eps \to0\limplus}   \cp(\Om_{a-\eps},\Om),
\quad 0<a\le 1, 
\end{equation*}
yield the $\le$-inequality in \ref{a2b}. 
The $\ge$-inequality follows from \ref{a2a} and 
$\Om_a\subset\Om^a$ when $a<1$, while 
Proposition~\ref{prop-cp} and $\Cp(E \setm \Om^1)=0$ give
$\cp(E,\Om) \le \cp(\Om^1,\Om)$ for $a=1$.

Next, if $\cp(E,\Om)=\infty$ then the left-hand sides in 
\ref{a1}--\ref{a2a} are clearly bounded from below by 
\[
\cp(E \cap \Om_a,\Om) =\infty,
\]
since $\Cp(E \setm \Om_a)=0$ when $0\le a<1$.
So assume that $\cp(E,\Om)<\infty$ and  $0\le a<1$.

We begin by proving the $\le$-inequalities.
For $j,k=2,3,\dots$\,, let 
\[
E_k=E\cap B_k \quad \text{and} \quad G_j=\Om\cap B_{j+1}.
\]

Fix $k$ for the moment. Since $E_k$ is bounded and 
$\cp(E_k,\Om)\le \cp(E,\Om)<\infty$,
there is an admissible function for $\cp(E_k,\Om)$.
Multiplying it by a suitable cutoff function shows that 
\begin{equation}   \label{eq-cp-Ek-Omj<infty}
\cp(E_k,G_j)<\infty \quad \text{for }j \ge k.
\end{equation}
Let 
\[
G^k_{j,a}=\{x \in G_j : \pot{G_j}{E_k} >a\}
\quad \text{for } j =k,k+1,\dots.
\]
By Lemma~\ref{lem-pot-conv}, $\pot{G_j}{E_k} \nearrow \pot{\Om}{E_k}$
as $j \to \infty$,
while $\pot{\Om}{E_k} \nearrow \pot{\Om}{E}$ and
$\pot{G_k}{E_k} \nearrow \pot{\Om}{E}$ as $k \to \infty$. Hence
\begin{equation}   \label{eq-Om-k,a=cup}
   \Om^k_a:= \{x \in \Om : \pot{\Om}{E_k} >a\} = \bigcup_{j=k}^\infty G^k_{j,a}
 \quad \text{and} \quad 
  \Om_a = \bigcup_{k=1}^\infty \Om^k_a = \bigcup_{k=1}^\infty G^k_{k,a}.
\end{equation}
As $\Cp(E_k \setm G^k_{j,a})=0$ 
for all $j\ge k$, it follows from 
Proposition~\ref{prop-cp-inc-gen} and 
Bj\"orn--Bj\"orn--Lehrb\"ack~\cite[Lemma~3.4]{BBLehGreen} that
\begin{align}
\cp(E_k \cap G^k_{k,a},\Om^k_a)
 & = \lim_{j \to \infty} \cp(E_k \cap G^k_{k,a},G^k_{j,a}) 
  = \lim_{j \to \infty} \cp(E_k \cap G^k_{j,a},G^k_{j,a}) 
  \nonumber \\
  & = \lim_{j \to \infty} (1-a)^{1-p}\cp(E_k,G_j) 
   = (1-a)^{1-p}\cp(E_k,\Om), \label{eq-cap-Om-a}
\end{align}
where all the involved capacities are finite because 
of~\eqref{eq-cp-Ek-Omj<infty} and the monotonicity of capacity.
Now, letting $k\to\infty$ and using \eqref{eq-Om-k,a=cup},
\eqref{eq-cap-Om-a} and Proposition~\ref{prop-cp-inc-liminf}
yields
\begin{align*}
  \cp(E \cap \Om_a,\Om_a)
 &    \le  \liminf_{k\to\infty} \cp(E_k \cap G^k_{k,a},\Om^k_a)  \nonumber \\
  &=  \lim_{k\to\infty} (1-a)^{1-p}\cp(E_k,\Om)
  = (1-a)^{1-p}\cp(E,\Om). 
\end{align*}
This shows the $\le$-inequality in \ref{a1}.

In \ref{a2a},  we proceed similarly, assuming that $0<a<1$. 
For each fixed $k$ and $j\ge k$, 
Lemma~3.4 in~\cite{BBLehGreen}
implies that 
\[ 
\cp(G^k_{j,a},\Om)
  \le    \cp(G^k_{j,a},G_j)
  = a^{1-p} \cp(E_k,G_j).
\]
Together with  Theorem~\ref{thm-cp}\ref{cp-Choq-E},
Lemma~\ref{lem-cp-inc-1} (or Proposition~\ref{prop-cp-inc-gen})
and \eqref{eq-Om-k,a=cup},  this yields
\[ 
  \cp(\Om^k_a,\Om) 
 =  \lim_{j \to \infty} \cp(G^k_{j,a},\Om) 
 \le  \lim_{j \to \infty} a^{1-p} \cp(E_k,G_j) =
a^{1-p} \cp(E_k,\Om)
\] 
and
\[ 
\cp(\Om_a,\Om)
=\lim_{k \to \infty}   \cp(\Om^k_a,\Om) 
 \le \lim_{k \to \infty} a^{1-p} \cp(E_k,\Om)
= \cp(E,\Om),
\] 
which
proves the $\le$-inequality in \ref{a2a}.

To conclude the proof of \ref{a1}--\ref{a2a}, we use 
Proposition~\ref{prop-ex-pot-unbdd}
and Lemma~\ref{lem-level-both}, together with the already proved
$\le$-inequalities, as follows:
(When $a=0$, we also use that 
$\pot{\Om}{E}\equiv0$ in $\Om \setm \Om_a$ to obtain the third equality.)
\begin{align}  \label{eq-le-to-ge}
 \cp(E,\Om)
  & = \int_{X} g_{\pot{\Om}{E}}^p \, d\mu \nonumber\\
  & = \int_{\Om_a} g_{\pot{\Om}{E}}^p \, d\mu 
   + \int_{\Om \setm \Om_a} g_{\pot{\Om}{E}}^p \, d\mu \\
  & = (1-a)^p\int_{\Om_a} g_{\pot{\Om_a}{E \cap \Om_a}}^p \, d\mu 
    + a^p    \int_{\Om \setm \Om_a} g_{\pot{\Om}{\Om_a}}^p \, d\mu \nonumber\\
   & = (1-a)^p\cp(E \cap \Om_a,\Om_a)
    + a^p    \cp(\Om_a,\Om) \nonumber\\
  & \le  (1-a)\cp(E,\Om) + a \cp(E,\Om)  \nonumber\\
  & = \cp(E,\Om).\nonumber
\end{align}
This shows that we must have equality throughout.
\end{proof}

\begin{cor}  \label{cor-cp-level-pot-ab}
Let $0 \le a < b \le 1$.
Then 
\begin{equation} \label{eq-pot-b-a-1}
\cp(\Om^b,\Om_a)  = (b-a)^{1-p} \cp(E,\Om)
\end{equation}
and 
\begin{equation} \label{eq-pot-b-a}
     \pot{\Om_a}{\Om^b} 
   = \min\biggl\{1,\frac{\pot{\Om}{E}-a}{b-a}\biggr\}_\limplus.
\end{equation}
If moreover $b<1$ then also
\begin{equation} \label{eq-pot-b-a-3}
\cp(\Om_b,\Om_a) = \cp(\Om^b,\Om_a).
\end{equation}
\end{cor}

\begin{proof}
If
$\cp(E,\Om)=\infty$, then  \eqref{eq-pot-b-a-1} is obvious by monotonicity of 
the capacity.
So assume that $\cp(E,\Om)< \infty$.
Note that by Lemma~\ref{lem-level-both},
\[
\Om^{b}= \{x \in \Om_a : \pot{\Om_a}{E \cap \Om_a} \ge \be\} =: (\Om_a)^{\be},
\quad \text{where }
\be=\frac{b-a}{1-a} >0.
\]
Hence, Theorem~\ref{thm-cp-level-pot} 
applied twice (first \ref{a2b} and then \ref{a1})
yields
\begin{align*}
\cp(\Om^b,\Om_a)
  &= \cp((\Om_a)^{\be},\Om_a) =
 \be^{1-p} \cp(E \cap \Om_a,\Om_a) \\
  &= 
 \be^{1-p} (1-a)^{1-p}\cp(E,\Om)
  = (b-a)^{1-p}\cp(E,\Om).
\end{align*} 
The equality~\eqref{eq-pot-b-a-3}
is also a consequence of Theorem~\ref{thm-cp-level-pot}.

Finally, \eqref{eq-pot-b-a} follows directly from
Lemma~\ref{lem-level-both}.  
\end{proof}

\section{Singular functions}
\label{sect-singular}

The following definition summarizes the requirements we 
impose on singular functions. 
It is slightly different from the definition in the introduction,
but Proposition~\ref{prop-sing-eqv} shows that 
they are equivalent.
Recall that $B_r = B(x_0,r)$ and $\clB_r=\itoverline{B(x_0,r)}$.

\begin{deff} \label{deff-sing}
Let $\Om \subset X$ be a domain, i.e.\ a connected
open set.
A function $u:X^* \to [0,\infty]$ 
is a \emph{singular function} in $\Om$
with singularity at $x_0 \in \Om$ if
$u>0$ in $\Om$,  $u=0$ outside $\Om$, and
it satisfies the following
properties\/\textup{:} 

\begin{enumerate}
\renewcommand{\theenumi}{\textup{(S\arabic{enumi})}}%
\item \label{dd-s}
$u$ is superharmonic in $\Om$,
\item \label{dd-h}
$u$ is \p-harmonic in $\Om \setm \{x_0\}$,
\item \label{dd-sup}
$u(x_0)=\sup_\Om u$,
\item \label{dd-inf}
$\inf_\Om u = 0$,
\setcounter{saveenumi}{\value{enumi}}
\item \label{dd-P}
$  u=P_{\Om \setm \clB_r} u$ in $\Om \setm \clB_r$
whenever $B_r \Subset \Om$.
\end{enumerate}
\end{deff}

The essence of \ref{dd-P} is the inequality $u \le \lP_{\Om \setm \clB} u$,
while $u \ge \uP_{\Om \setm \clB} u$ clearly follows from the other 
properties of $u$.
It is convenient to have $u=0$ in $X^* \setm \Om$, even though
it might be more correct to call only $u|_\Om$ the singular function.

\begin{prop}  \label{prop-sing-lika}
Let $\Om$ be a bounded domain and $u:X^* \to [0,\infty]$ 
be such that  $u=0$ outside $\Om$.
  Then $u$ is a singular function in $\Om$ with singularity at $x_0$
  according to Definition~\ref{deff-sing} if and only if $u|_\Om$ is 
a singular function according to Definition\/~\textup{1.1} in 
Bj\"orn--Bj\"orn--Lehrb\"ack\/~\textup{\cite{BBLehGreen}}.
\end{prop}  

In particular, if $\Om$ is bounded then many of the
results  in this and later sections were already obtained
in~\cite{BBLehGreen}.

\begin{proof}
Apart from taking restriction to $\Om$, 
it is only \ref{dd-P} that 
differs from~\cite[Definition~1.1]{BBLehGreen}.
Since 
$u=0$ outside $\Om$, and $\Om$ is bounded,
\ref{dd-P} therein becomes
\begin{enumerate}
\renewcommand{\theenumi}{\textup{(S\arabic{enumi}$'$)}}%
\setcounter{enumi}{\value{saveenumi}}
\item \label{dd-Np0}
$u \in \Nploc(X \setm \{x_0\})$.
\end{enumerate}

Assume
that $u$ is singular according to Definition~\ref{deff-sing}.
Let $B_r \Subset \Om$ 
and let $\eta$ be a Lipschitz function
such that 
$\chione_{B_r \setm B_{r/2}} \le \eta \le 1$ and 
$\supp \eta \Subset \Om \setm B_{r/4}$.
Then $u\eta \in \Np(X)$
and by \ref{dd-P},
\[
u = P_{\Om \setm \clB_r} u = P_{\Om \setm \clB_r} (u \eta)
\quad \text{in } \Om \setm \clB_r.
\]
Since Perron and Sobolev solutions of the Dirichlet problem
(on the bounded set $\Om \setm \clB_r$)
coincide for Newtonian boundary data, by
Theorem~5.1 
in Bj\"orn--Bj\"orn--Shan\-mu\-ga\-lin\-gam~\cite{BBS2}
(or \cite[Theorem~10.12]{BBbook}),
it follows that $u-u\eta\in\Np_0(\Om\setm\clB_r)$ and so
$u \in \Np(X \setm \clB_r)$.
Letting $r \to 0$ shows that \ref{dd-Np0} holds.

Conversely, if $u$ is singular according 
to~\cite[Definition~1.1]{BBLehGreen}, let
$B_r \Subset \Om$.
Then, by 
Theorem~5.1 
in~\cite{BBS2}
(or \cite[Theorem~10.12]{BBbook})
again,
$u=P_{\Om \setm \clB_r} u$ in $\Om \setm \clB_r$.
Hence \ref{dd-P} holds.
\end{proof}

\begin{lem} \label{lem-parabolic-superh-X}
  Assume that $\Cp(X \setm \Om)=0$ and that $X$ 
is either \p-parabolic or bounded.
If  $u$ is a superharmonic function in $\Om$, which is bounded from below,
then $u$ is constant.

In particular, in such $\Om$, there are no singular functions.
\end{lem}  

That there may exist superharmonic functions which are not bounded from below
in \p-parabolic spaces is easily seen by looking at unweighted $\R^2$ with $p=2$:
take $u(x)=-\log|x|$ or $u(x)=\min\{-\log|x|,0\}$ (for an example which is bounded
from above).

\begin{proof}[Proof of Lemma~\ref{lem-parabolic-superh-X}]
By Theorem~6.3 in Bj\"orn~\cite{ABremove} (or \cite[Theorem~12.3]{BBbook}),
$u$ has an extension as a superharmonic function on $X$ (also called $u$).
As $u$ is lsc-regularized, $\inf_X u = \inf_\Om u > -\infty$.

We may assume that $\inf_X u=0$.
Let $B$ be a ball.
By the lower semicontinuity, $m:=\min_{\clB} u$ exists.
Then $u \ge m \pot{X}{\clB}$  by definition~\eqref{def-pot-unbdd} 
of the capacitary potential.
Since $X$ is \p-parabolic or bounded, we have $\cp(\clB,X)=0$, by
Proposition~\ref{prop-hyp-char-1}, and hence
$u \ge m \pot{X}{\clB} \equiv m$, by Lemma~\ref{lem-cp=0}.
As $\inf_X u=0$, we must have $m=0$, and thus
it follows from the strong minimum principle 
(\cite[Theorem~9.13]{BBbook}) that $u$ is constant.
\end{proof}

We will need the following special case of
Proposition~4.4 in~\cite{BBLehGreen}.

\begin{lem} \label{lem-punctured-ball}
Let $u \ge 0$ be a function which is superharmonic in $\Om$  and
\p-harmonic in $\Om\setm \{x_0\}$, where $x_0 \in \Om$.
Then the limit
$a:=\lim_{x \to x_0} u(x)$ exists\/ {\rm(}possibly infinite\/{\rm)}
and $u(x_0)=a$. 

In particular, if $u$ is a singular function in a domain $\Om$,
then $u$ is $\eR$-continuous in $\Om$.
\end{lem}  

The following simple observation is worth keeping in mind.

\begin{lem} \label{lem-sing-x_0}
Let $u$ be a singular function in a domain $\Om$ with singularity at $x_0$,
Then $u(x_0)=\infty$ if and only if $\Cp(\{x_0\})=0$.
\end{lem}

\begin{proof}
If $\Cp(\{x_0\})>0$, then 
Proposition~2.2
in Kinnunen--Shan\-mu\-ga\-lin\-gam~\cite{KiSh06}
(or \cite[Corollary~9.51]{BBbook})
implies that $u(x_0)<\infty$.

When $\Cp(\{x_0\})=0$, assume for a contradiction that
$u(x_0)<\infty$. 
Then by \ref{dd-sup},
$u$ is  bounded in $\Om$.
Hence by Theorem~6.2 in Bj\"orn~\cite{ABremove} (or \cite[Theorem~12.2]{BBbook}),
$u$ is \p-harmonic in $\Om$.
But then \ref{dd-sup} and \ref{dd-inf} contradict the strong maximum principle.
Therefore $u(x_0)=\infty$.
\end{proof}

The following result shows that there are some
  redundancies in the definition of singular functions.
In particular, it shows that Definition~\ref{deff-sing}
is equivalent to the definition in the introduction.

\begin{prop} \label{prop-sing-eqv}
  Let $\Om$ be a domain
  and let $u:X^* \to [0,\infty]$ 
be a function such that 
$u>0$ in $\Om$ and  $u=0$ outside $\Om$.
Then the following are equivalent\/\textup{:}
\begin{enumerate}
\item \label{s1}
  $u$ is singular in $\Om$ with singularity at $x_0$,
\item \label{s2}
  \ref{dd-s}, \ref{dd-inf} and \ref{dd-P} hold,
\item  \label{s3}
  \ref{dd-s} and \ref{dd-P} hold, and $u$ is nonconstant in $\Om$,
\item   \label{s4}
  \ref{dd-s} and \ref{dd-P} hold, and $\Cp(X \setm \Om)>0$ or $X$ is \p-hyperbolic.
\end{enumerate}  
\end{prop}  

In order to prove this we will need the following lemma,
which in particular applies to complements of balls.
The role of $K$ is to patch ``holes'' in $\Om$.

\begin{lem} \label{lem-hyperbolic-lim-infty-simple}
Assume that $\Om$ is unbounded and 
that there is a compact $K \subset X$ such that
$\Cp(X \setm (\Om \cup K))=0$.
Then either $\lP_{\Om} \chione_{\bdy \Om}\equiv1$ or
\begin{equation}     \label{eq-lim-infty=0}
    \liminf_{\Om \ni x \to \infty} \lP_{\Om} \chione_{\bdy \Om} (x)=0.
\end{equation}
In particular, if $\Cp(X\setm\Om)=0$ or 
$X$ is  \p-hyperbolic, then \eqref{eq-lim-infty=0} holds.
\end{lem}

The proof below uses
Lemma~4.5 from \cite{BBbook}.
Its proof only uses
the \p-Poincar\'e inequality on each ball separately, without
assuming any
uniformity in the constants, and is thus available here.

\begin{proof}
Proposition~\ref{prop-ex-pot-unbdd} (with $\Om$ and $E$
replaced by $X$ and $X\setm\Om$, respectively) shows that 
\begin{equation*} 
u:= \pot{X}{X\setm\Om} = \lP_{\Om} \chione_{\bdy \Om} = P_{\Om} \chione_{\bdy \Om\setm I}
\quad  \text{in }\Om,
\end{equation*}
where  $I\subset \bdystar \Om$ is the set of irregular boundary points for $\Om$.
Let
\begin{equation*} 
 c=\liminf_{\Om \ni x \to \infty} u(x).
\end{equation*}
If $c=0$, then \eqref{eq-lim-infty=0} holds.
If instead $c>0$, then 
\[
u-c  \ge (1-c) P_{\Om} \chione_{\bdy\Om\setm I}  = (1-c)u \quad
\text{in } \Om.
\]
Hence $u\ge1$ in $\Om$, which together with the trivial 
inequality $u\le1$ proves the first part of the lemma.

If $\Cp(X \setm \Om)=0$, then $u \equiv 0$ by definition,
and so \eqref{eq-lim-infty=0} holds.

Finally, assume that $X$ is \p-hyperbolic and $\Cp(X\setm \Om)>0$.
Then by Propositions~\ref{prop-cp} and~\ref{prop-hyp-char-1},
\[
0<\cp(X\setm\Om,X) 
\le \cp(X \setm (\Om \cup K),X) + \cp(K,X)
= \cp(K,X) <\infty.
\]
Proposition~\ref{prop-ex-pot-unbdd} 
then implies that 
$\int_X g_u^p\,d\mu>0$, which 
excludes the possibility 
that  $u\equiv1$ in $\Om$,  and so $c=0$.
\end{proof}

\begin{proof}[Proof of Proposition~\ref{prop-sing-eqv}]
\ref{s1}\imp\ref{s2}
This is trivial.

\ref{s2}\imp\ref{s1}
Note first that 
\ref{dd-h} follows directly by letting $r\to0$ in \ref{dd-P}
 and using Proposition~9.21 in~\cite{BBbook}.
Let $M_r=\max_{\bdy B_r} u$ for $0 < r< \dist(x_0,X\setm \Om)$.
By \ref{dd-P},  $u\le M_r$ in $\Om\setm \clB_r$.
Thus, by Lemma~\ref{lem-punctured-ball}, $u(x_0)=\lim_{r \to 0} M_r = \sup u$,
i.e.\ \ref{dd-sup} holds.

\ref{s2}\imp\ref{s3}
As $u(x_0)>0$ and \ref{dd-inf} holds, $u$ must be nonconstant
in $\Om$.

\ref{s3}\imp\ref{s4}
This follows directly from Lemma~\ref{lem-parabolic-superh-X}.

\ref{s4}\imp\ref{s2}
If $\Cp(X \setm \Om)>0$, then \ref{dd-inf} follows from
\ref{dd-P} and the Kellogg property (Theorem~\ref{thm-Kellogg}).
On the other hand, if $X$ is \p-hyperbolic 
and $\Cp(X \setm \Om)=0$, then \ref{dd-inf} follows from
\ref{dd-P} and Lemma~\ref{lem-hyperbolic-lim-infty-simple}
(applied with $K=\clB_r \Subset \Om$
 and $\Om \setm \clB_r$ instead of  $\Om$),
since $u\le \chione_{\bdy(\Om\setm\clB_r)}\sup_{\bdy B_r} u$ 
on $\bdy(\Om \setm \clB_r)$.
\end{proof}

We are now ready to show the existence of singular functions.

\begin{thm} \label{thm-exist-singular}
Let $\Om$ be an   unbounded domain.
  Assume that $X$ is \p-hyperbolic or that  $\Cp(X \setm \Om)>0$.
Then there is a singular function in $\Om$ with singularity at~$x_0$.
\end{thm}

\begin{proof}
Let $\Om_1 \subset \Om_2 \subset \dots \subset \Om=\bigcup_{j=1}^\infty \Om_j$
be an increasing sequence of bounded domains such that $x_0 \in \Om_1$.
(One may e.g.\ let $\Om_j$ be the component of $\Om \cap B_j$ containing
$x_0$.
By Lemma~4.10 in
Bj\"orn--Bj\"orn~\cite{BBsemilocal} 
this exhausts $\Om$,
i.e.\ $\Om=\bigcup_{j=1}^\infty \Om_j$.)

Let $u_j$, extended by 0 to $X^*\setm\Om_j$,
be a Green function in $\Om_j$ with singularity at $x_0$
according to Definition~1.2 in 
Bj\"orn--Bj\"orn--Lehrb\"ack~\cite{BBLehGreen},
which exists by Theorem~1.3 in~\cite{BBLehGreen}.
Let $r_0>0$ be so small that $2B_{r_0} \subset \Om_1$
and let $\la_0$ be the dilation for the \p-Poincar\'e inequality
within $B_{r_0}$.
Then let $0 <r \le  r_0/51 \la_0$.
By Theorem~7.1 in Bj\"orn--Bj\"orn--Lehrb\"ack~\cite{BBLehIntGreen},
\begin{equation} \label{eq-uj-simeq-cp-closed}
u_j(x) \simeq \cp(B_{r},\Om_j)^{1/(1-p)}
\quad \text{if }x \in S_{r}:=\{x : d(x,x_0)=r\},
\end{equation}
with comparison constants depending only on the doubling
and Poincar\'e constants within $B_{r_0}$,
but independent of $j$ and $r$, see~\cite[Section~11]{BBLehIntGreen}.

Now $\{\cp(B_{r},\Om_j)\}_{j=1}^\infty$ is a decreasing sequence
tending to  $\cp(B_{r},\Om)>0$, by Proposition~\ref{prop-cp-inc-gen}
  and Corollary~\ref{cor-cpE-Om>0}.
Hence, by~\eqref{eq-uj-simeq-cp-closed}, all $u_j$ are uniformly bounded on $S_r$.
By Harnack's convergence theorem,
see Shan\-mu\-ga\-lin\-gam~\cite[Proposition~5.1]{Sh-conv}.
(or  \cite[Theorems~9.36 and~9.37]{BBbook}),
there is a subsequence (also denoted $\{u_j\}_{j=1}^\infty$)
such that $\{u_j\}_{j=1}^\infty$ converges
locally uniformly in $\Om \setm \{x_0\}$ to a \p-harmonic
function $u$ in $\Om \setm \{x_0\}$.
By letting $j\to\infty$ in~\eqref{eq-uj-simeq-cp-closed},
\[
u(x) \simeq \cp(B_r, \Om)^{1/(1-p)}
\quad \text{when } 0 < d(x,x_0)=r \le  \frac{r_0}{51 \la_0},
\] 
with the same comparison constants as in \eqref{eq-uj-simeq-cp-closed}.
Let $u(x_0)=\liminf_{j \to \infty} u_j(x_0)$
  and  $u \equiv 0$ on $X^* \setm \Om$.

Since each $u_j$ is superharmonic in $\Om_j$, 
Theorem~6.2 in Bj\"orn--Bj\"orn--Par\-vi\-ai\-nen~\cite{BBParv}
(or \cite[Theorem~9.31]{BBbook})
together with
Proposition~9.21 in~\cite{BBbook} 
implies
that the lsc-regularization $u_*$ 
of $u$ is superharmonic in $\Om$.
By \p-harmonicity, $u=u_*$ in $\Om\setm\{x_0\}$.
If $\cp(B_r, \Om)\to0$, as $r\to0$, then 
\[
u_*(x_0)=\infty=\lim_{j \to \infty} u_j(x_0)=u(x_0)
\]
and $\Cp(\{x_0\})=0$.
Thus $u$ is superharmonic in $\Om$, i.e.\ \ref{dd-s} holds.

On the other hand, if $\cp(B_r, \Om)\ge c>0$ for all $0<r\le r_0$, then
by \eqref{eq-uj-simeq-cp-closed}, all $u_j$ as well as $u$ are uniformly
bounded and thus $\Cp(\{x_0\})>0$ and $u(x_0)<\infty$.
In this case the Green function $u_j$ in $\Om_j$ is unique.
By Theorem~9.3 in~\cite{BBLehGreen}
and  Proposition~\ref{prop-cp-inc-gen}, the limit
\[
u(x_0)=\lim_{j \to \infty} u_j(x_0)
= \lim_{j \to \infty}\cp(\{x_0\},\Om_j)^{1/(1-p)}
= \cp(\{x_0\},\Om)^{1/(1-p)}
\]
exists and is finite.
Also in this case, $u=u^*$ is superharmonic in $\Om$ and \ref{dd-s} holds.

Next, we turn to \ref{dd-P}.
Recall that 
$u_j \equiv 0$ in $X^* \setm \Om_j$. 
Let $B=B_r  \Subset \Om$ and $\eps>0$. 
By the locally uniform convergence, there is $j_0$ such that $B\Subset\Om_{j_0}$ and
\[
|u_j(x)-u(x)| < \eps \quad \text{for
every $x \in S_r$ and every $j \ge j_0$.}
\]
For such $u_j$, \eqref{eq-lP-uP},  the definition of Perron
solutions and \ref{dd-P} (from Proposition~\ref{prop-sing-lika})
show that
\[
u \ge \uP_{\Om \setm \clB} u
\ge \lP_{\Om \setm \clB} u
\ge  \lP_{\Om_j \setm \clB} (u_j-\eps) = u_j-\eps
\quad \text{in } \Om_j \setm \clB.
\]
Letting $j \to \infty$ and then $\eps \to 0$ shows that 
\[
u \ge \uP_{\Om \setm \clB} u 
\ge \lP_{\Om \setm \clB} u \ge  u
\quad \text{in } \Om \setm \clB.
\]
Hence \ref{dd-P} holds for $u$, and thus 
$u$ is a singular function in $\Om$, by Proposition~\ref{prop-sing-eqv}.
\end{proof}

\begin{lem} \label{lem-sing-subdomain}
Let $u$ be a singular function in a domain $\Om$ with singularity at $x_0$,
and let $0 < a< u(x_0)$.
Then 
$\Om_a:=\{x : u(x) > a\}$ is a domain and
$v=(u-a)_\limplus$ is a singular function in 
$\Om_a$
with singularity at $x_0$.
\end{lem}  

\begin{proof}
As $u$ is $\eR$-valued continuous in $\Om$, by Lemma~\ref{lem-punctured-ball},
we see that $\Om_a$ is open.
Let $U$ be the component of $\Om_a$ containing $x_0$ and let $B_r\Subset U$.
Then $u=P_{\Om \setm \clU} u$ 
in the open set $\Om \setm \clU$,
by \ref{dd-P} and Lemma~\ref{lem-subset-Perron} applied with
$G=\Om\setm \clU$.
Since $u=0$ outside $\Om$ and $u \le a$ on $\Om \cap \bdy (\Om \setm \clU)$,
we conclude that $u=P_{\Om \setm \clU} u \le a$ in $\Om \setm \clU$.
It follows that $\Om_a=U$ and it is thus connected.
Clearly, $v>0$ in $\Om_a$ and
the conditions \ref{dd-s}--\ref{dd-inf} hold.

Finally, let $B_r \Subset \Om_a$ be arbitrary.
By \ref{dd-P} for $u$ and the last part of Lemma~\ref{lem-subset-Perron}
with $G=\Om_a \setm \clB_r$,
\[
v= u-a = P_{\Om_a \setm \clB_r} u-a = P_{\Om_a \setm \clB_r} v
\quad \text{in } \Om_a \setm \clB_r,
\]
i.e.\ $v$ also satisfies \ref{dd-P}.  
\end{proof}

We shall now see that when $\Cp(\{x_0\})>0$, singular functions
are just multiples of the capacitary potential for $\{x_0\}$.
When $\Cp(\{x_0\})=0$, we will only prove a comparison between singular 
functions in Proposition~\ref{prop-comparable} below.
Note that in that case, $u(x_0)=\infty$ and hence 
$P_{\Om \setm \{x_0\}} u\equiv0$, since $au\in \UU_u(\Om \setm \{x_0\})$
for every $a>0$.

\begin{prop} \label{prop-Cp>0-unique}
Assume that $\Cp(\{x_0\})>0$.
Let $u$ be a singular function in a domain $\Om$ with singularity at $x_0$.
Then 
\[
u=P_{\Om \setm \{x_0\}} u \quad \text{in } \Om \setm \{x_0\} 
\qquad \text{and} \qquad 
u=u(x_0)\pot{\Om}{\{x_0\}} \quad \text{in } X.
\]
\end{prop}  

\begin{proof}
Lemma~\ref{lem-punctured-ball}  shows that
\[
u(x_0)=\lim_{x\to x_0} u(x).  
\]
Hence $u \in \UU_u(\Om \setm \{x_0\})$
and thus $u \ge \uP_{\Om \setm \{x_0\}} u$.

Conversely, let $v=\lP_{\Om \setm \{x_0\}} u\ge0$.
It follows from Lemmas~\ref{lem-punctured-ball} and~\ref{lem-sing-x_0} 
that $u$ is continuous in $\Om$ and $u(x_0)<\infty$.
By the Kellogg property (Theorem~\ref{thm-Kellogg}), also
$\lim_{x \to x_0} v(x)=u(x_0)$.
Let $\eps >0$ and find a ball $B_r \Subset \Om$ such that $r<\eps$ and
$|u-v| < \eps$ in $\clB_r$.
By continuity, $v+\eps \in \UU_u(\Om \setm \clB_r)$
and thus, using also \ref{dd-P} and \eqref{eq-lP-uP},
\[
v+\eps \ge P_{\Om \setm \clB_r} u = u \ge \uP_{\Om \setm \{x_0\}} u \ge v
\quad \text{in } \Om \setm \clB_r.
\]
Letting $\eps \to 0$ (and thus $r \to 0$),
concludes the proof of the first identity.
The second identity 
now follows from~\eqref{eq-cp-E-unbdd} in Proposition~\ref{prop-ex-pot-unbdd}.
\end{proof}

\begin{prop} \label{prop-comparable}
Let $u$ and $v$ be two singular functions in a domain $\Om$ with singularity at~$x_0$.
  Then $u \simeq v$ in $\Om$ with comparison constants depending 
also on $u$ and $v$.
\end{prop}

\begin{proof}
If $\Cp(\{x_0\})>0$, then this follows directly from 
Proposition~\ref{prop-Cp>0-unique}.

Assume therefore that $\Cp(\{x_0\})=0$.
Let $B_r \Subset \Om$
and  $M=\sup_{\bdy B_r} u(x)$.
Then $u-M$ is a singular function in 
the bounded domain $\{x \in \Om : u(x)>M\} \subset B_r$, 
by Lemma~\ref{lem-sing-subdomain} and \ref{dd-P}.
Similarly, $v-M'$ is a singular function in 
the bounded domain $\{x \in \Om : v(x)>M'\} \subset B_r$,
where $M'=\sup_{\bdy B_r} v(x)$.

It now follows from (7.2) in Theorem~7.1 in Bj\"orn--Bj\"orn--Lehrb\"ack~\cite{BBLehIntGreen}
that there is a smaller ball $B_\rho \subset B_r$ such that
$u \simeq v$ in $\clB_\rho$.
As $u=P_{\Om \setm \clB_\rho} u$ and $v=P_{\Om \setm \clB_\rho} v$,
by \ref{dd-P},
we see that  $u \simeq v$ also in $\Om \setm \clB_\rho$.
\end{proof}

\section{Green functions}
\label{sect-Green}

Recall from the introduction that a
\emph{Green function} on a domain $\Om \subset X$
is a singular function which satisfies
\begin{equation} \label{eq-normalized-Green}
\cp(\Om^b,\Om) = b^{1-p}
\quad \text{when }
   0  <b < u(x_0),
\end{equation}
where $\Om^b=\{x\in \Om:u(x)\ge b\}$.
In fact it follows from  the proof of 
Theorem~\ref{thm-sing} below that it is enough
if \eqref{eq-normalized-Green} holds for one
value of $b$.

\begin{thm} \label{thm-sing}
Let $u$ be a singular function in a domain $\Om$ with singularity at $x_0$.
Then there is a constant $c$ such that
\[
\cp(\Om^b,\Om_a) = c (b-a)^{1-p} 
\quad \text{when }
   0 \le a <b \le  u(x_0),
\] 
where $\Om_a=\{x\in \Om:u(x)>a\}$,  $\Om^b=\{x\in \Om:u(x)\ge b\}$
and we interpret $\infty^{1-p}$ as $0$.
\end{thm}

As $u$ is $\eR$-continuous in $\Om$ we directly see that
$\Om_a$ is open and that $\Om^b$ is relatively closed in $\Om_a$ if $a<b$.
Note that $\Om^{u(x_0)}$ can be unbounded when $u(x_0)<\infty$
(i.e.\ when $\Cp(\{x_0\})>0$), see Example~\ref{ex-R-half-weighted}.

\begin{proof}
Assume first that $\Cp(\{x_0\})=0$.
Then $u(x_0)=\infty$, by Lemma~\ref{lem-sing-x_0}.
Let $0<m<\infty$ and $B=B_\rho\Subset \Om_m\cap B_{1/m}$.
By \ref{dd-P}, $u=P_{\Om \setm \clB} u$ in $\Om \setm \clB$,
and thus by the last part of Lemma~\ref{lem-subset-Perron}
(with $\Om$ replaced by $\Om\setm\clB$, $h=u$ and $G=\Om\setm\Om^m$)
together with Proposition~\ref{prop-ex-pot-unbdd},
\[
\frac{u}{m}= \frac{P_{\Om \setm \Om^{m}} u}{m} =  P_{\Om \setm \Om^{m}}  \chione_{\Om^m}
= \pot{\Om}{\Om^m}
\quad \text{in } \Om \setm \Om^{m}.
\]
It then follows from Corollary~\ref{cor-cp-level-pot-ab} that
for $0 \le a < b  \le m$,
\begin{align*}
\cp(\Om^b,\Om_a) 
  &=  
\cp\biggl( \biggl\{ x\in \Om : \pot{\Om}{\Om^m}(x) \ge \frac bm\biggr\},
      \biggl\{x \in \Om : \pot{\Om}{\Om^m}(x) > \frac am\biggr\} \biggr) \\
  &= \biggl(\frac{b-a}{m}\biggr)^{1-p} \cp(\Om^m,\Om) 
  =: c(b-a)^{1-p}.
\end{align*}
Letting $m \to \infty$ and thus $\rho \to 0$, 
we see that $c$ must be independent
of $m$ and that this holds for all $0 \le a < b \le \infty$.

If $\Cp(\{x_0\})>0$ we instead note that 
$\pot{\Om}{\{x_0\}}=u/u(x_0)$ by
Proposition~\ref{prop-Cp>0-unique}. 
Thus by Corollary~\ref{cor-cp-level-pot-ab}  
we get that  for $0 \le a < b  \le u(x_0)$, 
\begin{align*}
&\cp(\Om^b,\Om_a) \\
  & \qquad  
  = \cp \biggl( \biggl\{ x \in \Om : 
               \pot{\Om}{\{x_0\}}(x)
\ge \frac{b}{u(x_0)} \biggr\},
     \biggl\{x \in \Om : \pot{\Om}{\{x_0\}}(x)
> \frac{a}{u(x_0)} \biggr\} \biggr)  \\
  & \qquad 
= \biggl(\frac{b-a}{u(x_0)}\biggr)^{1-p} \cp(\{x_0\},\Om) 
  =: c(b-a)^{1-p}.\qedhere
\end{align*}
\end{proof}

We are now ready to prove Theorems~\ref{thm-main} and~\ref{thm-Green-intro}.
We start with the latter as it will be used to prove the former one.

\begin{proof}[Proof of Theorem~\ref{thm-Green-intro}]
As in \eqref{eq-def-level-sets} with $u$ replaced by $v$, let
\[
G_a=\{x\in \Om:v(x)>a\} \quad \text{and} \quad G^b=\{x\in \Om:v(x)\ge b\},
\]
while $\Om_a$ and $\Om^b$ are for $u$.
Also let $c$ be the constant provided for $v$ by Theorem~\ref{thm-sing}.
Then
\begin{equation}  \label{eq-cap-Omt}
\cp(\Om^b,\Om_a) 
=\cp(G^{b/\alp},G_{a/\alp})
=c \biggl(\frac{b-a}{\alp}\biggr)^{1-p},
\end{equation}
which equals $(b-a)^{1-p}$ if and only if $\alp=c^{1/(1-p)}$,
i.e.\ $u$ is a Green function if and only if $\alp=c^{1/(1-p)}$.

Moreover, for this value of $\alp$, 
the last part of Lemma~\ref{lem-subset-Perron}
(with $\Om$ replaced by $\Om\setm \clB_r$ for some ball 
$ B_r\subset \Om^b$, $h=u$ and $G=\Om_a\setm\Om^b$),
Proposition~\ref{prop-ex-pot-unbdd} and~\eqref{eq-cap-Omt} yield
\[ 
\int_{\Om_a \setm \Om^b} g_u^p\,d\mu
= (b-a)^p \int_{\Om_a \setm \Om^b} g^p_{\pot{\Om_a}{\Om^b}} \, d\mu 
= (b-a)^p \cp(\Om^b,\Om_a)
= b-a.
\qedhere 
\] 
\end{proof}

\begin{proof}[Proof of Theorem~\ref{thm-main}]
\ref{m-Green}\imp\ref{m-superh}
Let $u$ be the Green function, then $\min\{u,1\}$ is 
a nonconstant bounded superharmonic function in $\Om$.

$\neg$\ref{m-hyp-cp}\imp$\neg$\ref{m-superh}
This follows directly from Lemma~\ref{lem-parabolic-superh-X}.

\ref{m-hyp-cp}\imp\ref{m-sing}
If $\Om$ is unbounded, then this
follows from Theorem~\ref{thm-exist-singular}.
On the other hand, if $\Om$ is bounded then it follows from
Theorem~1.3\,(a) in Bj\"orn--Bj\"orn--Lehrb\"ack~\cite{BBLehGreen},
in view of Proposition~\ref{prop-sing-lika}.

\ref{m-sing}\imp\ref{m-Green}
This follows from Theorem~\ref{thm-Green-intro}.
\end{proof}

We also obtain the following consequences.

\begin{cor} \label{cor-Green-subdomain}
Let $u$ be a Green function in $\Om$ with singularity at $x_0$,
and let $0 < a< u(x_0)$.
Then 
$\Om_a:=\{x : u(x) > a\}$ is a domain and
$v=(u-a)_\limplus$ is a Green function in 
$\Om_a$
with singularity at $x_0$.
\end{cor}

\begin{proof}
By Lemma~\ref{lem-sing-subdomain}, $v$ is a singular function.
It now follows from Theorem~\ref{thm-Green-intro}, that $v$ satisfies
the normalization required for Green functions.
\end{proof}  

\begin{cor}
Assume that $X$ is \p-hyperbolic or $\Cp(X \setm \Om)>0$, 
and in addition that $\Cp(\{x_0\})>0$.
Then the Green function in $\Om$ with singularity at $x_0$ is unique.
\end{cor}

The existence of the Green function follows from Theorem~\ref{thm-main}.

\begin{proof}
By Proposition~\ref{prop-Cp>0-unique}, every Green function $u$ 
in $\Om$ with singularity at $x_0$
is of the form
$u=\alp\pot{\Om}{\{x_0\}}$, where $\alp>0$.
It now follows from Theorem~\ref{thm-Green-intro} that there
is only one choice for $\alp$, i.e.\ the Green function is unique.
\end{proof}

\begin{remark} \label{rmk-HoSh}
The singular function $g$ constructed  
in Holo\-pai\-nen--Shan\-mu\-ga\-lin\-gam~\cite[Theorems~3.12 and 3.14]{HoSh}
is \p-harmonic in $X\setm\{x_0\}$.
However, the capacity estimates for the superlevel sets of $g$
in \cite[pp.~325--326]{HoSh} do not seem to be fully justified:
In particular, they seem to  use that 
\[
\Capp_p(E_j,\Om_j) \to \Capp_p(E,\Om)
\quad \text{when $E_j \nearrow E$ and $\Om_j \nearrow \Om$},
\]
which can fail in general, as shown by 
Example~\ref{ex-warning-ring}.
Without suitable superlevel set or energy estimates, it is also not quite
clear why $g$ should be nonconstant when $\Cp(\{x_0\})>0$.

In the setting of \cite{HoSh}, provided that $X$ is also proper,
it follows from our Proposition~\ref{prop-cp-inc-liminf}
that $g$ 
satisfies (in their notation)
\begin{equation} \label{eq-Xb}
\Capp_p(X^b,X_a) \le p^2(b-a)^{1-p}
\quad \text{when } \mu(X^b)<\infty, 
\end{equation}
but it is less
clear how to obtain a similar lower bound.
Our argument in \eqref{eq-le-to-ge} cannot be used since it
is based on the precise $\le$-inequalities in the first part of the proof
of Theorem~\ref{thm-cp-level-pot}, which are stronger
than both \eqref{eq-Xb} and the  superlevel set normalization 
in \cite[Condition~4 on p.~325]{HoSh}.

The condenser capacity 
$\Capp_p$ defined in  \cite[Definition~2.3 and (9)]{HoSh}
coincides with $\cpt$ in our Proposition~\ref{prop-cpt},
and is thus in general neither countably subadditive nor a Choquet capacity.
In particular, the left-hand side in~\eqref{eq-Xb} is infinite
when $\mu(X^b)=\infty$.
If $\Capp_p$ is replaced by our capacity $\cp$, then 
our Green functions satisfy all the requirements 
on the singular functions in~\cite{HoSh} with $b_0=0$.

When $\Om=X$ is unbounded, Definition~3.11 in~\cite{HoSh} does not seem to 
acknowledge the ``zero boundary data at $\infty$'', 
nor is \ref{dd-inf} necessarily satisfied. 
In fact, if $u$ is a Green function in $X$
and $\Cp(\{x_0\})=0$
then for any $c>0$, the function $v(x):=u(x)+c$ satisfies 
Definition~3.11 in~\cite{HoSh}
(upon choosing a sufficiently large $b_0$).
Similarly, if $\Cp(\{x_0\})>0$ 
then  the convex combination 
\[
v(x):=\theta u(x)+(1-\theta) \cp(\{x_0\},X)^{1/(1-p)}, 
\quad \text{with any } (1-1/p)^2\le \theta<1,
\]
satisfies Definition~3.11 in~\cite{HoSh} with 
\[
   b_0\ge\frac{(1-\theta)\cp(\{x_0\},X)^{1/(1-p)}}{1-\theta/p^{2/(p-1)}}.
\]

When $X$ is only locally compact (and nonproper) as in~\cite{HoSh}, 
it is not even known whether 
the condenser capacity $\cp$ is countably subadditive.
See Section~12 in Bj\"orn--Bj\"orn--Lehrb\"ack~\cite{BBLehGreen}
for a similar discussion in the case when $\Om$ is bounded.
\end{remark}

\section{Boundary regularity at \texorpdfstring{$\infty$}{oo} and resolutivity}
\label{sect-reg}

The following result is the key to proving 
Theorem~\ref{thm-char-hyp-intro}.

\begin{lem} \label{lem-global-u-lim=0}
Assume that  $X$ is \p-hyperbolic and that $\mu$ is globally doubling
  and supports a global \p-Poincar\'e inequality.
Let $\Om$ be an unbounded open set.
Then $\infty$ is regular for $\Om$.

Moreover, if $u$ is a singular function in a domain $\Om$, then
\begin{equation}   \label{ex-lim-u=0}
   \lim_{x \to \infty} u(x) = 0.
\end{equation}
\end{lem}

Example~\ref{ex-R-half-weighted} below shows that the global assumptions
cannot be replaced by local ones.

\begin{proof}
If $\Om=X$, then $\infty$ is always regular for
$\Om$, because all boundary data on $\bdystar \Om$ are constant and
so are their Perron solutions, see Example~\ref{ex-bdystar-one-pt}.

Let $\ut$ be a singular function in $X$ with singularity at $x_0$, which
exists by Theorem~\ref{thm-exist-singular}.
By the strong minimum principle \cite[Theorem~9.13]{BBbook} and~\ref{dd-inf}, 
\begin{equation}   \label{eq-liminf-ut}
\liminf_{x \to \infty} \ut(x) = 0.
\end{equation}
Applying Proposition~4.4 from 
Bj\"orn--Bj\"orn--Lehrb\"ack~\cite{BBLehGreen}
to balls $B_R$ shows that for all $R>0$,
\[
\max_{d(x_0,x)=R} \ut \le C \min_{d(x_0,x)=R} \ut,
\]
with $C$ independent of $R$. 
Letting $R\to\infty$, together with \eqref{eq-liminf-ut},
shows that \eqref{ex-lim-u=0} holds for $\ut$.

If $\Om \ne X$, then let $x_0 \in X \setm \Om$.
The singular function $\ut$ in $X$, considered above,
is \p-harmonic in $X\setm\{x_0\}$ and $\ut>0$ in $X$.
Thus
\[
   \liminf_{\Om \ni y \to x} \ut(y) = \ut(x) >0
   \quad \text{for every } x \in \bdy \Om.
\]
Since $\ut$ also satisfies~\eqref{ex-lim-u=0},
it is a barrier at $\infty$ for $\Om$,
and  hence $\infty$ is regular for $\Om$,
by Proposition~\ref{prop-barrier=>}.

If $\Om$ is connected and
$u$ is a singular function in $\Om$, then \ref{dd-P}
and the regularity of $\infty$ for $\Om\setm\clB$
imply~\eqref{ex-lim-u=0} and conclude the proof.
\end{proof}

Before continuing we need to take a closer look
at the case when $\Cp(X \setm\Om)=0$.
See Hei\-no\-nen--Kil\-pe\-l\"ai\-nen--Martio~\cite[p.~183]{HeKiMa}
for a similar discussion on weighted $\R^n$ with global assumptions.

If $X$ is bounded and $\Om=X$, then there is no boundary
and the Dirichlet problem does not make sense.
See Example~\ref{ex-bdystar-one-pt} for the case when
$\bdystar \Om=\{a\}$ consists of exactly one point
(with $a=\infty$ if $\Om=X$ is unbounded).
The case when  $\bdystar \Om$ contains at least two points
is a bit more interesting.

\begin{prop} \label{prop-Cp=0-para}
Assume that $X$ is bounded or \p-parabolic,
that $\Cp(X\setm \Om)=0$ and that
$\bdystar \Om \ne \emptyset$.
Then
\[
 \uP_{\Om} f \equiv \sup_{\bdystar \Om} f
 \quad \text{and} \quad
 \lP_{\Om} f \equiv \inf_{\bdystar \Om} f.
\]
In particular, no nonconstant $f \in C(\bdystar \Om)$ is resolutive.
Moreover, all boundary points are irregular if $\bdystar \Om$ contains at least
two points.
\end{prop}  

\begin{proof}
  Let $f \in C(\bdystar \Om)$  and $\phi \in \UU_f(\Om)$.
Then $\phi$ is constant, by Lemma~\ref{lem-parabolic-superh-X},
and thus $\phi \ge \sup_{\bdystar \Om} f$.
Taking the infimum over all such $\phi$ gives
$\uP f \equiv \sup_{\bdystar \Om} f$.
Similarly $\lP f \equiv \inf_{\bdystar \Om} f$.
So $f$ is resolutive only if it is constant.
Choosing a nonconstant $f \in C(\bdystar \Om)$ shows that
no boundary point is regular if $\bdystar \Om$ contains at least
two points.
\end{proof}

\begin{prop} \label{prop-reg-par/hyp}
If $X$ is \p-parabolic and $K \ne \emptyset$ is compact,
then $\infty$ is irregular for $\Om:=X \setm K$.  
\end{prop}

Observe that this is false if 
$K=\emptyset$, see Example~\ref{ex-bdystar-one-pt}. 
To prove Proposition~\ref{prop-reg-par/hyp}, we will use the following definition from
Hansevi~\cite{Hansevi2}, which generalizes Definition~\ref{def-p-par}
to general open sets.
For $\Om=X$, these two definitions are clearly equivalent.

\begin{deff}   \label{def-p-par-set}
An unbounded open set $\Om\subset X$ is \emph{\p-parabolic} if for each 
compact set $K\subset\Om$ there exist functions $u_j\in\Np(\Om)$
such that $u_j\ge1$ on $K$ for all $j=1,2,\dots$ and
\begin{equation*} 
\int_\Om g_{u_j}^p\,d\mu \to 0 \quad \text{as } j\to\infty.
\end{equation*}
\end{deff}

\begin{proof}[Proof of Proposition~\ref{prop-reg-par/hyp}]
Clearly, $\Om$ is also \p-parabolic.
Let $f=\chione_K$.
Then $f \in C(\bdy^* \Om)$.
If $\Cp(K)>0$, then it follows from Theorem~7.8 in Hansevi~\cite{Hansevi2}
that 
$P_\Om f \equiv 1$,
and thus $\infty$ is irregular.
On the other hand, if $\Cp(K)=0$, then the irregularity follows from
Proposition~\ref{prop-Cp=0-para}. 
\end{proof}

We are now ready to prove Theorem~\ref{thm-char-hyp-intro}.

\begin{proof}[Proof of Theorem~\ref{thm-char-hyp-intro}]
\ref{i-hyp}\eqv\ref{i-cap(X,X)}
This follows from Proposition~\ref{prop-hyp-char-1}.

  \ref{i-hyp}\imp\ref{i-reg}
This follows from Lemma~\ref{lem-global-u-lim=0}.

\ref{i-reg}\imp\ref{i-reg-x0}
This is trivial.

$\neg$\ref{i-hyp}\imp$\neg$\ref{i-reg-x0}
This follows from Proposition~\ref{prop-reg-par/hyp} 
with $K=\{x_0\}$.

\ref{i-hyp}\eqv\ref{i-integral}
This was shown in 
Bj\"orn--Bj\"orn--Lehrb\"ack~\cite[Theorem~5.5]{BBLehIntGreen}.
\end{proof}

The following example shows that the global assumptions
  in Theorem~\ref{thm-char-hyp-intro} and
Lemma~\ref{lem-global-u-lim=0} cannot be replaced by local ones.

\begin{example} \label{ex-R-half-weighted}
Let $X=\R$ and $d\mu=w\,dx$ be as in Example~\ref{ex-R-half-weighted-cap},
where it was shown that $X$ is \p-hyperbolic.
Let $\Om=\R \setm \{0\}$ and $f=\chione_{\{0\}}$ with $f(\infty)=f(\pm\infty)=0$.
(Keeping with the earlier notation, $\infty=\pm\infty$ is the point added in the one-point 
compactification of $\R$.)
The function $u_0$ from~\eqref{eq-def-ur} is \p-harmonic in $(0,+\infty)$
and satisfies 
\[
\lim_{x \to 0\limplus} u_0(x)=1
\quad \text{and}  \quad \lim_{x \to +\infty} u_0(x)=0.
\]
Hence $Pf \equiv u_0$ in $(0,+\infty)$.

On the other hand, $(-\infty,0)$ is a \p-parabolic set and thus,
by Theorem~7.8 in Hansevi~\cite{Hansevi2},
$Pf \equiv 1$ in $(-\infty,0)$,
showing that $\infty$ is irregular.
Thus \ref{i-hyp}$\not\imp$\ref{i-reg} in 
Theorem~\ref{thm-char-hyp-intro}.
Since $\Cp(\{0\})>0$, it follows that 
$u=Pf$ (with $u(0)=1$ and $u(\infty)=0$)
is a singular
 function in $X$ with singularity at $x_0=0$.
Moreover, \eqref{ex-lim-u=0} in Lemma~\ref{lem-global-u-lim=0}
fails for the singular function  $u$.
\end{example}

Theorem~6.1 in
Bj\"orn--Bj\"orn--Shan\-mu\-ga\-lin\-gam~\cite{BBS2}
(i.e.\ \cite[Theorem~10.22] {BBbook}) shows
that if $\Om$ is bounded, $\Cp(X \setm\Om)>0$, $f \in C(\bdy \Om)$ and
$u$ is a bounded \p-harmonic function in $\Om$ such that
\eqref{eq-unique} below holds, then $u= Pf$.
This result was extended to unbounded \p-parabolic sets $\Om$
with $\Cp(X \setm \Om)>0$ (and $f\in C(\bdystar\Om)$)
by Hansevi~\cite[Corollary~7.9]{Hansevi2}.
We can now give another type of generalization to unbounded
sets.

\begin{prop}		\label{prop-unique}
  Let $\Om$ be an unbounded open set. 
Assume that  $f\in C(\bdystar\Om)$ is such that 
\[
\lim_{\Om \ni x \to \infty} \uP f(x) = f(\infty).
\]
Let $u$ be a bounded \p-harmonic function in $\Om$ such that 
\[
\lim_{\Om \ni x \to \infty} u(x)= f(\infty)
\]
and 
\begin{equation}   \label{eq-unique}
\lim_{\Om\ni y\to x}u(y)=f(x) \quad\text{for q.e.\ }x\in\bdy\Om.
\end{equation}
Then $u=\uP f$ in $\Om$.

In particular, if $\infty$ is regular
for $\Om$
then every $f \in C(\bdystar \Om)$ is resolutive.
\end{prop}

Proposition~\ref{prop-unique}
can be applicable even
when there are nonresolutive $f \in C(\bdystar \Om)$:
Consider $\Om=\R^n \setm \{0\}$ (unweighted), with $1 < p \le n$,
 and let $f(0)=0$ and $f(\infty)=1$.
Then $\lP f \equiv 0$ and  $\uP f \equiv 1$, by Proposition~\ref{prop-Cp=0-para}.

\begin{proof}[Proof of Proposition~\ref{prop-unique}]
By Tietze's extension theorem, we can extend $f$ so that $f \in C(X^*)$.
  We can assume that $f(\infty)=0$
and $|u|,|f|\le 1$.
Then also $|\uP f|\le 1$.
Fix $x_0 \in \Om$.
Let $\eps>0$ be arbitrary and find $r\ge1/\eps$ such that $|u|,|f|,|\uP f|\le \eps$
outside $B_r$.
By the Kellogg property (Theorem~\ref{thm-Kellogg}),
\eqref{eq-unique} holds also with $u$ replaced by $\uP f$.
Let
\[
   E = \{x\in\bdy\Om : \text{\eqref{eq-unique} fails for $u$ or $\uP f$}\}.
\]
As $u$ and $\uP f$ are bounded and \p-harmonic in $\Om_r:=\Om\cap B_r$, we have 
by the definition of Perron solutions that
\begin{equation}   \label{eq-u,Pf}
\uP f \le \lP_{\Om_r} (f+2\eps +2\cdot \chione_E)  
\quad \text{and} \quad 
u \ge \uP_{\Om_r} (f-2\eps -2\cdot \chione_E) 
\qquad \text{in } \Om_r. 
\end{equation}
Since $\Cp(E)=0$, applying 
Theorem~6.1 in Bj\"orn--Bj\"orn--Shan\-mu\-ga\-lin\-gam~\cite{BBS2}
(or \cite[Theorem~10.22]{BBbook}) to the bounded set $\Om_r$,
we get from \eqref{eq-u,Pf} that 
\[
\uP f \le 
P_{\Om_r} f+2\eps \le u+4\eps
\quad \text{in } \Om_r. 
\]
Letting $\eps\to0$ (and thus $r\to\infty$) gives $\uP f \le u$ in $\Om$.
The reverse inequality is proved similarly by switching the roles of
$\uP f$ and $u$ in \eqref{eq-u,Pf}.

If $\infty$ is regular, then by definition,
\[
\lim_{\Om \ni y \to \infty} \lP f(y)
=     \lim_{\Om \ni y \to \infty} \uP f(y) = f(\infty).
\]
The Kellogg property (Theorem~\ref{thm-Kellogg}) implies that 
$u=\lP f$ satisfies \eqref{eq-unique}.
Hence, $\lP f = \uP f$, i.e.\ $f$ is resolutive.
\end{proof}

\begin{proof}[Proof of Theorem~\ref{thm-main-resolutive}]
If $X$ is \p-hyperbolic and $\Om$ is unbounded, then $\infty$ is regular, by
Lemma~\ref{lem-global-u-lim=0}.
Thus $f$ is resolutive by Proposition~\ref{prop-unique}.

If $X$ is \p-parabolic, $\Om$ is unbounded and $\Cp(X \setm \Om)>0$, then $\Om$ is
a  \p-parabolic set, and the resolutivity was shown by
  Hansevi~\cite[Theorem~7.8]{Hansevi2}. 
If $\Om$ is bounded and $\Cp(X \setm \Om)>0$, then
the resolutivity was shown by
Bj\"orn--Bj\"orn--Shan\-mu\-ga\-lin\-gam~\cite[Theorem~6.1]{BBS2}.
\end{proof}  

As far as we know, the
following perturbation result is new even in unweighted $\R^n$ for 
$2\ne p<n$
(also when $\supp h$ is just a singleton).
In unweighted $\R^n$, the resolutivity of $f \in C(\bdystar \Om)$ was
obtained by 
Kilpel\"ainen~\cite[Theorem~1.10 and Remark~2.15]{Kilp89}.
For bounded~$\Om$ 
and bounded $h$ with $\Cp(\supp h)=0$, the equality $P(f+h)=Pf$
can be deduced from Kilpel\"ainen--Mal\'y~\cite{KilpMaly89},
see Remark~\ref{rmk-KilpMaly89} below.
However, the following situation cannot be treated by~\cite{KilpMaly89}:
bounded $\Om\subset\R^n$, 
$1< p \le n$ and $h=\chione_E$ for a 
dense countable set $E\subset\bdy\Om$.

When $\Om$ is bounded or \p-parabolic, and $\Cp(X\setm \Om)>0$,
there are stronger perturbation results in~\cite[Theorem~6.1]{BBS2}
and~\cite[Theorem~7.8]{Hansevi2}, covering e.g.\ the above mentioned
$h=\chione_E$ as well as unbounded $h$ like $h=\infty \chione_E$.
(By Proposition~\ref{prop-Cp=0-para} the assumption $\Cp(X\setm \Om)>0$
cannot be dropped in those results.)
We are now able to treat also unbounded sets $\Om$ in \p-hyperbolic
spaces (albeit only for bounded $h$).

\begin{prop} \label{prop-perturbation}
Assume that $X$ is \p-hyperbolic,  $\Om$ is
unbounded and  $\infty$ is regular for $\Om$.
Let $f \in C(\bdystar \Om)$ and let $h:\bdystar \Om \to \R$ be a bounded function
such that for all $j=1,2,\dots$\,,
\[
E_j:=\{x \in \bdystar \Om : |h(x)| >1/j\}
\]
is bounded and $\Cp(\clE_j)=0$.  
Then $P(f+h)=Pf$ in $\Om$.
\end{prop}

\begin{proof}
By Proposition~\ref{prop-unique}, $f$ is resolutive.

Let $k_j=f+(h-1/j)_\limplus$.
Then $k_j$ is continuous at every point in $A_j:=\bdystar \Om \setm \clE_j$.
It thus follows from Proposition~\ref{prop-reg5.1} that 
\[
    \lim_{\Om \ni y \to x} \uP k_j(y) = k_j(x) = f(x)
\quad \text{for every regular } x \in A_j,
\]
in particular for $x=\infty$ and for q.e.\ $x \in \bdy \Om$ 
by the Kellogg property (Theorem~\ref{thm-Kellogg}).
Thus, by a direct comparison of Perron solutions and Proposition~\ref{prop-unique},
\[
\uP (f+h)-1/j = \uP (f+h-1/j) \le \uP k_j = Pf \quad \text{in } \Om.
\]
Letting $j\to\infty$ shows that $\lP(f+h) \le \uP(f+h) \le Pf$.
Similarly, $\lP(f+h)\ge Pf$ in $\Om$.
\end{proof}

\begin{remark} \label{rmk-KilpMaly89}
For bounded $\Om$, the earliest perturbation result
for Perron solutions in the nonlinear case that we are aware of is
by Kilpel\"ainen--Mal\'y~\cite{KilpMaly89} on unweighted $\R^n$, $1<p \le n$.
It is somewhat hidden in that paper so we seize the opportunity to point
out how it follows from their results.
These arguments apply also to so-called $\mathcal{A}$-harmonic functions.

\medskip

\emph{Statement.
Let $\Om \subset \R^n$, $n \ge 2$, be a nonempty bounded open set and $1<p\le n$.
Let $f \in C(\bdy \Om)$ and let $h:\bdy \Om \to \R$ be a bounded function
with 
$\Cp(\supp h)=0$.  
Then $P(f+h)=Pf$ in $\Om$.
}

\medskip

\emph{Proof.}
Let $I_\Om$ be the set of irregular boundary points.
Then $\Cp(I_\Om)=0$, by the Kellogg property,
which had been obtained in this situation by 
Kilpel\"ainen~\cite[Corollary~5.6]{Kilp89}.
Hence $\Cp(E)=0$ for $E:=I_\Om \cup \supp h$.

Lemma~2.17 in~\cite{Kilp89} 
shows that
\[
    \lim_{\Om \ni y \to x} \uP(f+h)(y) = f(x)
    \quad \text{for every  } x \in \bdy \Om \setm E.
\]
Hence  $\uP(f+h)$ is a ``generalized solution'' 
of the Dirichlet problem with boundary data $f$ in the sense of 
\cite[p.~40]{KilpMaly89} for
\[ 
   \mathcal{V}=\{u:  u \text{ is superharmonic and bounded from below in } \Om \}
\quad \text{and} \quad
\mathcal{E}=\{E\}.
\] 
Moreover, it was shown in~\cite[Theorem~1.5]{Kilp89} that regular boundary 
points are ``$ \mathcal{V}$-exposed" in the sense of \cite[Section~4]{KilpMaly89}.
It then follows from the uniqueness theorem
\cite[Proposition~5.2]{KilpMaly89} that  $\uP(f+h)=P f$.
Similarly  $\lP(f+h)=P f$.
\qedsymbol

\medskip

Using uniform convergence, as in the proof of 
Proposition~\ref{prop-perturbation}, one can 
in the above statement allow $h:\bdy \Om \to \R$
to be any bounded function such that
\[
\Cp(\overline{\{x \in  \bdy \Om : |h(x)|> \eps\}})=0
\quad \text{for every } \eps >0.
\]
\end{remark}

The following result complements Proposition~\ref{prop-Cp=0-para}.

\begin{prop} \label{prop-resolutive-cp=0}
Assume that $X$ is unbounded,  $\Cp(X\setm \Om)=0$ and 
$\bdystar \Om$ contains at least two points.
Then all finite boundary points of $\Om$ are irregular
and the following statements are equivalent\/\textup{:}
\begin{enumerate}
\item \label{d-res}
  All $f \in C(\bdystar \Om)$ are resolutive.
\item \label{d-unbdd}
$\infty$ is regular for $\Om$.
\item \label{d-f-infty}
$Pf \equiv f(\infty)$ for all $f \in C(\bdystar \Om)$.
\setcounter{saveenumi}{\value{enumi}}
\end{enumerate}  
Moreover, if any of \ref{d-res}--\ref{d-f-infty} holds, then $X$ is \p-hyperbolic.
\end{prop}

\begin{proof}
Note that $\Om$ is unbounded since $\Cp(X\setm \Om)=0$.
First, we show that 
\begin{equation} \label{eq-f-infty}
\lP f \le f(\infty) \le \uP f \quad \text{in } \Om
\qquad \text{if } 
f \in C(\bdystar \Om).
\end{equation}
Let $u \in \UU_f(\Om)$.
Since $\Cp(X\setm\Om)=0$ and $u$ is bounded from below,
Theorem~6.3 in Bj\"orn~\cite{ABremove} (or \cite[Theorem~12.3]{BBbook}) implies that
$u$ has an extension as a superharmonic function on $X$ (also called $u$).
By the lsc-regularity of $u$, we get that
\[
\liminf_{X \ni y \to \infty} u(y)
= \liminf_{\Om \ni y \to \infty} u(y)
\ge f(\infty),
\]
and it follows from the strong minimum principle for superharmonic functions
(\cite[Theorem~9.13]{BBbook}) that $u \ge f(\infty)$ in $X$. 
Taking the infimum over all $u\in \UU_f(\Om)$ shows the 
second inequality in \eqref{eq-f-infty}.
The first one is shown similarly.

Next, 
let
\[
        	d_{\infty}(x) 
	= \begin{cases}
             e^{-d(x,x_0)},  & \text{if }x\neq\infty, \\
		0, & \text{if }x=\infty.
	\end{cases}
\]
As $0 \in \LL_{d_\infty}(\Om)$ it follows from \eqref{eq-f-infty} that
$\lP d_\infty \equiv  0$,
and hence 
all finite boundary points
are irregular
(regardless of if \ref{d-res}--\ref{d-f-infty} hold or not).

Now we are ready to show that \ref{d-res}--\ref{d-f-infty} are equivalent.

\ref{d-res}\imp\ref{d-f-infty}
Let $f \in C(\bdystar \Om)$.
As $f$ is resolutive (by assumption), it follows from \eqref{eq-f-infty}
that $Pf \equiv f(\infty)$.

\ref{d-f-infty}\imp\ref{d-unbdd}
This is trivial.

\ref{d-unbdd}\imp\ref{d-res}
This follows from Proposition~\ref{prop-unique}.

Finally, if $X$ is \p-parabolic, then \ref{d-res}--\ref{d-f-infty} fail by 
Proposition~\ref{prop-Cp=0-para}.
\end{proof}

\end{document}